\documentclass{article}

\usepackage[utf8]{inputenc}
\usepackage[normalem]{ulem}
\usepackage{amsthm}
\usepackage{amsmath}
\usepackage{calc}
\usepackage[english]{babel}
\usepackage{hyperref}
\usepackage[a4paper]{geometry}
\usepackage{graphicx,cite}
\usepackage{float}
\usepackage{amssymb}
\usepackage{bbm}
\usepackage{xcolor}
\usepackage{titlesec}
\usepackage{framed}
\usepackage{csquotes}
\usepackage{todonotes}
\usepackage{mathtools}
\usepackage{enumitem}
\usepackage{mathrsfs}
\usepackage{MnSymbol}

\theoremstyle{plain}
\newtheorem{proposition}{Proposition}[section]
\newtheorem{lemma}[proposition]{Lemma}
\newtheorem{theorem}[proposition]{Theorem}

\newtheorem{example}[proposition]{Example}

\theoremstyle{definition}
\newtheorem{definition}[proposition]{Definition}

\theoremstyle{remark}
\newtheorem{remark}[proposition]{Remark}

\numberwithin{equation}{section}

\title{Non-negative Martingale Solutions to the Stochastic Porous Medium Equation with Sticky Behavior}
\author{Ben Hambly\thanks{Mathematical Institute, University of Oxford.}, Dörte Kreher\thanks{Humboldt University Berlin; Department of Mathematics, Unter den Linden 6, 10099 Berlin.}\hspace{0.3em}\thanks{Dörte Kreher acknowledges support from DFG CRC/TRR 388 Rough Analysis, Stochastic Dynamics and Related Fields and from DFG IRTG 2544 Stochastic Analysis in Interaction.}, Konstantins Starovoitovs\thanks{Humboldt University Berlin; Department of Mathematics, Unter den Linden 6, 10099 Berlin.}\hspace{0.3em}\thanks{Konstantins Starovoitovs acknowledges support from DFG IRTG 2544 Stochastic Analysis in Interaction.}}

\date{\today}

\newcommand{\E}{\mathbb E}
\newcommand{\R}{\mathbb R}
\newcommand{\Z}{\mathbb Z}
\newcommand{\N}{\mathbb N}

\newcommand{\1}{\mathbbm 1}

\renewcommand{\d}{\mathrm d}
\renewcommand{\P}{\mathbb P}

\begin{document}

\maketitle

\begin{abstract}
We construct non-negative martingale solutions to the stochastic porous medium equation in one dimension with homogeneous Dirichlet boundary conditions which exhibit a type of sticky behavior at zero. The construction uses the stochastic Faedo--Galerkin method via spatial semidiscretization, so that the pre-limiting system is given by a finite-dimensional diffusion with Wentzell boundary condition. We derive uniform moment estimates for the discrete systems by an Aubin--Lions-type interpolation argument, which enables us to implement a general weak convergence approach for the construction of martingale solutions of an SPDE using a Skorokhod representation-type result for non-metrizable spaces. We rely on a stochastic argument based on the occupation time formula for continuous semimartingales for the identification of the diffusion coefficient in the presence of an indicator function.
\end{abstract}

\medskip
{\bf Key words.} Stochastic porous medium equation, non-negative solutions, sticky behavior, weak convergence approach, martingale solutions, scaling limits, stochastic partial differential equations.

\smallskip
{\bf 2010 Mathematics Subject Classification.} 60H15, 35R60, 76S05. 

\section{Introduction}

In this paper we construct a non-negative martingale solution to the stochastic porous medium equation in one dimension satisfying homogeneous Dirichlet boundary conditions and exhibiting a type of sticky behavior at zero. Specifically, for $\alpha\in \left[4, \infty \right)$ we construct a solution to the equation of the form
\begin{align}
    \label{eq: intro1}
    \d u(t) & =\partial_x^2\left(u^\alpha(t)\right) \d t+\1_{u(t)>0} B(u(t)) \d W(t) + \1_{u(t)=0} R(u(t)) \d t, \\
    \label{eq: intro2}
    u(t,0) &= u(t,1) = 0,\qquad t\in[0,T]\\
    \label{eq: intro3}
    u(0) &= u_0.
\end{align}
We consider Nemytskii-type diffusion coefficients $B$ satisfying certain growth and regularity conditions, in particular admitting affine linear functions of the solution $u$. As for the construction of finite-dimensional sticky-reflected diffusions, the diffusion coefficient $\1_{u(t)>0} B(u(t))$ in \eqref{eq: intro1} is ``switched off'' at zero, due to the presence of the pre-factor $\1_{u(t)>0}$. For a given Nemytskii-type sojourn coefficient $R$ satisfying certain growth and regularity conditions, the pushing term $\1_{u(t)=0} R(u(t)) \d t$ -- akin to the reflection term in the classical reflected SPDE \cite{NualartWhiteNoiseDriven1992, DonatiMartinWhiteNoiseDriven1993, xu2009white} -- ensures that the solution remains non-negative. It exhibits weaker repulsion at zero compared to the classical reflected SPDE, which we show can be thought of a type of sticky behavior at zero.

The stochastic porous medium equation has attracted considerable attention in recent years. The initial work 
(cf.~\cite{daprato2004invariant, daprato2004weakSolutionsPorousMedia, bogachev2004invariant, barbu2006weak, daprato2006strong, kim2006stochastic}) considered the case of additive noise. This was followed by analysis of the case of linear multiplicative noise (cf.~\cite{barbu2008existence, barbu2009stochastic, barbu2009existence, ciotir2010existence}), which led in particular to non-negativity of solutions, finite-time extinction results \cite{barbu2009finite, barbu2012finite}, and finite speed of propagation results \cite{barbu2012localization, gess2013finite, fischer2015finite}. In the articles \cite{ren2007stochastic} and \cite{rockner2008nonmonotone}, the authors considered a more general dependence of the diffusion coefficients on the solution, where in the particular case of Nemytskii-type diffusion coefficients one could accommodate linear affine functions of the solution $u$. Further developments include the study of quasi-linear equations driven by the subdifferential of a quasi-convex function \cite{gess2012strong}, the consideration of fast diffusion equations for systems driven by multi-valued maximal monotone operators \cite{gess2014multi, barbu2015stochastic, gess2015singular}, the study of kinetic solutions \cite{hofmanova2013degenerate, debussche2016degenerate} and pathwise kinetic solutions \cite{fehrman2019well, fehrman2021path}, an $L^1$-approach with general assumptions on non-linearities \cite{gess2018well}, the study of entropy solutions \cite{dareiotis2019entropy}, and the construction of martingale solutions to a stochastic chemotaxis system with a porous medium diffusion by a stochastic version of a Schauder--Tychonoff-type fixed-point theorem \cite{hausenblas2022martingale}. On the numerical analysis aspects, the convergence of non-negativity preserving fully-discrete schemes for stochastic porous medium equation has been studied in \cite{grillmeier2019nonnegativity}. 

In our context, of particular interest is the work on the stochastic porous medium equation with reflection \cite{rockner2013porousreflection}, where the authors motivate the dynamics as a potential model of self-organized criticality, interpreting the solution as an energy functional, which is guaranteed to be non-negative in the case of additive noise due to the reflection term.
In the work \cite{dirr2021stochastic} the authors construct a solution to the stochastic porous medium equation with conservative multiplicative noise by means of the stochastic Faedo--Galerkin method. While our approach exhibits certain similarities to their scheme, in our framework conditions on the coefficients allow for affine dependence of the diffusion coefficients on the solution. As opposed to the case of linear multiplicative noise, we aim in particular to accommodate diffusion coefficients which do not vanish in the vicinity of the reflecting boundary.

Diffusions with sticky boundaries were first investigated by Feller \cite{feller1952parabolic} and have since been extensively studied \cite{graham1988martingale, yamada1994reflecting, bass2014stochastic, engelbert2014stochastic, racz2015multidimensional, bou2020sticky, barraquand2020large}. In the context of stochastic interface models, in the case $\alpha=1$ the dynamics \eqref{eq: intro1}--\eqref{eq: intro3} can be understood as the scaling limit of the sticky counterpart of the Ginzburg--Landau $\nabla \phi$ interface model with repulsion studied in \cite{funaki2001fluctuations, funaki2003hydrodynamic}. An attempt to study a similar static system with stickiness was made in \cite[Section 15]{funaki2005stochastic}, where the authors introduce an interface model with pinning dynamics, leading to competing effects of pinning and repulsion at the boundary. A broader construction of multidimensional sticky-reflected distorted Brownian motion was explored in \cite{fattler2016construction, GrothausStrongFellerProperty2018}, where the processes were constructed using Dirichlet form methods, allowing for a more general drift. However, these approaches are limited to symmetric processes and require the invariant measure to be known in advance. In \cite{funaki2001fluctuations} the authors study convergence of finite-dimensional interface models to a Nualart--Pardoux-type reflected SPDE, as is the case in the work \cite{etheridge2015scaling} for weakly asymmetric interfaces. By this analogy, the system \eqref{eq: intro1}--\eqref{eq: intro3} can be regarded as a counterpart of the reflected SPDE studied in \cite{NualartWhiteNoiseDriven1992, DonatiMartinWhiteNoiseDriven1993, xu2009white} with stickiness at the boundary and porous medium diffusion operator. Furthermore, since our construction is given by spatial discretization, where the pre-limiting system corresponds to a sticky-reflected diffusion with Wentzell boundary condition \cite{graham1988martingale}, one can view our construction as an infinite-dimensional diffusion limit of that system as the spatial scaling parameter goes to zero. In a recent work \cite{KonarovskyiSTICKYREFLECTEDSTOCHASTICHEAT2021} a sticky-reflected stochastic heat equation with additive noise was studied, where the construction was also carried out using spatial semidiscretization. The moment estimates in that work were derived from discrete heat kernel estimates of \cite{funaki1983random}, which are not applicable in our case, necessitating a different approach. Furthermore, in our framework, the diffusion and sojourn coefficients are allowed to depend on the solution $u$ itself.

Our analysis is facilitated by the fact that the pushing term is less singular compared to the case of the classical reflected SPDE \cite{NualartWhiteNoiseDriven1992, DonatiMartinWhiteNoiseDriven1993, xu2009white}, where the solution is constructed by penalization with a term of the type $\varepsilon^{-1} u_- \d t$, which may blow up as $\varepsilon \downarrow 0$. By contrast, in the setup we consider, the pushing term is absolutely continuous in space and time with respect to Lebesgue measure as the sojourn coefficient $R$ is assumed to satisfy certain growth conditions. This aspect makes our setup slightly more straightforward in comparison to the classical reflected SPDE.

For the construction of a solution to the system \eqref{eq: intro1}--\eqref{eq: intro3} we employ the stochastic Faedo--Galerkin method through spatial semidiscretization. Similar techniques were employed, for example, for the stochastic porous medium equation driven by conservative noise in \cite{dirr2021stochastic} and the stochastic thin-film equation in \cite{fischer2018existence}. In order to derive moment estimates in Section \ref{sec: moment estimates}, we consider a finite-dimensional counterpart $\mathbf u_n$ of the dynamics. As is common in the analysis of systems with porous medium diffusion operator, we apply It\^o's lemma to norms of $\mathbf u^{\alpha+1}_n$ in Section \ref{subsec: estimates 1} and Section \ref{subsec: estimates 3}, where we invoke dissipativity of the system and Sobolev estimates for discrete pointwise products. Additionally, writing $\mathbf{H}_n^{\theta}$ for the discrete fractional Sobolev space of order $\theta$, we also derive moment estimates for the $L^{\infty} ([0, T], \mathbf{H}_n^{(\alpha-2) / \alpha} )$ norm of $\mathbf u^{\alpha}_n$ in Section \ref{subsec: estimates 2}. For that estimate, we are ultimately required to control the norm of $\mathbf u^{\alpha-2}_n$, where we invoke estimates for discrete Sobolev norms of powers of $\mathbf u_n$ from Lemma \ref{lemma: discrete sobolev power estimates}, akin to those presented in Appendix \ref{sec: appendix sobolev power estimates}. However, this approach imposes a restriction on the Sobolev parameter corresponding to the spatial regularity, ultimately leading to the constraint $\alpha\in[4,\infty)$. Importantly, in our framework we accommodate diffusion coefficients which do not vanish near zero, i.e. $\1_{u(t)>0}B(u(t)) \neq B(u(t))$, which is common for problems with reflection. However, the presence of the indicator $\1_{u(t)>0}$ in the diffusion coefficient $\1_{u(t)>0}B(u(t))$ poses a significant challenge for the derivation of uniform estimates for the moments of the $W^{s,p}([0,T], \mathbf H^\theta_n)$ norm for $\theta>0$. Consequently, application of It\^o's lemma to the norm of $\mathbf u^{\alpha}_n$ has an additional advantage: it provides a regularizing effect by yielding additional factors $\mathbf u^{\alpha-1}_n$ and $\mathbf u^{\alpha-2}_n$, which eliminate the indicator next to the diffusion coefficient. Thereafter, we proceed with an Aubin--Lions type interpolation argument along the lines of \cite{fischer2018existence, dirr2021stochastic, dareiotis2021martingaleSolutions}, i.e.~we first show moment estimates in spaces with varying space and time regularities and then conclude with the intended estimate by interpolation.

For the identification we employ a general weak convergence method for the construction of martingale solutions of an SPDE \cite{ondrejat2010stochastic, fischer2018existence, brzezniak2019martingale, gess2020stochastic, dareiotis2021martingaleSolutions, dirr2021stochastic}, relying on a Skorokhod representation-type result for non-metrizable spaces by Jakubowski \cite{jakubowski1998almost}. For the identification of the diffusion coefficient and of the pushing term, both containing an indicator, we use a stochastic argument, applying the occupation time formula for semimartingales to the pre-limiting discrete system. Overall, we obtain the existence of a martingale solution to \eqref{eq: intro1}--\eqref{eq: intro3}, while the question of strong uniqueness remains open and truly challenging due to our construction of the solution as a limit in law and the presence of the indicator in the diffusion coefficient $\1_{u(t)>0}B(u(t))$.

One of our motivations for studying systems of the type presented in this work comes from SPDE approximations of limit order book models (cf.~\cite{bayer2017functional, horst2018second, horst2019scaling, horst2019diffusion, kreher2023jump, horst2023second}), where under heavy traffic conditions one considers an order book model with a continuum of price levels as an approximation of the discrete system. In such a context, the infinite-dimensional stochastic system should remain non-negative, which led to a number of approximation results for order book models by SPDEs exhibiting reflection at zero \cite{hambly2019reflected, hambly2020limit, hambly2020stefan}. Especially for illiquid assets the macroscopic description of the limit order book through sticky-reflected diffusions might be interesting as the order books of illiquid assets tend to exhibit large gaps between standing orders. Since sticky-reflected diffusions allow individual queues to spend positive time at zero, dynamics of the type studied in this work might be employed as a sticky-reflected counterpart of the approximations mentioned above, potentially enabling the modelling of illiquidity.

This paper is organized as follows: in Section \ref{sec: setup and main result} we introduce the model setup and the approximate discrete systems, and we present the main results. In Section \ref{sec: moment estimates} we derive the moment estimates for the norms of the discrete solutions. In Section \ref{sec: convergence} we implement compactness arguments and identify the limit. Some references for estimates of Sobolev norms of pointwise powers are given in Appendix \ref{sec: appendix sobolev power estimates}, and the proofs of some auxiliary results appear in Appendix \ref{sec: appendix proof of auxiliary results}.

\section{Setup and main result}

\label{sec: setup and main result}

\subsection{Preliminaries}

\subsubsection{Notation}

For $a,b \geq 0$ we write $a \lesssim b$ whenever there exists a positive constant $C < \infty$ such that $a \leq Cb$. We denote by $\langle \cdot, \cdot \rangle$ scalar products and by $\llangle \cdot , \cdot \rrangle$ covariation processes. We also use the notation $\odot$ for pointwise multiplication of vectors in $\R^n$. For $\mathbf u_n \in \R^{n}$ we write $\1_{\mathbf u_n>0}\in\R^{n}$ and $\1_{\mathbf u_n=0}\in\R^{n}$ such that for $i\in\{1,\dots,n\}$
\[
    \1_{\mathbf u_n>0}(i) := \1_{(0,\infty)}(\mathbf u_n(i)) ,\qquad \1_{\mathbf u_n=0}(i) := \1_{\{0\}}(\mathbf u_n(i)).
\]
Analogously, for some domain $X$ and a function $u:X\rightarrow\R$ we consider $\1_{u>0}:X\rightarrow\{0,1\}$ and $\1_{u=0}:X\rightarrow\{0,1\}$ such that for $x\in X$
\[
    \1_{u>0}(x) := \1_{(0,\infty)}(u(x)), \qquad \1_{u=0}(x) := \1_{\{0\}}(u(x)).
\]

\subsubsection{Functional setting}

Without loss of generality we confine our analysis to the spatial domain $[0,1]$. We fix the time horizon $T>0$ and define $\mathcal Q_T:=[0,T]\times[0,1]$. Let $X$ be an open subset of $\R^d$ and let $E$ be a Banach space. We denote by $B(X, E)$ the Banach space of all bounded maps from $X$ to $E$, equipped with the supremum norm $\|\cdot\|_\infty$. For $m\in\N$, $\gamma\in(0,1)$, $s\in(0,\infty)$ and $p\in[1,\infty]$, denoting multi-indices by $\alpha \in \N_0^d$, we define norms
\[
    \begin{aligned}
        \|u\|_{W^{m,p}(X, E)} &:=
        \begin{cases}
            \left(\sum_{|\alpha| \leq m} \|\partial^\alpha u\|_{L^p(X, E)}^p \right)^{\frac1p}, & p\in[1,\infty) , \\
            \max_{|\alpha|\leq m} \| \partial^\alpha u\|_{\infty},&p=\infty,
        \end{cases}\\
        \|u\|_{W^{s,p}(X, E)} &:=
        \begin{cases}
            \left(\|u\|_{W^{\lfloor s \rfloor, p}(X, E)}^p + \sum_{|\alpha| = \lfloor s \rfloor} \iint \frac{\|\partial^\alpha u(x) - \partial^\alpha u(y)\|_{E}^p}{|x-y|^{1+(s - \lfloor s \rfloor)p}} \d x \d y\right)^{\frac1p},&p\in[1,\infty),\\
            \|u\|_{W^{\lfloor s \rfloor,\infty}(X, E)} + \sum_{|\alpha| = \lfloor s \rfloor} \sup_{x\neq y} \frac{\|\partial^\alpha u(x) - \partial^\alpha u(y)\|_E}{|x-y|^{s - \lfloor s \rfloor}},& p = \infty,
        \end{cases}\\
        \|u\|_{C^\gamma(X, E)} &:= \|u\|_\infty + \sup_{x\neq y} \frac{\|u(x) - u(y)\|_E}{|x-y|^\gamma}.
    \end{aligned}
\]
We define Sobolev, Slobodeckii and Hölder spaces by
\[
    \begin{aligned}
        W^{m,p}(X, E) &:= \left\{ u \in L^p(X, E) : \|u\|_{W^{m,p}(X, E)} < \infty \right\},\\
        W^{s,p}(X, E) &:= \{u \in L^p(X, E): \|u\|_{W^{s,p}(X, E)} < \infty\},\\
        C^\gamma(X, E) &:= \{u \in B(X, E): \|u\|_{C^\gamma(X, E)} < \infty\},\\
        H^{s}(X, E) &:= W^{s,2}(X, E).
    \end{aligned}
\]
Moreover, we will also make use of Hölder spaces with different regularity in space and time. That is, for $\gamma_1, \gamma_2 \in [0,1)$ we define
\begin{align*}
    [u]_{1,\gamma_1} &:=  \sup_{x\in[0,1]} \sup_{t_1 \neq t_2 \in[0,T]} \frac{|u(t_1, x) - u(t_2, x)|}{|t_1 - t_2|^{\gamma_1}}, \\
    [u]_{2,\gamma_2} &:= \sup_{t\in[0,T]} \sup_{x_1 \neq x_2 \in [0,1]} \frac{|u(t,x_1) - u(t,x_2)|}{|x_1 - x_2|^{\gamma_2}}, \\[7pt]
    \|u\|_{C^{\gamma_1, \gamma_2}_0(\mathcal Q_T)} &:= \sup_{(t,x)\in \mathcal Q_T} |u(t,x)| + \1_{\gamma_1 > 0} [u]_{1,\gamma_1} + \1_{\gamma_2 > 0} [u]_{2,\gamma_2}
\end{align*}
and we define
\[
    C^{\gamma_1, \gamma_2}_0(\mathcal Q_T) := \left\{ u \in C(\mathcal Q_T): \|u\|_{C^{\gamma_1, \gamma_2}_0(\mathcal Q_T)} < \infty, u(t, 0) = u(t, 1) = 0 \text{ for all } t \in [0,T] \right\}.
\]
For Banach spaces $E,F$ we denote by $L_2(E, F)$ the space of Hilbert-Schmidt operators from $E$ to $F$, and for $\omega\in(0,1)$ we denote by $[E,F]_{\omega}$ the respective interpolation space \cite[Section VI.2.3/VII.2.7]{amann2019linear}. Whenever we omit specification of the domain and co-domain we mean functions mapping from $[0,1]$ into $\R$, e.g.~we write just $L^2$ for $L^2([0,1],\R)$, and analogously for other spaces. Furthermore, whenever we omit the specification of the scalar product, we mean the scalar product in $L^2$, i.e. $\langle f,g \rangle := \langle f,g \rangle_{L^2}$.

We now aim to specify Sobolev spaces $H^s_0$ with homogeneous Dirichlet boundary conditions for $s\in[-1,1]$. We define
\[
    \langle u, v \rangle_{H^1_0} := \langle \partial_x u, \partial_x v \rangle,\qquad H^1_0 := \overline{C^\infty_c}^{\left\|\cdot\right\|_{H^1_0}}.
\]
Note that we have according to the previous definitions
\[
    \langle u, v \rangle_{H^1} = \langle u, v \rangle + \langle \partial_x u, \partial_x v \rangle, \qquad H^1 = \overline{C^\infty}^{\left\| \cdot \right\|_{H^1}}.
\]
We define $g_k := \sqrt{2} \sin(k\pi x)$ and $\lambda_k := (\pi k)^2$. Let $-\partial_x^2$ be the negative Laplacian with homogeneous Dirichlet boundary conditions on the domain $[0,1]$. The operator $-\partial_x^2$ constitutes the Riesz map from $H^1_0$ into $(H^{1}_0)^*$, while the Riesz map from $H^1$ into $(H^1)^*$ is given by $1-\partial_x^2$. Both operators $-\partial_x^2$ and $1-\partial_x^2$ admit the same eigenbasis $\left\{g_k: k\in \N \right\}$, while their corresponding eigenvalues are given by $\lambda_k $ and $1+\lambda_k$, respectively. For $s\in[-1,1]$ by means of the scalar products
\[
    \begin{aligned}
    \left\langle u, v \right\rangle_{H^s_0} &:= \left\langle (-\partial_x^2)^{s/2} u, \left(-\partial_x^2\right)^{s/2} v \right\rangle = \sum_{k=1}^\infty \lambda_k^s \langle u, g_k \rangle \langle v, g_k \rangle,\\
    \left\langle u, v \right\rangle_{H^s} &:= \left\langle \left(1-\partial_x^2\right)^{s/2} u, \left(1-\partial_x^2\right)^{s/2} v \right\rangle = \sum_{k=1}^\infty \left(1+\lambda_k\right)^s \langle u, g_k \rangle \langle v, g_k \rangle,
    \end{aligned}
\]
we define the fractional Sobolev spaces
\[
    H^s_0 := \overline{C_c^\infty}^{\|\cdot\|_{H^s_0}}, \qquad H^s := \overline{C^\infty}^{\|\cdot\|_{H^s}}.
\]
Since $\lambda_k \leq 1+\lambda_k \leq \left( 1 + \frac{1}{\pi^2} \right) \lambda_k$, it follows that for any $s\in[-1, 1]$ there exist constants $c_1, c_2 > 0$ such that for any $u\in H^s_0$,

\begin{equation}
    \label{eq: equivalence of norms Hs and Hs0}
    c_1 \|u\|_{H^s} \leq \|u\|_{H^s_0} \leq c_2 \|u\|_{H^s}.
\end{equation}
Moreover, for $s\in[0,1]$ we have that
\[
    H^s_0 = \left[H^1_0, H^0\right]_{1-s},\qquad H^{-s}_0 = \left[\left(H^1_0\right)^*, H^0\right]_{1-s}.
\]

We have the following Poincaré-type result.

\begin{lemma}
    \label{lemma: poincare type result}
    For any $s_1, s_2 \in [-1, 1]$, $s_1 \leq s_2$,
    \[
        \|u\|_{H^{s_1}_0} \leq \|u\|_{H^{s_2}_0}.
    \]
\end{lemma}

\begin{proof}
    Note that we have $\lambda_k > 1$ for all $k\in\N$, so that $s \mapsto \lambda_k^s$ is an increasing function. Therefore, we have
    \[
        \|u\|_{H^{s_1}_0}^2 = \sum_{k=1}^\infty \lambda_k^{s_1} \langle u, g_k \rangle^2 \leq \sum_{k=1}^\infty \lambda_k^{s_2} \langle u, g_k \rangle^2 = \|u\|_{H^{s_2}_0}^2.
    \]
\end{proof}

\subsection{Discretization}
\label{subsec: discretization}

For $n\in\N$ we discretize the interval $[0,1]$ into $n+1$ equidistant intervals of length $h_n := \frac{1}{n+1}$. For the piecewise constant approximation, for $n\in\N$ and $i\in\{1,\dots,n\}$ we define
\[
    e_{i,n}(x) := h_n^{-1/2} \1_{((i-1)h_{n}, ih_{n}]}(x).
\]
Whenever necessary, we also use the convention
\[
    e_{0,n} = e_{n+1,n} = 0.
\]
For $n\in\N$ we define subspaces
\begin{equation}
    \label{eq: definition Vn}
    V_n := \operatorname{span} \{e_{i,n}: i \in \{1, \dots, n\}\}.
\end{equation}
Note that for $u_n, v_n \in V_n$ we have
\[
    \langle u_n, v_n \rangle = h_n \mathbf u_n^T \mathbf v_n.
\]
Furthermore, we consider maps $I_n: L^2 \rightarrow V_n$ and $P_n: V_n \rightarrow \R^n$ given by
\begin{equation}
    \label{eq: definition of Pn In}
    I_n(u) = \sum_{i=1}^{n} \langle u, e_{i,n} \rangle e_{i,n}, \qquad P_n(u_n)(i) = u_n(ih_n).
\end{equation}
Note that
\begin{equation}
    \label{eq: In(u) leq u}
    \begin{aligned}
        \|I_n(u)\|_{L^2}^2 &= \sum_{i=1}^n \int_{(i-1)h_n}^{ih_n} \left(h_n^{-1} \int_{(i-1)h_n}^{ih_n} u(y) \d y \right)^2 \d x \leq \sum_{i=1}^n \int_{(i-1)h_n}^{ih_n} u^2(y) \d y = \|u\|_{L^2}^2.
    \end{aligned}
\end{equation}
Furthermore, by the Lebesgue differentiation theorem, for a.e. $x\in[0,1]$ we have $|u(x) - I_n(u)(x)| \rightarrow 0$, which coupled with $|u - I_n(u)|^2 \leq u^2 + I_n(u)^2 \in L^1$ and dominated convergence yields
\begin{equation}
    \label{eq: In(u) - u -> 0}
    \|u - I_n(u)\|_{L^2} \rightarrow 0.
\end{equation}

To construct the piecewise linear approximation, for $n\in\N$ and $i\in\{1,\dots,n\}$ we take $\mathfrak{e}_{i,n}\in H^1_0$ defined by
\[
    \begin{aligned}
    \mathfrak{e}_{1,n}(x) &:= \begin{cases}
        h_n^{-3/2}x,&0\leq x \leq h_n\\
        -h_n^{-3/2}\left(x-h_n\right) + h_n^{-1/2} ,& h_n \leq x \leq 2 h_n \\
        0,&\text{otherwise}
    \end{cases},\\
    \mathfrak{e}_{i,n}(x) &:= \mathfrak{e}_{1,n}\left(x - (i-1) h_n \right).
    \end{aligned}
\]
Whenever necessary, we also use the convention
\[
    \mathfrak e_{0,n} = \mathfrak e_{n+1,n} = 0.
\]
We consider a finite-dimensional subspace of piecewise linear functions defined by
\begin{equation}
    \label{eq: definition frak Vn}
    \mathfrak V_n := \operatorname{span}\{\mathfrak{e}_{i,n}: i\in\{1,\dots,n\}\}.
\end{equation}
We also define
\[
    \mathfrak{H}^0_n := \left(\mathfrak V_n, \|\cdot\|_{L^2}\right), \qquad \mathfrak{H}^1_n := \left(\mathfrak V_n, \|\cdot\|_{H^1_0}\right), \qquad \mathfrak{H}^{-1}_n := \left(\mathfrak{H}^1_n\right)^*,
\]
and, for $\theta\in(0,1)$, the interpolation spaces
\[
    \mathfrak{H}^\theta_n := \left[ \mathfrak{H}^1_n, \mathfrak{H}^0_n \right]_{1-\theta}, \qquad \mathfrak{H}^{-\theta}_n := \left[ \left( \mathfrak{H}^{1}_n \right)^*, \mathfrak{H}^0_n \right]_{1-\theta}.
\]
Furthermore, we define maps $\mathfrak I_n: L^2 \rightarrow \mathfrak V_n$ and $\mathfrak P_n: \mathfrak V_n \rightarrow \R^n$ given by
\begin{equation}
    \label{eq: definition of frak Pn In}
    \mathfrak I_n(u) := \sum_{i=1}^n \langle u, \mathfrak e_{i,n} \rangle \mathfrak e_{i,n}, \qquad \mathfrak P_n(\mathfrak u_n)(i) := \mathfrak u_n(ih_n).
\end{equation}

We generally use the boldface notation $\mathbf u_n\in\R^n$ for vectors in $\R^n$, $u_n \in V_n$ for their piecewise constant representation in $V_n$ defined in \eqref{eq: definition Vn}, and $\mathfrak u_n\in \mathfrak V_n$ for their piecewise linear representation in $\mathfrak V_n$ defined in \eqref{eq: definition frak Vn}. Specifically, for $i\in\{0,\dots,n+1\}$ we have
\[
    u_n(ih_n) = \mathfrak u_n(ih_n) = \mathbf u_n(i).
\]

For $i\in\{1,\dots,n\}$ it is easy to see that
\begin{gather*}
    \langle \mathfrak{e}_{i,n}, \mathfrak{e}_{i,n} \rangle = \frac{2}{3}, \qquad\langle \mathfrak{e}_{i,n}, \mathfrak{e}_{i+1,n} \rangle = \langle \mathfrak{e}_{i,n}, \mathfrak{e}_{i-1,n} \rangle = \frac16, \\
    \langle \mathfrak{e}_{i,n}', \mathfrak{e}_{i,n}' \rangle = \frac{2}{h_n^2}, \qquad\langle \mathfrak{e}_{i,n}', \mathfrak{e}_{i+1,n}' \rangle = \langle \mathfrak{e}_{i,n}', \mathfrak{e}_{i-1,n}' \rangle = -\frac{1}{h_n^2},
\end{gather*}
while for $|i-i'| > 1$ we have $\langle \mathfrak{e}_{i,n}, \mathfrak{e}_{i',n} \rangle = \langle \mathfrak{e}'_{i,n}, \mathfrak{e}'_{i',n} \rangle = 0$. Thus, the Grammian matrices $\mathbf L_{0,n}, \mathbf L_{1,n} \in \R^{n \times n}$ are given by
\[
    \mathbf L_{0,n} = \begin{bmatrix}
    \frac{2}{3} & \frac16 & 0 & \cdots & \cdots & 0 \\
    \frac16 & \frac{2}{3} & \frac16 & 0 & \cdots & 0 \\
    \vdots & & \ddots & & & \vdots \\
    0 & \cdots & \cdots & 0 & \frac16 & \frac{2}{3} \\
    \end{bmatrix}, \qquad
    \mathbf L_{1,n} = -\frac{1}{h_n^2}\begin{bmatrix}
    -2 & 1 & 0 & \cdots & \cdots & 0 \\
    1 & -2 & 1 & 0 & \cdots & 0 \\
    \vdots & & \ddots & & & \vdots \\
    0 & 0 & \cdots & 0 & 1 & -2 \\
    \end{bmatrix},
\]
which are both tridiagonal and positive definite. Note that if $\mathbf u_n \in \R^n$ is an eigenvector of $\mathbf L_{0,n}$ with eigenvalue $\lambda \in \R$, then it is also an eigenvector of $\mathbf L_{1,n}$ and vice versa, since (denoting by $\mathbf I_n$ the identity matrix)
\[
    \mathbf L_{1,n} \mathbf u_n = \frac{6}{h_n^2} (\mathbf I_n - \mathbf L_{0,n}) \mathbf u_n = \frac{6(1-\lambda)}{h_n^2} \mathbf u_n.
\]
Therefore, the matrices $\mathbf L_{0,n}$ and $\mathbf L_{1,n}$ are simultaneously diagonalizable. 

According to \cite[Remark 3.1, Section 4.2]{arioli2009discrete}, whenever the matrices $\mathbf L_{0,n}$ and $\mathbf L_{1,n}$ are simultaneously diagonalizable, for $\theta \in [-1,1]$ the matrix representation of the respective interpolation space $\mathfrak H^\theta_n$ is given by
\[
    \mathbf L_{0,n}^{(1-\theta)/2} \mathbf L_{1,n}^{\theta} \mathbf L_{0,n}^{(1-\theta)/2}
\]
Note that spectra of the matrices $\mathbf L_{0,n}$ irrespective of $n\in\N$ satisfy
\[
    \sigma\left(\mathbf L_{0,n}\right) \subset \left(\frac13, 1\right),
\]
so that for $\theta\in(-1, 1)$ we also have
\[
    \sigma\left(\mathbf L_{0,n}^{(1-\theta)/2}\right) \subset \left(\frac13, 1 \right).
\]
It follows that the matrices $\mathbf L_{0,n}^{(1-\theta)/2} \mathbf L_{1,n}^{\theta} \mathbf L_{0,n}^{(1-\theta)/2}$ and $\mathbf L_{1,n}^\theta$ induce equivalent norms with equivalence constants independent of $n\in\N$. Therefore, defining $\mathbf L_n := \mathbf L_{1,n}$, for $\theta\in[-1, 1]$ we define the discrete fractional Sobolev space induced by the matrices $\mathbf L_n^\theta$, so that for $\mathfrak u_n, \mathfrak v_n \in \mathfrak V_n$,
\begin{equation}
    \label{eq: discrete laplacian scalar product}
    \langle \mathbf u_n, \mathbf v_n \rangle_{\mathbf H^\theta_n} := \mathbf u_n^T \mathbf L_n^\theta \mathbf v_n = \left\langle \mathfrak u_n, \mathfrak v_n \right\rangle_{\mathfrak H^\theta_n}, \qquad \|\mathbf u_n\|_{\mathbf H^{\theta}_n}:=\langle \mathbf u_n, \mathbf u_n\rangle_{\mathbf H^\theta_n}, \qquad \mathbf H^\theta_n := \left(\R^n, \|\cdot\|_{\mathbf H^\theta_n}\right).
\end{equation}

We introduce the following lemma establishing continuity of the operator $\mathfrak I_n$.

\begin{lemma}
    \label{lemma: In continuity in Htheta}
    There exists a constant $C>0$ such that for all $\theta\in[0,1]$, $n\in\N$, and $u\in H^\theta_0$,
    \[
        \|\mathfrak I_n (u)\|_{H^\theta_0} \leq C \|u\|_{H^\theta_0}.
    \]
\end{lemma}

\begin{proof}
    See Appendix \ref{subsec: proof of lemma In continuity in Htheta}.
\end{proof}

The following lemma addresses the discretization of powers of functions.
\begin{lemma}
    \label{lemma: discretization estimates for function powers}
    There exists a constant $C > 0$, independent of $n\in\N$, such that for any $\mathbf u_n \in \R^n_+$, $\theta\in[0,1]$, and $\mu \geq 1$ it holds
    \[
        \begin{aligned}
            \left\| \mathfrak P^{-1}_n (\mathbf u_n)^\mu \right\|_{H^\theta_0} \leq C \left\| \mathfrak P^{-1}_n (\mathbf u_n^\mu) \right\|_{H^\theta_0}.
        \end{aligned}
    \]
\end{lemma}

\begin{proof}
    See Appendix \ref{subsec: proof of discretization estimates for function powers}.
\end{proof}

The following statement establishes equivalence of the discrete and continuous fractional Sobolev norms.
\begin{lemma}
    \label{lemma: equivalence of discrete and continuous norms}
    Let $\mathfrak u_n \in \mathfrak H^\theta_n$. There exist constants $C_1 > C_0 > 0$, independent of $n\in\N$, such that for $\theta\in[-1, 1]$,
    \[
        C_0 \|\mathfrak u_n \|_{H^\theta_0} \leq \|\mathfrak u_n\|_{\mathfrak H^\theta_n} \leq C_1 \|\mathfrak u_n \|_{H^\theta_0}.
    \]
\end{lemma}

\begin{proof}
    We want to invoke \cite[Lemma 2.3]{arioli2009discrete}. To this end, it suffices to prove that there exist $M_0, M_1 > 0$ such that for all $n\in\N$,
    \[
        \begin{aligned}
            \|\mathfrak I_n (u)\|_{L^2} \leq M_0 \|u\|_{L^2} \quad \forall u \in L^2, \\
            \|\mathfrak I_n (u)\|_{H^1_0} \leq M_1 \|u\|_{H^1_0} \quad \forall u \in H^1_0 .
        \end{aligned}
    \]
    But this follows directly from Lemma \ref{lemma: In continuity in Htheta}.
\end{proof}

We proceed with the discrete counterparts of the estimates for the Sobolev norms of powers in Appendix \ref{sec: appendix sobolev power estimates}.

\begin{lemma}
    \label{lemma: discrete sobolev power estimates}
    For $\mathbf u_n \in \R^n_+$, $\mu \geq 1$, $\theta \in \left(0,1\right]$, and $p\in[1,\infty)$,
    \[
        \left\| \mathbf u_n^\mu \right\|_{\mathbf H^\theta_n} \lesssim \left\| \mathbf u_n \right\|_{\mathbf H^\theta_n} \left\| \mathbf u_n \right\|^{\mu-1}_{\infty}
    \]
    For $\mathbf u_n \in \R^n_+$, $\mu \leq 1$, $\theta \in \left(0,1\right]$, and $p\in[1,\infty)$,
    \[
        \left\| \mathbf u_n^\mu \right\|_{\mathbf H^{\mu\theta}_n} \lesssim \left\| \mathbf u_n \right\|_{\mathbf H^\theta_n}^{\mu}.
    \]
\end{lemma}

\begin{proof}
    See Appendix \ref{subsec: proof of discrete sobolev power estimates}.
\end{proof}

Finally, we also introduce the following discrete Poincaré-type inequality.

\begin{lemma}
    \label{lemma: discrete poincare type result}
    Let $-1 \leq \theta < \theta' \leq 1$. Then for $\mathbf u_n\in \R^n$,
    \[
        \|\mathbf u_n\|_{\mathbf H^\theta_n} \lesssim \|\mathbf u_n\|_{\mathbf H^{\theta'}_n}.
    \]
\end{lemma}

\begin{proof}
    By Lemma \ref{lemma: poincare type result} and Lemma \ref{lemma: equivalence of discrete and continuous norms} we have that
    \[
        \|\mathbf u_n\|_{\mathbf H^\theta_n} = \|\mathfrak u_n\|_{\mathfrak H^\theta_n} \lesssim \left\|\mathfrak u_n \right\|_{H^\theta_0} \leq \left\|\mathfrak u_n\right\|_{H^{\theta'}_0} \lesssim \|\mathfrak u_n\|_{\mathfrak H^{\theta'}_n} = \|\mathbf u_n\|_{\mathbf H^{\theta'}_n}.
    \]
\end{proof}

\subsection{Setting}

We consider a parameter $\alpha\in\left[4,\infty\right)$, which corresponds to the exponent in the Laplacian term in \eqref{eq: intro1}, and a non-negative initial value $u_0 \in H^{(\alpha-2)/\alpha}_0$. We define the Hilbert space $U := L^2$ and an operator $Q \in \mathcal L( U, U )$ by
\[
    Q g_k = \mu_k^2 g_k,\quad k\in\N,
\]
such that the $\mu_k \geq 0$ satisfy
\begin{equation}
    \label{eq: noise coloring}
    \sum_{k=1}^\infty k^2 \mu_k^2 < \infty.
\end{equation}
We define a $U$-valued $Q$--Wiener process
\[
    W(t) = \sum_{k=1}^\infty \mu_k g_k \xi_k(t),\quad t\in[0,T]
\]
where the processes $\xi_k$ are identically distributed independent one-dimensional standard Brownian motions. We define $U_0 := Q^{\frac12}(U)$ and consider the induced norm. Moreover, we denote by $\Pi_n$ the orthogonal projection onto $\operatorname{span}\{g_1, \dots, g_n\}$ in $U$, so that in particular
\[
    \Pi_n W(t) := \sum_{k=1}^n \langle W(t), g_k \rangle g_k.
\]

We introduce parameters corresponding to time- and space-Hölder regularity of the solution given by
\begin{equation}
    \label{eq: definition of beta1 and beta2}
    \beta_1 := \frac{(\alpha-2)(\alpha-4)}{4(\alpha-1)(\alpha+1)(\alpha^2+\alpha-4)},\qquad \beta_2 := \frac{\alpha-4}{2\alpha^2}.
\end{equation}
For $\delta\in\left(0,1\right)$ we define the spaces
\[
    \begin{aligned}
        \Xi_{u} &:= C(\mathcal Q_T),\\
        \Xi_{\mathfrak u}^\delta &:= C^{\delta\beta_1, \delta\beta_2}_0(\mathcal Q_T), \\
        \Xi_\eta &:= L^2\left([0,T], L^2 \right)_{weak}, \\
        \Xi_W &:= C([0,T], U),
    \end{aligned}
\]
and we define their product space
\[
    \Xi^\delta := \Xi_u \times \Xi_{\mathfrak u}^\delta \times \Xi_\eta \times \Xi_W.
\]

For some $b \in C^1_b([0,1] \times \R_+, \R_+)$ satisfying for some $C>0$ a growth condition
\begin{equation}
    \label{eq: diffusion coefficient growth condition}
    \sup_{x\in[0,1]} b(x, u) < C(1 + u)
\end{equation}
we consider a diffusion coefficient defined as
\[
    B(u)w(x) := b(x,u(x)) w(x).
\]
In particular, note that $B: L^2 \rightarrow L_2(U_0, L^2)$ due to \eqref{eq: noise coloring}. We also consider the map $\Sigma: L^2 \rightarrow L^1$ defined as
\[
    \Sigma(u) := \sum_{k=1}^\infty \mu_k^2 [B(u)g_k]^2.
\]
Note that $\Sigma(u)(x) = v(x, u(x))$ for
\[
    v(x, u) := b^2(x,u) \sum_{k=1}^\infty \mu_k^2 g_k(x)^2.
\]
Furthermore, given some $r \in C^1_b([0,1] \times \R_+, \R_+)$ satisfying the growth condition
\begin{equation}
    \label{eq: sojourn coefficient growth condition}
    \sup_{x\in[0,1]} r(x, u) < C(1 + u)
\end{equation}
for some $C>0$ and the support condition
\begin{equation}
    \label{eq: sojourn coefficient support condition}
    r(x,u) = \1_{v(x,u) > 0} r(x,u), \qquad (x,u) \in [0,1]\times\R_+,
\end{equation}
we define a coefficient $R: L^2 \rightarrow L^2$ modulating the sojourn behavior of the system by
\[
    R(u)(x) := r(x, u(x)).
\]
One can infer from \eqref{eq: intro1} that lower values of $r$ indicate weaker reflection and thus stronger pinning behavior at the boundary. The growth condition \eqref{eq: sojourn coefficient growth condition} ensures that the pre-limiting system always exhibits non-trivial sojourn behavior at the boundary, as will be discussed in Remark \ref{remark: graham system sojourn}. Furthermore, the support condition \eqref{eq: sojourn coefficient support condition} is used for the identification of the limit of the pushing term from the stochastic argument in Section \ref{sec: convergence}. 

\subsection{Approximate discrete model}

For the coefficients defined so far we also define their discrete counterparts
\begin{align*}
    \qquad
    \mathbf{B}_n(\mathbf u_n)(w) &:= P_n I_n\circ \left[ B \left( P_n^{-1} (\mathbf u_n) \right)(w) \right], &\quad
    \mathbf r_n(\mathbf u_n) &:= P_n I_n\circ R \left( P_n^{-1} (\mathbf u_n) \right),\\
    \boldsymbol\Sigma_n(\mathbf u_n) &:= \sum_{k=1}^n \mu_k^2 [\mathbf B_n (\mathbf u_n)(g_k)]^2, &\quad \boldsymbol\rho_{n}(\mathbf u_n) &:= \mathbf r_n^{-1}(\mathbf u_n),
\end{align*}
where the inverse in the last definition is taken pointwise. We now introduce a finite-dimensional sticky-reflected system serving as a spatially discretized counterpart of \eqref{eq: intro1}--\eqref{eq: intro3}. Let $\mathbf u_n^0 := \mathfrak P_n \mathfrak I_n u_0$ and consider a family of finite-dimensional processes $\{\mathbf u_n: n\in\N\}$, where each ${\mathbf u_n}$ satisfies the $n$-dimensional SDE
\begin{equation}
    \label{eq: discrete dynamics Rn}
    \begin{aligned}
        \d \mathbf u_n(t) &= - \mathbf L_n(\mathbf  u_n^\alpha(t)) \d t
        + \1_{\mathbf u_n(t)>0} \odot \mathbf{B}_n(\mathbf u_n(t)) \Pi_n \d W(t)
        + \1_{\mathbf u_n(t)=0} \odot \mathbf r_n(\mathbf u_n(t)) \d t, \\
        \mathbf u_n(0) &= \mathbf u_n^0.
    \end{aligned}
\end{equation}
We define
\[
    \boldsymbol \eta_n(t) := \1_{\mathbf u_n(t)=0} \odot \mathbf r_n(\mathbf u_n(t)).
\]
Note that it follows from the definition of $\boldsymbol\Sigma_n$ that
\[
    \d \llangle \mathbf u_n \rrangle_t = \1_{\mathbf u_n(r) > 0} \odot \boldsymbol\Sigma_n(\mathbf u_n(t)) \d t.
\]
We define the energy functional
\[
    \mathscr E_n(\mathbf u_n) := \left\| \mathbf u_n^{(\alpha+1)/2} \right\|_{\mathbf H^0_n}^2.
\]
We also fix some $\kappa > 0$, which is used for the critical energy threshold $h_n^{-\kappa}$. 
We now introduce a result giving the existence of a weak solution for the discrete system.

\begin{proposition}
    \label{prop: existence for discrete systems}
    For each $n\in\N$ there exists a stochastic basis $(\Omega, \mathcal F, \mathbb F, \P)$ supporting a $U$--valued $Q$--Wiener process $W$ and an $\R^n$-valued process $\mathbf u_n$ which satisfies the SDE \eqref{eq: discrete dynamics Rn} on $[0,\tau_n]$ and which is constant on $[\tau_n, T]$, where 
    \[
        \tau_n := T \wedge \inf \{ t \geq 0: \mathscr E_n(\mathbf u_n(t)) \geq h_n^{-\kappa} \}.
    \]
    We call the triple $((\Omega, \mathcal F, \mathbb F, \P), \mathbf u_n, W)$ a \emph{martingale solution} to \eqref{eq: discrete dynamics Rn}.
\end{proposition}

\begin{proof}
    We consider the SDE
    \begin{equation}
        \label{eq: discrete dynamics Rn brownian sum}
        \begin{aligned}
            \d \mathbf u_n(t) &= - \mathbf L_n(\mathbf  u_n^\alpha(t)) \d t
            + \1_{\mathbf u_n(t)>0} \odot \sum_{k=1}^n \mathbf{B}_n(\mathbf u_n(t))(\mu_k g_k) \d \xi_k(t)
            + \1_{\mathbf u_n(t)=0} \odot \mathbf r_n(\mathbf u_n(t)) \d t, \\
            \mathbf u_n(0) &= \mathbf u_n^0.
        \end{aligned}
    \end{equation}
    Moreover, let $\sigma: \R_+ \rightarrow[0,1]$ be a smooth cut-off function such that $\sigma \equiv 1$ on $[0, 1]$ and $\sigma \equiv 0$ on $[2, \infty)$. Note that due to their definitions, $\mathbf B_n$ is locally Lipschitz continuous and $\boldsymbol \rho_n$ is continuous. Furthermore, one can easily see that the coefficient $-\mathbf L_n(\mathbf  u_n^\alpha)$ is also locally Lipschitz continuous. Now consider the  modified system \eqref{eq: discrete dynamics Rn brownian sum}, where the coefficients $-\mathbf L_n(\mathbf  u_n^\alpha)$, $\mathbf B_n(\mathbf u_n)$, and $\boldsymbol\rho_n(\mathbf u_n)$ are replaced by their truncated versions
    \begin{equation} \label{eq: truncated coefficients}
        \begin{split}
        \widehat{\mathbf A}_n(\mathbf u_n) &:= -\sigma\left(\frac{\mathscr E_n(\mathbf u_n)}{h_n^{-\kappa}}\right) \mathbf L_n(\mathbf  u_n^\alpha), \qquad \widehat{\mathbf B}_n(\mathbf u_n) := \sigma\left(\frac{\mathscr E_n(\mathbf u_n)}{h_n^{-\kappa}}\right) \mathbf B_n(\mathbf u_n),\\
        \widehat{\boldsymbol \rho}_n(\mathbf u_n) &:= \sigma\left(\frac{\mathscr E_n(\mathbf u_n)}{h_n^{-\kappa}}\right) \boldsymbol \rho_n(\mathbf u_n).
        \end{split}
    \end{equation}
    Note that the truncating pre-factor is bounded and globally Lipschitz, and so are the truncated coefficients in \eqref{eq: truncated coefficients}. Therefore, the system with truncated coefficients admits a weak solution $((\Omega, \mathcal F, \mathbb F, \P), \mathbf u_n, (\xi_{1}, \dots, \xi_{n}))$, i.e.~the process $\mathbf u_n(t \wedge \tau_n)$ satisfies the SDE \eqref{eq: discrete dynamics Rn brownian sum} for $t\in[0,\tau_n]$ and is constant on $[\tau_n, T]$, as it equals the sticky-reflected system with Wentzell boundary condition from \cite[Theorem II.2]{graham1988martingale} with bounded Lipschitz coefficients.

    We can additionally extend the stochastic basis to include further standard Brownian motions 
    $\{\xi_{n+k}: k\in\N\}$ and define
    \[
        W(t) := \sum_{k=1}^\infty \mu_k g_k \xi_k(t).
    \]
   Since
    \[
        \Pi_n W(t) = \sum_{k=1}^n \mu_k g_k \xi_k(t),
    \]
    it then follows that $((\Omega, \mathcal F, \mathbb F, \P), \mathbf u_n, W)$ is a martingale solution to \eqref{eq: discrete dynamics Rn} in the sense of the above definition.
\end{proof}

\begin{remark}[Non-negativity of the approximate system]
    In \cite[Theorem I.10]{graham1988martingale} the author establishes equivalence of the solution of the martingale problem with the solution of the SDE. So while the constraint $\mathbf u_n \geq 0$ is not explicit in \eqref{eq: discrete dynamics Rn}, the solution is still non-negative by construction via the martingale problem.
\end{remark}

\begin{remark}[Non-trivial sojourn behavior]
    \label{remark: graham system sojourn}
    We note that in the sticky-reflected system from \cite[Theorem II.2]{graham1988martingale} the coefficient $\boldsymbol{\rho}_n$ is allowed to vanish, resulting in a purely reflective behavior without sojourn. Thus, one distinguishes between coefficients $\boldsymbol{\gamma}_n$ and $\boldsymbol{\rho}_n$, modulating reflecting and sojourn behavior, respectively. Specifically, instead of \eqref{eq: discrete dynamics Rn} one may consider more generally a family of processes $\{(\mathbf u_n, \mathbf K_n):n\in\N\}$ taking values in $\R^n\times \R^n$ satisfying
    \begin{align}
        \label{eq: explicit discrete dynamics Rn}
        \d \mathbf u_n(t) &= -\mathbf L_n(\mathbf  u_n^\alpha(t)) \d t + \1_{\mathbf u_n(t)>0} \odot \mathbf{B}_n(\mathbf u_n(t)) \Pi_n \d W(t) + \boldsymbol \gamma_n(\mathbf u_n(t)) \d \mathbf K_n(t),\\
        \label{eq: explicit discrete dynamics Rn support condition}
        \d \mathbf K_n(t) &= \1_{\mathbf u_n(t)=0} \odot \d \mathbf K_n(t),\\
        \label{eq: explicit discrete dynamics Rn sojourn condition}
        \1_{\mathbf u_n(t)=0} \d t& = \boldsymbol \rho_{n}(\mathbf u_n(t)) \odot \d \mathbf K_n(t),
    \end{align}
    where one has an explicit support condition \eqref{eq: explicit discrete dynamics Rn support condition} and an explicit sojourn condition \eqref{eq: explicit discrete dynamics Rn sojourn condition}. However, our growth condition \eqref{eq: sojourn coefficient growth condition} implies that $\boldsymbol \rho_n(\mathbf u_n) = \mathbf r_n(\mathbf u_n)^{-1} > 0$, which indicates that all components of $\mathbf u_n$ have non-trivial sojourn behavior at the boundary and precludes the case $\boldsymbol \rho_n = 0$ (reflection without sojourn). This allows us to divide by $\boldsymbol\rho_n(\mathbf u_n)$, so that the boundary behavior of the system is characterized by the quotient $\frac{\boldsymbol \gamma_n(\mathbf u_n)}{\boldsymbol \rho_n(\mathbf u_n)}$. Technically, this leads to the pushing term in \eqref{eq: explicit discrete dynamics Rn} being absolutely continuous, which facilitates its control in Proposition \ref{prop: control of eta and tightness}.
    
    Therefore, without loss of generality we can set $\boldsymbol \gamma_n \equiv 1$ and specify the dynamics in terms of $\mathbf r_n(\mathbf u_n) = \boldsymbol \rho_n^{-1}(\mathbf u_n)$ instead of $\boldsymbol{\rho}_n(\mathbf u_n)$. This way the equation \eqref{eq: explicit discrete dynamics Rn sojourn condition} can be rewritten as
    \[
        \d \mathbf K_n(t) = \1_{\mathbf u_n(t) = 0} \odot \mathbf r_n(\mathbf u_n(t)) \d t
    \]
    and can be plugged directly into \eqref{eq: explicit discrete dynamics Rn}. This obviates the need for the separate specification of \eqref{eq: explicit discrete dynamics Rn support condition} and \eqref{eq: explicit discrete dynamics Rn sojourn condition}, and allows the system \eqref{eq: explicit discrete dynamics Rn}--\eqref{eq: explicit discrete dynamics Rn sojourn condition} to be rewritten as \eqref{eq: discrete dynamics Rn}.
\end{remark}

\subsection{Martingale solutions and main result}

We now formalize the definition of a solution to \eqref{eq: intro1}--\eqref{eq: intro3}. Analytically our definition is derived from the very weak formulation of the deterministic porous medium equation in \cite[Definition 5.2]{vazquez2007porous}, with time-invariant test functions possessing an additional degree of regularity to account for the construction of a solution through spatial discretization. Furthermore, we introduce a condition encoding sticky behavior for the limiting dynamics.

\begin{definition}
    \label{def: solution}
    A tuple $((\Omega, \mathcal F, \mathbb F, \P), u, W)$, where $(\Omega, \mathcal F, \mathbb F, \P)$ is a stochastic basis supporting a $U$-valued $\mathbb F$-adapted $Q$--Wiener process $W$ and an $\mathbb F$-adapted process $u$, is called a \emph{non-negative martingale solution to the stochastic porous medium equation} \eqref{eq: intro1}--\eqref{eq: intro3} if
    \begin{enumerate}[label=(\roman*)]
        \item $u \in \bigcap_{\delta\in(0,1)} \Xi_{\mathfrak u}^\delta$;
        \item $u(t,x) \geq 0$ $\P \otimes \d t \otimes \d x$-a.s.;
        \item for all $\varphi \in C^3_c([0,1],\R)$, $\P \otimes \d t$-a.s.~we have that 
        \begin{align}
            \nonumber
            &\langle \varphi, u(t) \rangle = \left\langle \varphi, u_0 \right\rangle + \int_0^t \left\langle \partial_x^2 \varphi, u^\alpha(r) \right\rangle \d r \\
            \label{eq: continuous dynamics}
            &\qquad+ \int_0^t \langle \varphi, \1_{u(r)>0} B(u(r)) \d W(r) \rangle + \int_0^t \langle \varphi, \1_{u(r)=0} R(u(r)) \rangle \d r.
        \end{align}
    \end{enumerate}
    Furthermore, let $\mathcal S: L^2(\mathcal Q_T) \rightarrow \R$ be a functional given by
    \begin{equation}
        \label{eq: definition of the functional S}
        \mathcal S(u) := \iint_{\mathcal Q_T} \1_{u(t,x)=0} \d t \d x.
    \end{equation}
    We say that the system \eqref{eq: intro1}--\eqref{eq: intro3} exhibits \textit{sticky behavior} if
    \begin{equation}
        \label{eq: limiting stickiness condition}
        \P\left( \mathcal S( u) > 0 \right) > 0.
    \end{equation}
\end{definition}

The following theorem is the main result of the paper. It addresses existence of a martingale solution by approximation through spatial discretization. Furthermore, it states that the limiting system satisfies sticky behavior \eqref{eq: limiting stickiness condition} whenever the pre-limiting systems exhibit non-vanishing sticky behavior uniformly in $n\in\N$.

\begin{theorem}
    \label{thm: main result}
    There exists a stochastic basis $(\Omega, \mathcal F, \mathbb F, \P)$ supporting processes $u \in \bigcap_{\delta\in(0,1)} \Xi^\delta_{\mathfrak u}$, $u_n \in C([0,T], L^2)$ for $n\in\N$ and $Q$--Wiener processes $W$ and $W_n$ for $n\in\N$, s.t.~$((\Omega, \mathcal F, \mathbb F, \P), u, W)$ is a non-negative martingale solution to the stochastic porous medium equation \eqref{eq: intro1}--\eqref{eq: intro3} in the sense of Definition \ref{def: solution}, satisfying for any $p \geq 1$
    \begin{equation}
        \label{eq: energy estimate of solution}
        \E\left[\sup_{r\in[0,T]} \left\| u(r) \right\|_{H^{(\alpha-2)/\alpha}_0}^p \right] + \E\left[\int_0^T \left\| u(r) \right\|_{H^1_0}^p \d r \right] < \infty.
    \end{equation}
    Moreover, for any $n\in\N$ the triples $\left((\Omega, \mathcal F, \mathbb F, \P), P_n(u_n), W_n \right)$ are martingale solutions to \eqref{eq: discrete dynamics Rn} in the sense of Proposition \ref{prop: existence for discrete systems}, and there exists a subsequence of $u_n$ converging to $u$ $\P$-a.s.~in $\Xi_u$. Furthermore, if for some $\varepsilon > 0$ the pre-limiting systems satisfy
    \begin{equation}
        \label{eq: uniform stickiness condition of discrete systems}
        \limsup_{n\rightarrow\infty} \P (\mathcal S(u_n) \geq \varepsilon) > 0,
    \end{equation}
    then the system \eqref{eq: intro1}--\eqref{eq: intro3} exhibits sticky behavior \eqref{eq: limiting stickiness condition}.
\end{theorem}

The proof of the above result is given at the end of Section \ref{sec: convergence}. We finally show that there are examples which have sticky behavior by constructing a discrete system satisfying the uniform stickiness condition \eqref{eq: uniform stickiness condition of discrete systems}.

\begin{example}{\rm
    Let $u_0 = 0$. For simplicity, we set $\mathbf r_n \equiv 1$, and choose $\mathbf{B}_n$ such that, for some constants $0 < m_0 < m_1$ and for all $\mathbf u_n \in \R_+^n$, the spectra satisfy $\sigma(\mathbf{B}_n(\mathbf u_n)) \subset (m_0, m_1)$. We consider a solution from Theorem \ref{thm: main result} and rewrite the dynamics \eqref{eq: discrete dynamics Rn} as
    \begin{equation}
        \label{eq: dynamics with Mt in the uniform stickiness example}
        \d \mathbf u_n(t) = \mathbf A_n (\mathbf u_n (t)) \d t + \d \mathbf M_n(t) + \d \mathbf K_n(t).
    \end{equation}
    We also write $\P_n := \P \circ \mathbf u_n^{-1},\ n\in\N$. For some $K>0$ we define the process
    \[
        \overline{\mathbf M}_n(t) := \mathbf M_n(t) + \int_0^t \1_{\mathbf A_n (\mathbf u_n (r)) \geq -K} \odot \1_{\mathbf u_n(r) > 0} \odot \mathbf A_n (\mathbf u_n (r)) \d r,\quad t\in[0,1].
    \]
    Introducing the localizing sequence 
    \begin{equation}
        \label{eq: localization in the stickiness example}
        \sigma_{i,n}^K := \inf\{t>0: \mathbf A_n(\mathbf u_n(t))(i) > K\}, \qquad \sigma_n^K := \min_{i\in\{1,\dots,n\}} \sigma_{i,n}^K,
    \end{equation}
    we can uniformly bound
    \[
        \mathbf B_n^{-1} (\mathbf u_n(t)) \left[ \1_{\mathbf A_n (\mathbf u_n(t)) \geq -K} \odot \1_{\mathbf u_n(t) > 0} \odot \mathbf A_n (\mathbf u_n(t)) \right]
    \]
    on $[0,\sigma_n^K]$, so that we can invoke Girsanov's theorem and perform a change of measure $\frac{\d \overline{\P}_n}{\d \P_n}$ such that the process $\overline{\mathbf M}_n$ becomes a $\overline{\P}_n$-martingale. Therefore, using that
    \[
        \1_{\mathbf u_n(t) = 0} \odot \1_{\mathbf A_n(\mathbf u_n(t)) < -K} \odot \mathbf A_n(\mathbf u_n(t)) = 0
    \]
    by the definition of $\mathbf A_n$, and that $\1_{\mathbf u_n(t) = 0} \d t = \d \mathbf K_n(t)$, the dynamics \eqref{eq: dynamics with Mt in the uniform stickiness example} can be rewritten as
    \[
        \d \mathbf u_n(t) = \1_{\mathbf A_n (\mathbf u_n (t)) < -K} \odot \mathbf A_n (\mathbf u_n (t)) \d t + \d \overline{\mathbf M}_n(t) + \left\{ 1 + \1_{\mathbf A_n(\mathbf u_n(t)) \geq -K} \odot \mathbf A_n (\mathbf u_n(t)) \right\} \d \mathbf K_n(t).
    \]
    For $\varepsilon > 0$ we define stopping times
    \[
        \tau_n^\varepsilon := \inf \left\{ t>0: \mathcal S \left(u_n \cdot \1_{[0,t]} \right) = \frac1n \sum_{i=1}^n \int_0^t \1_{\mathbf u_n(r)(i) = 0} \d r > \varepsilon \right\}.
    \]
    Using that
    \[
        \1_{\mathbf u_n(t) = 0} \odot \1_{\mathbf A_n(\mathbf u_n(t)) \geq -K} \odot \mathbf A_n(\mathbf u_n(t)) \geq 0
    \]
    due to the definition of $\mathbf A_n$, and the localization \eqref{eq: localization in the stickiness example}, we have
    \[
        \begin{aligned}
        &0 \leq \overline\E_n \left[\frac1n \sum_{i=1}^n \mathbf u_n \left(\tau_n^\varepsilon \wedge \sigma_n^K \right)(i) \right] = \overline\E_n \left[ \frac1n \sum_{i=1}^n \int_0^{\tau_n^\varepsilon \wedge \sigma_n^K} \1_{\mathbf A_n (\mathbf u_n (r))(i) < -K} \odot \mathbf A_n (\mathbf u_n (r))(i) \d r \right] \\
        &\qquad + \overline\E_n \left[ \frac1n \sum_{i=1}^n \int_0^{\tau_n^\varepsilon \wedge \sigma_n^K} \left\{ 1 + \1_{\mathbf A_n(\mathbf u_n(r))(i) \geq -K} \odot \mathbf A_n (\mathbf u_n(r))(i) \right\} \d \mathbf K_n(r) \right] \leq -K \overline\E_n\left[\tau_n^\varepsilon \wedge \sigma_n^K\right] + (1+K) \varepsilon.
        \end{aligned}
    \]
    Therefore, for $t\in[0,T]$ we have
    \[
        \overline{\P}_n \left(\tau_n^\varepsilon \wedge \sigma_n^K > t\right) \leq \frac{\overline \E_n\left[\tau_n^\varepsilon \wedge \sigma_n^K \right]}{t} \leq \frac{\varepsilon(1+K)}{Kt}.
    \]
    Sending $K\rightarrow\infty$, we have that for any $T>0$ there exists $\varepsilon>0$ and a constant $c > 0$ such that for all $n\in\N$,
    \[
        \overline{\P}_n (\tau_n^\varepsilon < T) > c>0.
    \]
    Plugging this into \eqref{eq: definition of the functional S}, we obtain
    \[
        \overline{\P}_n (\mathcal S(u_n) \geq \varepsilon) > c.
    \]
    By equivalence of measures $\overline\P_n$ and $\P_n$, and by equality of laws $\P \circ \mathbf u_n^{-1} = \P_n$ we also have that for all $n \in \N$, 
    \[
        \P (\mathcal S(u_n) \geq \varepsilon) > c,
    \]
    which yields \eqref{eq: uniform stickiness condition of discrete systems}.
}\end{example}

\section{Moment estimates}

\label{sec: moment estimates}

In this section we derive uniform moment estimates for the discretized pre-limiting system. In Section \ref{subsec: estimates 1}, we first obtain moment estimates in $L^\infty([0,T], \mathbf H^0_n)$ and $L^2([s,t], \mathbf H^1_n)$. Thereafter, in Section \ref{subsec: estimates 2} we obtain moment estimates in $L^\infty([0,T], \mathbf H^{(\alpha-2)/\alpha}_n)$, i.e.~with higher spatial regularity, where we invoke the dissipativity of the dynamics. We then obtain in Section \ref{subsec: estimates 3} moment estimates in $W^{s_1,p_1}([0,T], \mathbf H^{-1/\alpha}_n)$ for some $s_1 > 0$, i.e.~with weaker regularity in space and stronger regularity in time, where we rely on discrete Sobolev estimates for pointwise products. We combine these estimates in Section \ref{subsec: estimates 4} to implement an Aubin--Lions-type interpolation argument and thus obtain moment estimates in the interpolation space $W^{s_\omega, p_\omega}([0,T], \mathbf H^0_n)$ for some $\omega\in(0,1)$, and subsequently in $C^\gamma([0,T], \mathbf H^0_n)$ by embedding into Hölder spaces. Ultimately, we obtain moment estimates in space-time Hölder norm $C^{\delta \beta_1, \delta \beta_2}_0(\mathcal Q_T)$ by means of an interpolation argument.

In this section, for each $n\in\N$ we consider a martingale solution to \eqref{eq: discrete dynamics Rn} denoted by
\[
    ((\Omega_n, \mathcal F_n, \mathbb F_n, \P_n), \mathbf u_n, W_n)
\]
in the sense of Proposition \ref{prop: existence for discrete systems}. For brevity we omit the index $n\in\N$ and write
\[
    ((\Omega, \mathcal F, \mathbb F, \P), \mathbf u_n, W) := ((\Omega_n, \mathcal F_n, \mathbb F_n, \P_n), \mathbf u_n, W_n).
\]

We start with the following estimates for the discrete Sobolev norms of products involving discretized coefficients, which are going to be used on multiple occasions henceforth.

\begin{lemma}
    \label{lemma: estimates for discretized coefficients}
    For $n\in\N$, $\mathbf u_n\in \R^n$, $\theta\in[0,1]$, and $\mu \geq 1$ we have
    \begin{align}
        \label{eq: estimates for discretized coefficients B}
        \left\|\mathbf u_n^{\mu} \odot \mathbf B_n(\mathbf u_n) \Pi_n \right\|_{L_2(U_0, \mathbf H^\theta_n)}^2 &\lesssim \left\|\mathbf u_n^\mu\right\|_{\mathbf H^\theta_n}^2 + \left\|\mathbf u_n^{\mu+1}\right\|_{\mathbf H^\theta_n}^2, \\
        \label{eq: estimates for discretized coefficients Sigma}
        \left\|\mathbf u_n^{\mu} \odot \boldsymbol\Sigma_n (\mathbf u_n)\right\|_{\mathbf H^{\theta}_n}^2 &\lesssim \left\|\mathbf u_n^\mu\right\|_{\mathbf H^{\theta}_n}^2 + \left\|\mathbf u_n^{\mu+2}\right\|_{\mathbf H^{\theta}_n}^2, \\
        \label{eq: estimates for discretized coefficients r}
        \left\|\mathbf r_n(\mathbf u_n)\right\|^2_{\mathbf H^0_n} &\lesssim \|\mathbf u_n\|^2_{\mathbf H^0_n}.
    \end{align}
\end{lemma}

\begin{proof}
    See Appendix \ref{subsec: proof of estimates for discretized coefficients}.
\end{proof}

\subsection{Estimates in $L^p(\Omega, L^\infty([0,T], \mathbf H^0_n))$ and $L^{p}(\Omega, L^2([s,t], \mathbf H^1_n))$}

By following a common approach for systems with porous medium diffusion operator \cite[Section 3.2]{vazquez2007porous}, we first derive moment estimates in the norms $L^\infty([0,T], \mathbf H^0_n)$ and $L^2([s,t], \mathbf H^1_n)$ for $\mathbf u_n^{(\alpha+1)/2}$ and $\mathbf u_n^{\alpha}$, respectively.

\label{subsec: estimates 1}
    
\begin{proposition}
    \label{prop: estimates 1}
    Let $p\in[1,\infty)$. (1) Then
    \[
        \begin{aligned}
        \sup_{n\in\N} \E\left[\sup_{r\in[0,T]} \left\|\mathbf u_n^{(\alpha+1)/2}(r) \right\|_{\mathbf H^0_n}^p\right] &< \infty.
        \end{aligned}
    \]
    (2) Moreover, for $0 \leq s < t \leq T$,
    \[
        \begin{aligned}
        \sup_{n\in\N} \E\left[\left(\int_{s \wedge \tau_n}^{t \wedge \tau_n} \left\|\mathbf u_n^\alpha(r)\right\|_{\mathbf H^1_n}^2 \d r\right)^p\right] &\lesssim (t-s)^p + (t-s)^{\frac p2}.
        \end{aligned}
    \]
\end{proposition}

\begin{proof}
    It is sufficient to prove the statement for large $p$. Recall that we denote
    \[
        \mathscr E_n(\mathbf u_n(t)) = \left\|\mathbf u_n^{(\alpha+1)/2}(t)\right\|_{\mathbf H^0_n}^2 = \sum_{i=1}^n h_n \cdot \mathbf u_n^{\alpha+1}(t,i).
    \]
    By It\^o's lemma and \eqref{eq: discrete dynamics Rn} we have
    \begin{equation}
        \label{eq: first application ito lemma}
        \begin{aligned}
        \d \mathbf u_{n}^{\alpha+1}(t) &= (\alpha+1) \mathbf u^{\alpha}_{n}(t) \d \mathbf u_{n}(t) + \frac{\alpha(\alpha+1)}{2} \mathbf u_{n}^{\alpha-1}(t) \d \llangle \mathbf u_{n} \rrangle_t \\
        &= - (\alpha+1) \mathbf u_{n}^{\alpha}(t) \odot \mathbf L_n(\mathbf  u_n^\alpha(t)) \d t + (\alpha+1) \mathbf u_{n}^{\alpha}(t) \odot \mathbf{B}_{n}(\mathbf u_n(t)) \Pi_n \d W(t) \\
        &\qquad+ \frac{\alpha(\alpha+1)}{2} \mathbf u_{n}^{\alpha-1}(t) \odot \boldsymbol\Sigma_n(\mathbf u_n(t)) \d t + \underbrace{(\alpha+1) \mathbf u_{n}^{\alpha}(t) \odot \1_{\mathbf u_{n}(t) = 0} \odot\mathbf r_n(\mathbf u_n(t)) \d t}_{=0}.\\
        \end{aligned}
    \end{equation}
    Note that for $\mathbf w_n \in \R^n$, by \eqref{eq: discrete laplacian scalar product}, 
    \begin{equation}
        \label{eq: coercivity}
        -\sum_{i=1}^n h_n \cdot \mathbf w_n^\alpha(i) \mathbf L_n(\mathbf w_n^\alpha)(i) = -\left\|\mathbf w_n^\alpha\right\|_{\mathbf H^1_n}^2.
    \end{equation}
    Therefore, summing over all components $i\in\{1,\dots,n\}$ and integrating in time, we obtain
    \begin{equation}
        \label{eq: ito lemma for l2 norm of u}
        \begin{aligned}
        &\mathscr E_n(\mathbf u_n(t \wedge \tau_n)) = \mathscr E_n(\mathbf u_n(s \wedge \tau_n)) - (\alpha+1) \int_{s \wedge \tau_n}^{t \wedge \tau_n} \left\|\mathbf u_n^\alpha(r)\right\|_{\mathbf H^1_n}^2 \d r \\
        &\qquad+ (\alpha+1)\int_{s \wedge \tau_n}^{t \wedge \tau_n} \langle \mathbf u_n^\alpha(r), \mathbf B_n(u_n(r)) \Pi_n \d W(r) \rangle_{\mathbf H^0_n} \\
        &\qquad+ \frac{\alpha(\alpha+1)}{2} \int_{s \wedge \tau_n}^{t \wedge \tau_n} \left\|\mathbf u_n^{(\alpha-1)/2}(r) \odot \mathbf B_n(\mathbf u_n(r)) \Pi_n \right\|_{L_2(U_0,\mathbf H^0_n)}^2 \d r.
        \end{aligned}
    \end{equation}
    Furthermore, we apply It\^o's lemma (using the function $f(x)=x^{p/2}$), so that for $p\geq 4$ we obtain
    \[
        \begin{aligned}
        &\mathscr E_n(\mathbf u_n(t \wedge \tau_n))^{\frac p2} = \mathscr E_n(\mathbf u_n(s \wedge \tau_n))^{\frac p2} - \frac{p(\alpha+1)}{2} \int_{s \wedge \tau_n}^{t \wedge \tau_n} \mathscr E_n(\mathbf u_n(r))^{\frac{p-2}{2}} \left\|\mathbf u_n^\alpha(r)\right\|_{\mathbf H^1_n}^2 \d r \\
        &\qquad+ \frac{p(\alpha+1)}{2} \int_{s \wedge \tau_n}^{t \wedge \tau_n} \mathscr E_n(\mathbf u_n(r))^{\frac{p-2}{2}} \langle \mathbf u_n^\alpha(r), \mathbf B_n (\mathbf u_n(r)) \Pi_n\d W(r) \rangle_{\mathbf H^0_n} \\
        &\qquad+ \frac{p\alpha(\alpha+1)}{4} \int_{s \wedge \tau_n}^{t \wedge \tau_n} \mathscr E_n(\mathbf u_n(r))^{\frac{p-2}{2}} \left\| \mathbf u^{(\alpha-1)/2}_n(r) \odot \mathbf B_n(\mathbf u_n(r)) \Pi_n \right\|_{L_2(U_0,\mathbf H^0_n)}^2 \d r \\
        &\qquad+ \frac{p(p-2)(\alpha+1)^2}{8} \int_{s \wedge \tau_n}^{t \wedge \tau_n} \mathscr E_n(\mathbf u_n(r))^{\frac{p-4}{2}} \left\|\left(\mathbf u_n^{(\alpha-1)/2}(r) \odot \mathbf B_n(\mathbf u_n(r))\Pi_n\right)^* \mathbf u_n^{(\alpha+1)/2}(r) \right\|_{U_0}^2 \d r.
        \end{aligned}
    \]
    Therefore, due to \eqref{eq: estimates for discretized coefficients B} from Lemma \ref{lemma: estimates for discretized coefficients} we obtain
    \begin{equation}
        \label{eq: pathwise inequality for control of u lp and h1 norms}
        \begin{split}
        \mathscr E_n(\mathbf u_n(t \wedge \tau_n))^{\frac p2} &\lesssim 1 + \mathscr E_n(\mathbf u_n(s \wedge \tau_n))^{\frac p2} + \int_{s \wedge \tau_n}^{t \wedge \tau_n} \left( 1 + \mathscr E_n(\mathbf u_n(r)) \right) \mathscr E_n(\mathbf u_n(r))^{\frac{p-2}{2}} \d r \\
        &\qquad+ \int_{s \wedge \tau_n}^{t \wedge \tau_n} \mathscr E_n(\mathbf u_n(r))^{\frac{p-2}{2}} \left\langle \mathbf u_n^\alpha(r), \mathbf B_n (\mathbf u_n(r)) \Pi_n \d W(r) \right\rangle_{\mathbf H^0_n}.
        \end{split}
    \end{equation}
    Moreover, we have
    \[
        \begin{aligned}
        &\E\left[\sup_{r\in[s \wedge \tau_n,t \wedge \tau_n]} \mathscr E_n(\mathbf u_n(r))^{\frac p2}\right] \lesssim 1 + \E\left[ \mathscr E_n(\mathbf u_n(s \wedge \tau_n))^{\frac p2} \right] + \E\left[ \int_{s \wedge \tau_n}^{t \wedge \tau_n} \mathscr E_n(\mathbf u_n(r))^{\frac p2} \d r \right] \\
        &\qquad+ \E\left[\sup_{\sigma\in[s \wedge \tau_n,t \wedge \tau_n]} \left|\int_{s \wedge \tau_n}^{\sigma \wedge \tau_n} \mathscr E_n(\mathbf u_n(r))^{\frac{p-2}{2}} \left\langle \mathbf u_n^{(\alpha+1)/2}(r), \mathbf u_n^{(\alpha-1)/2}(r) \odot \mathbf B_n(\mathbf u_n(r)) \Pi_n \d W(r) \right\rangle_{\mathbf H^0_n} \right|\right].
        \end{aligned}
    \]
    By applying Burkholder--Davis--Gundy inequality, Young's inequality, and \eqref{eq: estimates for discretized coefficients B} from Lemma \ref{lemma: estimates for discretized coefficients}, we obtain for $\varepsilon > 0$,
    \[
        \begin{aligned}
        &\E\left[\sup_{\sigma\in[s,t]} \left|\int_{s \wedge \tau_n}^{\sigma \wedge \tau_n} \mathscr E_n(\mathbf u_n(r))^{\frac{p-2}{2}} \left\langle \mathbf u_n^{(\alpha+1)/2}(r), \mathbf u_n^{(\alpha-1)/2}(r) \odot \mathbf B_n(\mathbf u_n(r)) \Pi_n \d W(r) \right\rangle_{\mathbf H^0_n} \right|\right] \\
        &\qquad \lesssim \E\left[\left(\int_{s \wedge \tau_n}^{t \wedge \tau_n} \mathscr E_n(\mathbf u_n(r))^{p-1} \left\|\mathbf u_n^{(\alpha-1)/2}(r) \odot \mathbf B_n(\mathbf u_n(r)) \Pi_n \right\|^2_{L_2(U_0,\mathbf H^0_n)} \d r\right)^{\frac12}\right] \\
        &\qquad \lesssim \E\left[\sup_{r\in[s \wedge \tau_n, t \wedge \tau_n]} \mathscr E_n(\mathbf u_n(r))^{\frac p4} \left(\int_{s \wedge \tau_n}^{t \wedge \tau_n} \mathscr E_n(\mathbf u_n(r))^{\frac{p-2}{2}} \left\| \mathbf u_n^{(\alpha-1)/2}(r) \odot \mathbf B_n(\mathbf u_n(r)) \Pi_n \right\|^2_{L_2(U_0,\mathbf H^0_n)} \d r \right)^{\frac12}\right] \\
        &\qquad \leq \varepsilon \E\left[\sup_{r\in[s \wedge \tau_n, t \wedge \tau_n]} \mathscr E_n(\mathbf u_n(r))^{\frac p2} \right] + C_\varepsilon \E\left[\int_{s \wedge \tau_n}^{t \wedge \tau_n} \mathscr E_n(\mathbf u_n(r))^{\frac{p-2}{2}} \left\|\mathbf u_n^{(\alpha-1)/2}(r) \odot \mathbf B_n(\mathbf u_n(r)) \Pi_n \right\|_{L_2(U_0,\mathbf H^0_n)}^2 \d r \right] \\
        &\qquad \leq \varepsilon \E\left[\sup_{r\in[s \wedge \tau_n, t \wedge \tau_n]} \mathscr E_n(\mathbf u_n(r))^{\frac p2} \right] + C_\varepsilon \left(1 + \E\left[ \int_{s \wedge \tau_n}^{t \wedge \tau_n} \mathscr E_n(\mathbf u_n(r))^{\frac p2} \d r \right]\right).
        \end{aligned}
    \]
    Now to prove (1) we choose $s=0$ and $t= T$. Note that Lemma \ref{lemma: poincare type result}, Lemma \ref{lemma: power sobolev estimate nu <= 1}, and $u_0 \in H^{(\alpha-2)/\alpha}_0$ imply that
    \[
        \sup_{n\in\N} \mathscr E_n(\mathbf u_n(0)) = \sup_{n\in\N} \left\|\mathbf u_n^{(\alpha+1)/2}(0)\right\|_{\mathbf H^0_n}^2 \lesssim \left\|u^{(\alpha+1)/2}_0\right\|_{L^2}^2 \lesssim \left\| u_0 \right\|^{\alpha+1}_{H^{(\alpha-2)/\alpha}_0} < \infty.
    \]
    Therefore, by choosing $\varepsilon$ appropriately, invoking Gronwall's inequality, and recalling that $\mathbf u_n$ is constant on $[\tau_n, T]$ we obtain
    \begin{equation}
        \label{eq: bound on the sup u l2}
        \begin{aligned}
        &\sup_{n\in\N} \E\left[\sup_{r\in[0,T]} \mathscr E_n(\mathbf u_n(r))^{\frac p2} \right] \lesssim 1 + \sup_{n\in\N} \mathscr E_n(\mathbf u_n(0))^{\frac p2} < \infty,
        \end{aligned}
    \end{equation}
    as required. 
    
    For (2), note that it trivially holds that
    \[
        \sup_{\sigma\in[s \wedge \tau_n,t \wedge \tau_n]} \mathscr E_n(\mathbf u_n(\sigma)) - \mathscr E_n(\mathbf u_n(s \wedge \tau_n)) \geq 0.
    \]
    Therefore, due to \eqref{eq: ito lemma for l2 norm of u} we can write
    \[
        \begin{aligned}
        &\int_{s \wedge \tau_n}^{t \wedge \tau_n} \left\|\mathbf u_n^{\alpha}(r) \right\|_{\mathbf H^1_n}^2 \d r \leq \frac{\alpha}{2} \int_{s \wedge \tau_n}^{t \wedge \tau_n} \left\|\mathbf u_n^{(\alpha-1)/2}(r) \odot \mathbf B_n (\mathbf u_n(r))\Pi_n \right\|_{L_2(U_0, \mathbf H^0_n)}^2 \d r\\
        &\qquad + \sup_{\sigma\in[s \wedge \tau_n, t \wedge \tau_n]} \int_{s \wedge \tau_n}^{\sigma \wedge \tau_n} \left\langle \mathbf u_n^\alpha(r), \mathbf B_n(\mathbf u_n(r)) \Pi_n \d W(r) \right\rangle_{\mathbf H^0_n}.
        \end{aligned}
    \]
    Moreover, by virtue of \eqref{eq: estimates for discretized coefficients B} from Lemma \ref{lemma: estimates for discretized coefficients} we obtain
    \[
        \begin{aligned}
        &\E\left[ \left(\int_{s \wedge \tau_n}^{t \wedge \tau_n} \left\|\mathbf u_n^\alpha(r) \right\|_{\mathbf H^1_n}^2 \d r \right)^p \right] \lesssim \E\left[\left(\int_{s \wedge \tau_n}^{t \wedge \tau_n} \left(1 + \mathscr E_n(\mathbf u_n(r)) \right) \d r\right)^p \right] \\
        &\qquad+ \E\left[\sup_{\sigma\in[s \wedge \tau_n, t \wedge \tau_n]} \left|\int_{s \wedge \tau_n}^{\sigma \wedge \tau_n} \left\langle \mathbf u_n^{(\alpha+1)/2}(r), \mathbf u_n^{(\alpha-1)/2}(r) \odot \mathbf B_n(\mathbf u_n(r)) \Pi_n \d W(r) \right\rangle_{\mathbf H^0_n} \right|^p \right].
        \end{aligned}
    \]
    Furthermore, note that by Burkholder--Davis--Gundy inequality, \eqref{eq: estimates for discretized coefficients B} from Lemma \ref{lemma: estimates for discretized coefficients} and \eqref{eq: bound on the sup u l2}, we have
    \[
        \begin{aligned}
        &\E\left[\sup_{\sigma\in[s \wedge \tau_n, t \wedge \tau_n]} \left|\int_{s \wedge \tau_n}^{\sigma \wedge \tau_n} \left\langle \mathbf u_n^{(\alpha+1)/2}(r), \mathbf u_n^{(\alpha-1)/2}(r) \odot \mathbf B_n(\mathbf u_n(r)) \Pi_n \d W(r) \right\rangle_{\mathbf H^0_n} \right|^p\right] \\
        &\qquad \lesssim \E\left[\sup_{r\in[s \wedge \tau_n, t \wedge \tau_n]} \mathscr E_n(\mathbf u_n(r))^{\frac p2} \left(\int_{s \wedge \tau_n}^{t \wedge \tau_n} \left\|\mathbf u_n^{(\alpha-1)/2}(r) \odot \mathbf B_n(\mathbf u_n(r))\Pi_n \right\|^2_{L_2(U_0,\mathbf H^0_n) } \d r \right)^{\frac p2}\right] \\
        &\qquad \lesssim \E\left[\sup_{r\in[s \wedge \tau_n, t \wedge \tau_n]} \mathscr E_n(\mathbf u_n(r))^p \right]^{\frac12} \E\left[\left(\int_{s \wedge \tau_n}^{t \wedge \tau_n} \left(1 + \mathscr E_n(\mathbf u_n(r)) \right) \d r \right)^{p} \right]^{\frac12} \lesssim \E\left[\left(\int_{s \wedge \tau_n}^{t \wedge \tau_n} \left(1 + \mathscr E_n(\mathbf u_n(r)) \right) \d r\right)^p \right]^{\frac12}.
        \end{aligned}
    \]
    Combining the above and using \eqref{eq: bound on the sup u l2}, we obtain
    \[
        \begin{aligned}
        &\E\left[\left(\int_{s \wedge \tau_n}^{t \wedge \tau_n} \left\|\mathbf u_n^\alpha(r) \right\|_{\mathbf H^1_n}^2 \d r \right)^p\right] \lesssim (t-s)^{\frac p2} + (t-s)^p .
        \end{aligned}
    \]
    
\end{proof}

\subsection{Estimates in $L^p(\Omega, L^\infty([0,T], \mathbf H^{(\alpha-2)/\alpha}_n))$}

\label{subsec: estimates 2}

We now concern ourselves with moment estimates in $L^\infty([0,T], \mathbf H^{(\alpha-2)/\alpha}_n)$ for $\mathbf u_n^\alpha$. Apart from exploiting the dissipative structure of the dynamics by virtue of
\[
    -\left\langle\mathbf{u}_n^{\alpha-1}, \mathbf L_n^{(\alpha-2) / \alpha}\left(\mathbf{u}_n^\alpha\right) \odot \mathbf L_n\left(\mathbf{u}_n^\alpha\right)\right\rangle_{\mathbf{H}_n^0} \leq 0,
\]
we also note that the power $\mathbf u_n^{\alpha}$ of $\mathbf u_n$ provides a regularizing effect, which allows us to control terms that involve the diffusion coefficient with the indicator. Specifically, if one naively attempted to derive moment estimates for $\left\| \mathbf u_n(t) \right\|_{\mathbf H^\theta_n}^2$ for some $\theta > 0$, one would run into problems with the term
\begin{equation}
    \label{eq: sup h1 estimate bad term}
    \| \1_{\mathbf u_n > 0} \odot \mathbf B_n(\mathbf u_n)\|_{L_2(U_0, \mathbf H^\theta_n)}
\end{equation}
due to the presence of the indicator. However, since we consider instead moment estimates for the norm $\left\| \mathbf u_n^\alpha(t) \right\|_{\mathbf H^\theta_n}^2$, an application of It\^o's lemma yields a term
\begin{equation}
    \label{eq: sup h1 estimate good term}
    \left\| \mathbf u_n^{\alpha-1} \odot \mathbf B_n(\mathbf u_n) \right\|_{L_2(U_0, \mathbf H^\theta_n)}
\end{equation}
instead of \eqref{eq: sup h1 estimate bad term}, which can be controlled by discrete Sobolev estimates for pointwise products from Lemma \ref{lemma: estimates for discretized coefficients}.

\begin{proposition}
    \label{prop: estimates 2}
    For $p\in[1, \infty)$, we have
    \[
        \sup_{n\in\N} \E\left[\sup_{r\in[0,T]} \left\|\mathbf u_n^\alpha(r)\right\|_{\mathbf H^{(\alpha-2)/\alpha}_n}^p \right] < \infty.
    \]
\end{proposition}

\begin{proof}
    It is sufficient to prove the statement for large $p$. Let us define $\mathscr G_n(\mathbf u_n) := \left\| \mathbf u_n^\alpha \right\|_{\mathbf H^{(\alpha-2)/\alpha}_n}^2$.
    Note that
    \[
        \left\| \mathbf{u}_n^\alpha \right\|_{\mathbf{H}_n^{(\alpha-2)/\alpha}}^2 = \left\langle \mathbf L_n^{(\alpha-2)/\alpha} \left(\mathbf{u}_n^\alpha\right), \mathbf{u}_n^\alpha \right\rangle_{\mathbf H^0_n}. 
    \]
    Therefore, by It\^o's lemma we obtain
    \[
        \begin{aligned}
            &\mathscr G_n(\mathbf u_n(T \wedge \tau_n)) = \mathscr G_n(\mathbf u_n(0))  -2\alpha \int_0^{T \wedge \tau_n} \left\langle \mathbf u_n^{\alpha-1}(r), \mathbf L_n^{(\alpha-2)/\alpha} (\mathbf u_n^\alpha(r)) \odot \mathbf L_n(\mathbf u_n^\alpha(r)) \right\rangle_{\mathbf H^0_n} \d r \\
            &\qquad+ 2\alpha \int_0^{T \wedge \tau_n} \left\langle \mathbf L_n^{(\alpha-2)/\alpha} (\mathbf u_n^\alpha(r)), \mathbf u_n^{\alpha-1}(r) \odot \mathbf B_n(\mathbf u_n(r)) \Pi_n \d W_r \right\rangle_{\mathbf H^0_n} \\
            &\qquad+ \alpha(\alpha-1) \int_0^{T\wedge \tau_n} \left\langle \mathbf L_n^{(\alpha-2)/\alpha} (\mathbf u_n^\alpha(r)), \mathbf u_n^{\alpha-2} (r) \odot \boldsymbol\Sigma_n(\mathbf u_n(r)) \right\rangle_{\mathbf H^0_n} \d r \\
            &\qquad+ \frac{\alpha^2}{2} \int_0^{T \wedge \tau_n} \left\| \mathbf u_n^{\alpha-1}(r) \odot \mathbf B_n(\mathbf u_n(r)) \Pi_n \right\|_{L_2(U_0, \mathbf H^{(\alpha-2)/\alpha}_n)}^2 \d r.
        \end{aligned}
    \]
    Furthermore, we again apply It\^o's lemma (using the function $f(x) = x^{\frac p2}$), so that for $p\geq 4$ we obtain
    \[
        \mathscr G_n(\mathbf u_n(T \wedge \tau_n))^{\frac p2} = \mathscr G_n(\mathbf u_n(0))^{\frac p2} + \sum_{i=1}^5 K_i,
    \]
    where
    \[
        \begin{aligned}
            K_1 &:= -p\alpha \E\left[\int_0^{T \wedge \tau_n} \mathscr G_n(\mathbf u_n(r))^{\frac{p-2}{2}} \left\langle \mathbf u_n^{\alpha-1}(r), \mathbf L_n^{(\alpha-2)/\alpha}(\mathbf u_n^\alpha(r)) \odot \mathbf L_n(\mathbf u_n^\alpha(r)) \right\rangle_{\mathbf H^0_n} \d r \right],\\
            K_2 &:= \frac{p\alpha^2}{4} \E\left[ \int_0^{T \wedge \tau_n} \mathscr G_n(\mathbf u_n(r))^{\frac{p-2}{2}} \left\| \mathbf u_n^{\alpha-1}(r) \odot \mathbf B_n(\mathbf u_n(r)) \Pi_n \right\|_{L_2(U_0, \mathbf H^{(\alpha-2)/\alpha}_n)}^2 \d r \right], \\
            K_3 &:= p\alpha \E\left[ \int_0^{T \wedge \tau_n} \mathscr G_n(\mathbf u_n(r))^{\frac{p-2}{2}} \left\langle \mathbf L_n^{(\alpha-2)/\alpha}(\mathbf u_n^\alpha(r)), \mathbf u_n^{\alpha-1}(r) \odot \mathbf B_n(\mathbf u_n(r)) \Pi_n \d W_r \right\rangle_{\mathbf H^0_n} \right], \\
            K_4 &:= \frac{p\alpha(\alpha-1)}{2} \E\left[ \int_0^{T\wedge \tau_n} \mathscr G_n(\mathbf u_n(r))^{\frac{p-2}{2}} \left\langle \mathbf L_n^{(\alpha-2)/\alpha}(\mathbf u_n^\alpha(r)), \mathbf u_n^{\alpha-2}(r) \odot \boldsymbol\Sigma_n(\mathbf u_n(r)) \right\rangle_{\mathbf H^0_n} \d r \right], \\
            K_5 &:= \frac{p(p-2)\alpha^2}{2} \E\left[ \int_0^{T \wedge \tau_n} \mathscr G_n(\mathbf u_n(r))^{\frac{p-4}{2}} \left\| (\mathbf u_n^{\alpha-1}(r) \odot \mathbf B_n(\mathbf u_n(r)) \Pi_n)^* \mathbf L_n^{(\alpha-2)/\alpha}( \mathbf u_n^\alpha(r)) \right\|_{U_0}^2 \d r \right].
        \end{aligned}
    \]
    
    \emph{Estimating $K_1$.} Note that the expression
    \[
        \left\langle \mathbf u_n^{\alpha-1}, \mathbf L_n^{(\alpha-2)/\alpha}(\mathbf u_n^\alpha) \odot \mathbf L_n(\mathbf u_n^\alpha) \right\rangle_{\mathbf H^0_n} = h_n \cdot \left( \mathbf L_n^{(\alpha-2)/\alpha} (\mathbf u_n^\alpha) \right)^T \operatorname{diag} \left(\mathbf u_n^{\alpha-1}\right) \mathbf L_n (\mathbf u_n^\alpha)
    \]
    can be rewritten in the form
    \[
        h_n \cdot e^T D^\gamma \left( D^{1/2} \mathbf L_n^{(\alpha-2)/\alpha} D^{1/2} \right) \left( D^{1/2} \mathbf L_n D^{1/2} \right) D^\gamma e,
    \]
    where $D := \operatorname{diag}(\mathbf u_n^{\alpha-1})$, $e = (1, \dots, 1)$ and $\gamma = \frac{\alpha+1}{2(\alpha-1)}$. Since the matrices $\mathbf L_n$ and $\mathbf L_n^{(\alpha-2)/\alpha}$ are positive definite, so are $D^{1/2} \mathbf L_n D^{1/2}$ and $D^{1/2} \mathbf L_n^{(\alpha-2)/\alpha} D^{1/2}$. Thus their product
    \begin{equation}\label{eq:diagprod}
        \left( D^{1/2} \mathbf L_n^{(\alpha-2)/\alpha} D^{1/2} \right) \left( D^{1/2} \mathbf L_n D^{1/2} \right)
    \end{equation}
    has positive eigenvalues $\kappa_1, \dots, \kappa_n > 0$, and hence for some orthogonal matrix $O$ we can rewrite \eqref{eq:diagprod} as 
    \[
        h_n \cdot e^T D^\gamma O^T \operatorname{diag}(\kappa_1, \dots, \kappa_n) O D^\gamma e = h_n \cdot \sum_{i=1}^n \kappa_i \left(\sum_{j=1}^n o_{ij} \mathbf u_n^{(\alpha+1)/2}(j) \right)^2 \geq 0.
    \]
    Combining these ideas we have
    \[
        -\left\langle \mathbf u_n^{\alpha-1}, \mathbf L_n^{(\alpha-2)/\alpha}(\mathbf u_n^\alpha) \odot \mathbf L_n(\mathbf u_n^\alpha) \right\rangle_{\mathbf H^0_n} \leq 0,
    \]
    so that we have
    \[
        K_1 \leq 0.
    \]

    \emph{Estimating $K_2$.} We obtain by \eqref{eq: estimates for discretized coefficients B} from Lemma \ref{lemma: estimates for discretized coefficients} and Lemma \ref{lemma: discrete sobolev power estimates}
    \[
        \begin{aligned}
            &\mathscr G_n(\mathbf u_n)^{\frac{p-2}{2}} \|\mathbf  u_n^{\alpha-1} \odot \mathbf B_n(\mathbf u_n) \Pi_n\|_{L_2(U_0, \mathbf H^{(\alpha-2)/\alpha}_n)}^2 \lesssim \mathscr G_n(\mathbf u_n)^{\frac{p-2}{2}} \left( \left\| \mathbf u_n^{\alpha-1} \right\|_{\mathbf H^{(\alpha-2)/\alpha}_n}^2 + \left\| \mathbf u_n^{\alpha} \right\|_{\mathbf H^{(\alpha-2)/\alpha}_n}^2 \right) \\
            &\qquad \lesssim \mathscr G_n(\mathbf u_n)^{\frac{p-2}{2}} \left( 1 + \left\| \mathbf u_n^{\alpha} \right\|_{\mathbf H^1_n}^2 + \left\| \mathbf u_n^{\alpha} \right\|_{\mathbf H^{(\alpha-2)/\alpha}_n}^2 \right) \lesssim 1 + \mathscr G_n(\mathbf u_n)^{\frac p2} + \mathscr G_n(\mathbf u_n)^{\frac{p-2}{2}} \left\| \mathbf u_n^\alpha \right\|_{\mathbf H^1_n}^2.
        \end{aligned}
    \]
    Moreover, by Young's inequality we have
    \[
        \int_0^{T \wedge \tau_n} \mathscr G_n(\mathbf u_n(r))^{\frac{p-2}{2}} \|\mathbf u_n^\alpha(r) \|_{\mathbf H^1_n}^2 \d r \leq \varepsilon \sup_{r\in[0,T\wedge\tau_n]} \mathscr G_n(\mathbf u_n(r))^{\frac p2} + C_\varepsilon \left(\int_0^{T \wedge \tau_n} \mathscr \|\mathbf u_n^\alpha(r) \|_{\mathbf H^1_n}^2 \d r\right)^{\frac p2}.
    \]
    Hence taking expectations, integrating and using the above two results, we obtain
    \begin{equation}
        \label{eq: estimate in I2 used in I3}
        \begin{aligned}
            &\E\left[ \int_0^{T \wedge \tau_n} \mathscr G_n(\mathbf u_n(r))^{\frac{p-2}{2}} \|\mathbf  u_n^{\alpha-1}(r) \odot \mathbf B_n(\mathbf u_n(r)) \Pi_n\|_{L_2(U_0, \mathbf H^{(\alpha-2)/\alpha}_n)}^2 \d r \right] \\
            &\qquad\lesssim 1 + \E\left[ \int_0^{T \wedge \tau_n} \mathscr G_n(\mathbf u_n(r))^{\frac p2} \d r \right] + \varepsilon\E\left[ \sup_{r\in[0,T\wedge\tau_n]} \mathscr G_n(\mathbf u_n(r))^{\frac p2} \right] + C_\varepsilon \E\left[ \left( \int_0^{T \wedge \tau_n} \|\mathbf u_n^\alpha(r)\|_{\mathbf H^1_n}^2 \d r \right)^{\frac p2} \right].
        \end{aligned}
    \end{equation}
    
    \emph{Estimating $K_3$.}
    By invoking the Burkholder--Davis--Gundy inequality and \eqref{eq: estimate in I2 used in I3} we get
    \[
        \begin{aligned}
            &\E\left[\sup_{t\in[0,T]} \left|\int_0^{T \wedge \tau_n} \mathscr G_n(\mathbf u_n(r))^{\frac{p-2}{2}} \left\langle \mathbf L_n^{(\alpha-2)/\alpha} (\mathbf u_n^\alpha(r)), \mathbf u_n^{\alpha-1}(r) \odot \mathbf B_n(\mathbf u_n(r)) \Pi_n \d W(r) \right\rangle_{\mathbf H^0_n} \right|\right] \\
            &\lesssim \E\left[\left|\int_0^{T \wedge \tau_n} \mathscr G_n(\mathbf u_n(r))^{p-1} \left\| \mathbf u_n^{\alpha-1}(r) \odot \mathbf B_n(\mathbf u_n(r)) \Pi_n \right\|_{L_2(U_0, \mathbf H^{(\alpha-2)/\alpha}_n)}^2 \d r \right|^{\frac12} \right] \\
            &\lesssim \E\left[\sup_{r\in[0,T\wedge\tau_n]} \mathscr G_n(\mathbf u_n(r))^{\frac p4} \left|\int_0^{T \wedge \tau_n} \mathscr G_n(\mathbf u_n(r))^{\frac{p-2}{2}} \|\mathbf  u_n^{\alpha-1}(r) \odot \mathbf B_n(\mathbf u_n(r)) \Pi_n\|_{L_2(U_0, \mathbf H^{(\alpha-2)/\alpha}_n)}^2 \d r \right|^{\frac12} \right] \\
            &\lesssim  \varepsilon \E\left[ \sup_{r\in[0,T\wedge\tau_n]} \mathscr G_n(\mathbf u_n(r))^{\frac p2} \right] + C_\varepsilon \E\left[ \int_0^{T \wedge \tau_n} \mathscr G_n(\mathbf u_n(r))^{\frac{p-2}{2}} \|\mathbf  u_n^{\alpha-1}(r) \odot \mathbf B_n(\mathbf u_n(r)) \Pi_n\|_{L_2(U_0, \mathbf H^{(\alpha-2)/\alpha}_n)}^2 \d r \right] \\
            &\lesssim \varepsilon\E\left[ \sup_{r\in[0,T\wedge\tau_n]} \mathscr G_n(\mathbf u_n(r))^{\frac p2} \right] + C_\varepsilon \E\left[ \int_0^{T \wedge \tau_n} \mathscr G_n(\mathbf u_n(r))^{\frac p2} \d r \right] + C_\varepsilon \E\left[ \left( \int_0^{T \wedge \tau_n} \|\mathbf u_n^\alpha(r)\|_{\mathbf H^1_n}^2 \d r \right)^{\frac p2} \right].
        \end{aligned}
    \]

    \emph{Estimating $K_4$.} We obtain by \eqref{eq: estimates for discretized coefficients Sigma} from Lemma \ref{lemma: estimates for discretized coefficients} and Lemma \ref{lemma: discrete sobolev power estimates}
    \[
        \begin{aligned}
            &\mathscr G_n(\mathbf u_n)^{\frac{p-2}{2}} \|\mathbf  u_n^{\alpha-2} \odot \boldsymbol\Sigma_n(\mathbf u_n)\|_{\mathbf H^{(\alpha-2)/\alpha}_n}^2 \lesssim \mathscr G_n(\mathbf u_n)^{\frac{p-2}{2}} \left( \left\| \mathbf u_n^{\alpha-2} \right\|_{\mathbf H^{(\alpha-2)/\alpha}_n}^2 + \left\| \mathbf u_n^{\alpha} \right\|_{\mathbf H^{(\alpha-2)/\alpha}_n}^2 \right) \\
            &\qquad \lesssim \mathscr G_n(\mathbf u_n)^{\frac{p-2}{2}} \left( 1 + \left\| \mathbf u_n^{\alpha} \right\|_{\mathbf H^1_n}^2 + \left\| \mathbf u_n^{\alpha} \right\|_{\mathbf H^{(\alpha-2)/\alpha}_n}^2 \right) \lesssim 1 + \mathscr G_n(\mathbf u_n)^{\frac p2} + \mathscr G_n(\mathbf u_n)^{\frac{p-2}{2}} \left\| \mathbf u_n^\alpha \right\|_{\mathbf H^1_n}^2.
        \end{aligned}
    \]
    Therefore, by virtue of \eqref{eq: estimates for discretized coefficients Sigma} from Lemma \ref{lemma: estimates for discretized coefficients} we have
    \[
        \begin{aligned}
            &\E\left[\int_0^{T\wedge \tau_n} \mathscr G_n(\mathbf u_n(r))^{\frac{p-2}{2}} \left\langle \mathbf L_n^{(\alpha-2)/\alpha}(\mathbf u_n^\alpha(r)), \mathbf u_n^{\alpha-2} (r) \odot \boldsymbol\Sigma_n(\mathbf u_n(r)) \right\rangle_{\mathbf H^0_n} \d r\right] \\
            &\qquad\leq \varepsilon \E\left[ \sup_{r\in[0,T\wedge\tau_n]} \mathscr G_n(\mathbf u_n(r))^{\frac p2} \right] + C_\varepsilon \E\left[ \int_0^{T\wedge \tau_n} \mathscr G_n(\mathbf u_n(r))^{\frac{p-2}{2}} \left\|\mathbf u_n^{\alpha-2}(r) \odot \boldsymbol\Sigma_n(\mathbf u_n(r)) \right\|_{\mathbf H^{(\alpha-2)/\alpha}_n}^2 \d r \right] \\
            &\qquad\lesssim 1 + \varepsilon\E\left[ \sup_{r\in[0,T\wedge\tau_n]} \mathscr G_n(\mathbf u_n(r))^{\frac p2} \right] + C_\varepsilon \E\left[ \int_0^{T \wedge \tau_n} \mathscr G_n(\mathbf u_n(r))^{\frac p2} \d r \right] + C_\varepsilon \E\left[ \left( \int_0^{T \wedge \tau_n} \|\mathbf u_n^\alpha(r)\|_{\mathbf H^1_n}^2 \d r \right)^{\frac p2} \right].
        \end{aligned}
    \]

    \emph{Estimating $K_5$.} Note that
    \[
        \left\| \left( \mathbf u_n^{\alpha-1} \odot \mathbf B_n(\mathbf u_n) \Pi_n \right)^* \mathbf L_n^{(\alpha-2)/\alpha}( \mathbf u_n^\alpha) \right\|_{U_0}^2 \leq \mathscr G_n(\mathbf u_n) \left\| \mathbf u_n^{\alpha-1} \odot \mathbf B_n(\mathbf u_n) \Pi_n \right\|_{L_2(U_0, \mathbf H^{(\alpha-2)/\alpha}_n)}^2,
    \]
    so that in the same way as for $K_2$ we conclude that
    \[
        K_5 \lesssim \E\left[ \int_0^{T \wedge \tau_n} \mathscr G_n(\mathbf u_n(r))^{\frac p2} \d r \right] + \varepsilon\E\left[ \sup_{r\in[0,T\wedge\tau_n]} \mathscr G_n(\mathbf u_n(r))^{\frac p2} \right] + C_\varepsilon \E\left[ \left( \int_0^{T \wedge \tau_n} \|\mathbf u_n^\alpha(r)\|_{\mathbf H^1_n}^2 \d r \right)^{\frac p2} \right].
    \]

    \emph{Closing the estimate.} For $\varepsilon>0$ small enough we combine the above to obtain
    \[
        \E\left[\sup_{r\in[0,T \wedge \tau_n]} \mathscr G_n(\mathbf u_n(r))^{\frac p2} \right] \lesssim 1 + \|\mathbf u_n^\alpha(0)\|_{\mathbf H^{(\alpha-2)/\alpha}_n}^2 + \E\left[\int_0^{T \wedge \tau_n} \mathscr G_n(\mathbf u_n(r))^{\frac p2} \d r \right] + \E\left[\left(\int_0^{T \wedge \tau_n} \|\mathbf u_n^\alpha(r)\|_{\mathbf H^1_n}^2 \d r\right)^{\frac p2} \right].
    \]
    As $\mathbf u_n$ is constant on $[\tau_n, T]$,  Proposition \ref{prop: estimates 1} and Gronwall's lemma lead to
    \[
        \sup_{n\in\N} \E\left[\sup_{r\in[0,T]} \mathscr G_n(\mathbf u_n(r))^{\frac p2} \right] \lesssim 1 + \sup_{n\in\N} \|\mathbf u_n^\alpha(0)\|_{\mathbf H^{(\alpha-2)/\alpha}_n}^2.
    \]

    Due to Lemma \ref{lemma: discrete sobolev power estimates}, Sobolev embedding with $\alpha\geq 4$, Lemma \ref{lemma: equivalence of discrete and continuous norms} and Lemma \ref{lemma: In continuity in Htheta} we have
    \[
        \left\|\mathbf u_n^\alpha(0) \right\|_{\mathbf H^{(\alpha-2)/\alpha}_n} = \left\|\left(\mathfrak P_n \mathfrak I_n u_0 \right)^\alpha \right\|_{\mathbf H^{(\alpha-2)/\alpha}_n} = \left\| \mathfrak P_n \mathfrak I_n u_0 \right\|_{\mathbf H^{(\alpha-2)/\alpha}_n}^\alpha \lesssim \left\| \mathfrak I_n u_0 \right\|_{H^{(\alpha-2)/\alpha}_0}^\alpha \lesssim \left\| u_0 \right\|_{H^{(\alpha-2)/\alpha}_0}^\alpha.
    \]
    Therefore, noting that $u_0 \in H^{(\alpha-2)/\alpha}_0$, the claim follows.
    
\end{proof}

\subsection{Estimates in $L^p(\Omega, W^{s,p}([0,T], \mathbf H^{-1/\alpha}_n))$}

\label{subsec: estimates 3}

We now derive moment estimates for $\mathbf u_n^{\alpha+1}$ with lower Sobolev regularity in space, thereby permitting higher regularity in time. We note that in this section we control the moments of the norm of $\mathbf u_n^{\alpha+1}$ instead of $\mathbf u_n^\alpha$, as was the case in the previous section, in order to exploit the estimate
\[
    \left\|\mathbf u_n^\alpha \odot \mathbf L_n(\mathbf u_n^\alpha)\right\|_{\mathbf H^{-1}_n} \lesssim \|\mathbf u_n^\alpha\|_{\mathbf H^1_n}^2.
\]
This estimate is the discrete counterpart of the estimate \cite[Section 4.6.1, Proposition 1]{runst2011sobolev} for pointwise products, and is a consequence of the following lemma.

\begin{lemma}
    \label{lemma: u Lu H-1}
    For $\mathbf v_n, \mathbf w_n \in \R^n$,
    \begin{equation}
        \label{eq: discrete pointwise product h-1}
        \left\| \mathbf v_n \odot \mathbf w_n \right\|_{\mathbf H^{-1}_n} \lesssim \|\mathbf v_n\|_{\mathbf H^{-1}_n} \|\mathbf w_n\|_{\mathbf H^{1}_n}.
    \end{equation}
\end{lemma}

\begin{proof}
    See Appendix \ref{subsec: proof of u Lu H-1}.
\end{proof}

\begin{proposition}
    \label{prop: estimates 3}
    For $s \in\left(0, \frac14 \right)$ and $p\in\left[ 1, \infty \right)$, we have
    \[
        \sup_{n\in\N} \E\left[ \left\| \mathbf u_n^{\alpha+1} \right\|_{W^{s,p}([0,T], \mathbf H^{-1/\alpha}_n)}^p \right] < \infty.
    \]
\end{proposition}

\begin{proof}
    Note that
    \[
        \left\|\mathbf u_n^{\alpha+1}(r) - \mathbf u_n^{\alpha+1}(t)\right\|_{\mathbf H^{-1/\alpha}_n}^2 = \left\langle \mathbf L_n^{-1/\alpha} \left(\mathbf u^{\alpha+1}_{n}(r) - \mathbf u^{\alpha+1}_{n}(t) \right), \mathbf u_{n}^{\alpha+1}(r) - \mathbf u^{\alpha+1}_{n}(t) \right\rangle_{\mathbf H^0_n}.
    \]
    Therefore, by It\^o's lemma we obtain
    \[
        \begin{aligned}
        &\left\|\mathbf u_n^{\alpha+1}((t+\delta) \wedge \tau_n) - \mathbf u_n^{\alpha+1} (t \wedge \tau_n) \right\|_{\mathbf H^{-1/\alpha}_n}^2 \\
        &\qquad= 2(\alpha+1) \int_{t \wedge \tau_n}^{(t+\delta) \wedge \tau_n} \left\langle \mathbf L_n ^{1/2-1/\alpha} \left(\mathbf u_n^{\alpha+1}(r) - \mathbf u_n^{\alpha+1}(t)\right), \mathbf L_n^{-1/2}\left[ \mathbf u_n^\alpha(r) \odot (-\mathbf L_n)(\mathbf  u_n^\alpha(r)) \right] \right\rangle_{\mathbf H^0_n} \d r \\
        &\qquad\qquad +2(\alpha+1) \int_{t \wedge \tau_n}^{(t+\delta) \wedge \tau_n} \left\langle \mathbf L_n^{-1/\alpha} \left(\mathbf u_n^{\alpha+1}(r) - \mathbf u_n^{\alpha+1}(t) \right), \mathbf u_n^\alpha(r) \odot \mathbf B_n(\mathbf u_n(r)) \Pi_n \d W(r) \right\rangle_{\mathbf H^0_n} \\
        &\qquad\qquad +\alpha(\alpha+1) \int_{t \wedge \tau_n}^{(t+\delta) \wedge \tau_n} \left\langle \mathbf L_n^{-1/\alpha} \left(\mathbf u_n^{\alpha+1}(r) - \mathbf u_n^{\alpha+1}(t) \right), \mathbf u_n^{\alpha-1}(r) \odot \boldsymbol\Sigma_n(\mathbf u_n(r)) \right\rangle_{\mathbf H^0_n} \d r \\
        &\qquad\qquad +2 \int_{t \wedge \tau_n}^{(t+\delta) \wedge \tau_n} \left\| \mathbf u_n^\alpha(r) \odot \mathbf B_n(\mathbf u_n(r)) \Pi_n \right\|_{L_2(U_0, \mathbf H^{-1/\alpha}_n)}^2 \d r.
        \end{aligned}
    \]
    Note that since $p\geq 1$, we have
    \[
        \E\left[\left\|\mathbf u_n^{\alpha+1}((t+\delta) \wedge \tau_n) - \mathbf u_n^{\alpha+1}(t \wedge \tau_n) \right\|_{\mathbf H^{-1/\alpha}_n}^{p}\right] \lesssim \sum_{i=1}^4 K_i
    \]
    with
    \[
        \begin{aligned}
            K_1 &:= \E\left[\left|\int_{t \wedge \tau_n}^{(t+\delta) \wedge \tau_n} \left\langle \mathbf L_n^{1/2 -1/\alpha} \left(\mathbf u_n^{\alpha+1}(r) - \mathbf u_n^{\alpha+1}(t)\right), \mathbf L_n^{-1/2} \left[\mathbf u_n^\alpha (r) \odot \mathbf L_n(\mathbf u_n^\alpha(r)) \right] \right\rangle_{\mathbf H^0_n} \d r \right|^{\frac{p}{2}}\right],\\
            K_2 &:= \E\left[\left|\int_{t \wedge \tau_n}^{(t+\delta) \wedge \tau_n} \left\langle \mathbf L_n ^{-1/\alpha} \left(\mathbf u_n^{\alpha+1}(r) - \mathbf u_n^{\alpha+1}(t)\right), \mathbf u_n^\alpha(r) \odot \mathbf B_n(\mathbf u_n(r)) \Pi_n \d W(r) \right\rangle_{\mathbf H^0_n} \right|^{\frac{p}{2}}\right],\\
            K_3 &:= \E\left[\left|\int_{t \wedge \tau_n}^{(t+\delta) \wedge \tau_n} \left\langle \mathbf L_n^{-1/\alpha} \left(\mathbf u_n^\alpha(r) - \mathbf u_n^\alpha(t)\right), \mathbf u_n^{\alpha-1}(r) \odot \boldsymbol\Sigma_n(\mathbf u_n(r)) \right\rangle_{\mathbf H^0_n} \d r \right|^{\frac{p}{2}}\right],\\
            K_4 &:= \E\left[\left|\int_{t \wedge \tau_n}^{(t+\delta) \wedge \tau_n} \left\| \mathbf u_n^{\alpha-1}(r) \odot \mathbf B_n(\mathbf u_n(r)) \Pi_n \right\|_{L_2(U_0,\mathbf H^{-1/\alpha}_n)}^2 \d r\right|^{\frac{p}{2}}\right].
        \end{aligned}
    \]

    \emph{Estimating $K_1$.} 
    We have, due to Lemma \ref{lemma: u Lu H-1}, that
    \[
        \left\|\mathbf L_n^{-1/2} \left[\mathbf u_n^\alpha \odot \mathbf L_n(\mathbf  u_n^\alpha)\right] \right\|_{\mathbf H^0_n} = \|\mathbf u_n^\alpha \odot \mathbf L_n(\mathbf u_n^\alpha) \|_{\mathbf H^{-1}_n} \lesssim \|\mathbf u_n^\alpha\|_{\mathbf H^1_n} \|\mathbf L_n (\mathbf u_n^\alpha) \|_{\mathbf H^{-1}_n} = \|\mathbf u_n^\alpha\|_{\mathbf H^1_n}^2.
    \]
    Furthermore, due to Lemma \ref{lemma: discrete sobolev power estimates} and Sobolev embedding with $\alpha\geq 4$,
    \[
        \left\|\mathbf u_n^{\alpha+1} \right\|_{\mathbf H^{(\alpha-2)/\alpha}_n} \leq \left\|\mathbf u_n^{\alpha} \right\|_{\mathbf H^{(\alpha-2)/\alpha}_n} \left\|\mathbf u_n^{\alpha} \right\|_{\infty}^{1/\alpha} \leq \left\|\mathbf u_n^{\alpha} \right\|_{\mathbf H^{(\alpha-2)/\alpha}_n}^{(\alpha+1)/\alpha}.
    \]
    Therefore, using that $\mathbf u_n$ is constant on $[\tau_n, T]$ and invoking Proposition \ref{prop: estimates 1} and Proposition \ref{prop: estimates 2}, we obtain
    \begin{equation}
        \label{eq: derivation of I1}
        \begin{aligned}
        K_1 &\leq \E\left[\left(\int_{t \wedge \tau_n}^{(t+\delta) \wedge \tau_n} \left\|\mathbf L_n ^{1/2-1/\alpha} \left(\mathbf u^{\alpha+1}_n(r) - \mathbf u^{\alpha+1}_n(t) \right) \right\|_{\mathbf H^0_n} \left\| \mathbf L_n^{-1/2} \left[\mathbf u_n^\alpha(r) \odot \mathbf L_n (\mathbf u_n^\alpha(r)) \right] \right\|_{\mathbf H^0_n} \d r \right)^{\frac{p}{2}}\right] \\
        &\lesssim \E\left[\sup_{r\in[0,T]} \left\|\mathbf u_n^{\alpha+1}(r)\right\|_{\mathbf H^{(\alpha-2)/\alpha}_n}^{\frac{p}{2}} \left(\int_{t \wedge \tau_n}^{(t+\delta) \wedge \tau_n} \left\|\mathbf u_n^\alpha(r)\right\|_{\mathbf H^1_n}^2 \d r \right)^{\frac{p}{2}} \right] \\
        &\lesssim \E\left[\sup_{r\in[0,T]} \left\|\mathbf u_n^{\alpha}(r)\right\|_{\mathbf H^{(\alpha-2)/\alpha}_n}^{\frac{p(\alpha+1)}{\alpha}}\right]^{\frac12} \E\left[\left(\int_{t \wedge \tau_n}^{(t+\delta) \wedge \tau_n} \left\|\mathbf u_n^\alpha(r)\right\|_{\mathbf H^1_n}^2 \d r \right)^{p} \right]^{\frac12} \lesssim \delta^{\frac p2} + \delta^{\frac p4}.
        \end{aligned}
    \end{equation}

    \emph{Estimating $K_2$.}
    Using \eqref{eq: estimates for discretized coefficients B} from Lemma \ref{lemma: estimates for discretized coefficients}, Lemma \ref{lemma: discrete poincare type result}, Lemma \ref{lemma: discrete sobolev power estimates}, and Sobolev embedding with $\alpha\geq 4$ we obtain
    \begin{equation}
        \label{eq: control of uBu}
        \begin{aligned}
        &\left\|\mathbf u_n^{\alpha+1}\right\|_{\mathbf H^{-2/\alpha}_n}^2 \left\|\mathbf u_n^{\alpha} \odot \mathbf B_n(\mathbf u_n) \Pi_n\right\|_{L_2(U_0, \mathbf H^0_n)}^2 \leq \left\|\mathbf u_n^{\alpha+1}\right\|_{\mathbf H^{-2/\alpha}_n}^2 \left(1 + \left\|\mathbf u_n^{\alpha+1}\right\|_{\mathbf H^0_n}^2 \right) \\
        &\qquad \lesssim 1 + \left\|\mathbf u_n^{\alpha+1}\right\|_{\mathbf H^0_n}^4 \lesssim 1 + \left\|\mathbf u_n^{\alpha}\right\|_{\mathbf H^{(\alpha-2)/\alpha}_n} \left\|\mathbf u_n^{\alpha}\right\|_{\infty} ^{\frac{4(\alpha+1)}{\alpha}-1} \lesssim 1 + \left\|\mathbf u_n^{\alpha}\right\|_{\mathbf H^{(\alpha-2)/\alpha}_n} ^{\frac{4(\alpha+1)}{\alpha}}.
        \end{aligned}
    \end{equation}
    Therefore, using the Burkholder--Davis--Gundy inequality, Lemma \ref{lemma: discrete poincare type result}, and Proposition \ref{prop: estimates 2}, we obtain
    \[
        \begin{aligned}
        K_2 &\lesssim \E\left[\left(\int_{t \wedge \tau_n}^{(t+\delta) \wedge \tau_n} \left\|\mathbf u_n^{\alpha+1} (r) \right\|_{\mathbf H^{-2/\alpha}_n}^2 \left\|\mathbf u_n^\alpha(r) \odot \mathbf B_n(\mathbf u_n(r)) \Pi_n \right\|_{L_2(U_0, \mathbf H^0_n)}^2 \d r \right)^{\frac{p}{4}}\right] \\
        &\lesssim \delta^{\frac p4} \left(1 + \E\left[\sup_{r\in[0,T]} \left\| \mathbf u_n^{\alpha}(r) \right\|_{\mathbf H^{(\alpha-2)/\alpha}_n}^{\frac{p(\alpha+1)}{\alpha}} \right] \right) \lesssim \delta^{\frac p4}.
        \end{aligned}
    \]
    
    \emph{Estimating $K_3$.}
    Note that since by virtue of \eqref{eq: estimates for discretized coefficients Sigma} from Lemma \ref{lemma: estimates for discretized coefficients},
    \[
        \left\|\mathbf u_n^{\alpha-1} \odot \boldsymbol\Sigma_n(\mathbf u_n)\right\|_{\mathbf H^0_n}^2 \lesssim 1 + \left\|\mathbf u_n^{\alpha+1}\right\|_{\mathbf H^0_n}^2,
    \]
    we can proceed analogously to $K_2$, to obtain
    \[
        K_3 \lesssim \delta^{\frac p2}.
    \]
    
    \emph{Estimating $K_4$.} By Lemma \ref{lemma: discrete poincare type result} we have
    \[
        \left\| \mathbf u_n^\alpha \odot \mathbf B_n(\mathbf u_n) \Pi_n \right\|_{L_2(U_0, \mathbf H^{-1/\alpha}_n)}^2 \lesssim \left\| \mathbf u_n^\alpha \odot \mathbf B_n(\mathbf u_n) \Pi_n \right\|_{L_2(U_0, \mathbf H^{0}_n)}^2.
    \]
    Thus, proceeding as in the estimate for $K_2$, we obtain
    \[
        K_4 \lesssim \E\left[ \left(\int_{t \wedge \tau_n}^{(t+\delta) \wedge \tau_n} \left\|\mathbf u_n^\alpha(r) \odot \mathbf B_n (\mathbf u_n(r)) \Pi_n \right\|_{L_2(U_0, \mathbf H^0_n)}^2 \d r \right)^{\frac {p}{2}} \right] \lesssim \delta^{\frac p 2}.
    \]
    
    \emph{Closing the estimate.} Combining the previous estimates, we get
    \[
        \E\left[\left\|\mathbf u_n^{\alpha + 1}((t+\delta) \wedge \tau_n) - \mathbf u_n^{\alpha+1}(t \wedge \tau_n)\right\|_{\mathbf H^{-1/\alpha}_n}^{p}\right] \lesssim \delta^{\frac{p}{4}} + \delta^{\frac{p}{2}}.
    \]
    Now, we can write
    \[
        \begin{aligned}
            &\E\left[ \left\| \mathbf u_n^{\alpha+1} \right\|^p_{W^{s,p}([0,T], \mathbf H^{-1/\alpha}_n)}\right] = \iint_{[0,T]^2} \frac{\E\left[\left\|\mathbf u_n^{\alpha+1}(\sigma \wedge \tau_n) - \mathbf u_n^{\alpha+1}(\xi \wedge \tau_n) \right\|_{\mathbf H^{-1/\alpha}_n}^{p}\right] \d \sigma \d \xi }{|\sigma-\xi|^{1+s p}} \\
            &\qquad \lesssim 1 + \iint_{[0,T]^2} |\sigma-\xi|^{\frac{p}{4} - 1 - s p} \d \sigma \d \xi ,
        \end{aligned}
    \]
    where the last integral is finite for $s<\frac14$.
    
\end{proof}

\subsection{Estimates in $L^2(\Omega, C^{\gamma\beta_1, \beta_2}_0(\mathcal Q_T))$ for $\gamma \in (0,1)$}

\label{subsec: estimates 4}

In a first step, we deduce moment estimates in $C^\gamma([0,T], L^2)$ from Propositions \ref{prop: estimates 2} and \ref{prop: estimates 3} by interpolation.

\begin{proposition}
    \label{prop: estimates interpolation}
    Let $\gamma_{max}:= \frac{\alpha-2}{4(\alpha-1)(\alpha+1)}$ and $s\in\left(0,\gamma_{max}\right)$, $p\in\left[1,\infty\right)$. Then for any $q\in[1,\infty)$,
    \[
        \sup_{n\in\N} \E\left[\|\mathbf u_n\|_{W^{s,p}([0,T], \mathbf H^0_n)}^q \right] < \infty.
    \]
    Moreover, for $\gamma \in \left(0, \gamma_{max} \right)$ and $q\in[1,\infty)$,
    \[
        \sup_{n\in\N} \E\left[\|\mathbf u_n\|_{C^\gamma([0,T], \mathbf H^0_n)}^q \right] < \infty.
    \]
\end{proposition}

\begin{proof}

    It is sufficient to prove the statement for large $p$. Note that by Lemma \ref{lemma: discrete sobolev power estimates} and Sobolev embedding with $\alpha\geq 4$ we have
    \[
        \left\| \mathbf u_n^{\alpha+1} \right\|_{\mathbf H_n^{(\alpha-2)/\alpha}} \lesssim \left\| \mathbf u_n^{\alpha} \right\|_{\mathbf H_n^{(\alpha-2)/\alpha}} \left\| \mathbf u_n^{\alpha} \right\|_{\infty}^{1/\alpha} \lesssim \left\| \mathbf u_n^{\alpha} \right\|_{\mathbf H_n^{(\alpha-2)/\alpha}}^{(\alpha+1)/\alpha}.
    \]
    Therefore, due to Propositions \ref{prop: estimates 2} and \ref{prop: estimates 3} we have bounds for the moments of $\mathbf u_n^{\alpha+1}$ in $W^{s_0, p_0}([0,T], \mathbf H^{\theta_0}_n)$ and $W^{s_1, p_1}([0,T], \mathbf H^{\theta_1}_n)$  with
    \begin{align}
        \label{eq: interpolation parameter ranges 0}
        s_0 &= 0,& p_0 &\in [1,\infty),& \theta_0 &= \frac{\alpha-2}{\alpha},\\
        \label{eq: interpolation parameter ranges 1}
        s_1 &\in\left(0, \frac14\right),& p_1 &\in \left[1, \infty\right), & \theta_1 &= -\frac{1}{\alpha}.
    \end{align}
    Denoting the interpolation parameter by $\omega\in(0,1)$, we write
    \[
        s_\omega := (1-\omega) s_0 + \omega s_1, \qquad
        \frac{1}{p_\omega} := \frac{1-\omega}{p_0} + \frac{\omega}{p_1}, \qquad
        \theta_\omega := (1-\omega) \theta_0 + \omega \theta_1.
    \]
    We select $\omega := \frac{\theta_0}{\theta_0 - \theta_1} =\frac{\alpha-2}{\alpha-1}$, so that in particular $\theta_\omega = 0$. Note that we have
    \[
        \left[\mathbf H^{\theta_0}_n , \mathbf H^{\theta_1}_n \right]_{\omega} = \mathbf H^0_n.
    \]
    Therefore, by virtue of the interpolation result \cite[Theorem VII.2.7.2]{amann2019linear} we obtain
    \[
        \begin{aligned}
        &\E\left[ \left\|\mathbf u_n^{\alpha+1} \right\|^\frac{q}{\alpha+1}_{W^{s_\omega, p_\omega}([0,T], \mathbf H^0_n)}\right] \lesssim \E\left[ \left\|\mathbf u_n^{\alpha+1} \right\|_{W^{s_0, p_0}([0,T], \mathbf H^{\theta_0}_n)}^{\frac{q(1-\omega)}{\alpha+1}} \cdot \left\|\mathbf u_n^{\alpha+1} \right\|_{W^{s_1, p_1}([0,T], \mathbf H^{\theta_1}_n)}^{\frac{q\omega}{\alpha+1}}\right] \\
        &\qquad \leq \E\left[ \left\| \mathbf u_n^{\alpha+1} \right\|_{W^{s_0, p_0}([0,T], \mathbf H^{\theta_0}_n)}^{\frac{q p_1(1-\omega)}{(p_1-q\omega)(\alpha+1)}} \right]^{\frac{p_1-q\omega}{p_1}} \E\left[ \left\| \mathbf u_n^{\alpha+1} \right\|_{W^{s_1, p_1}([0,T], \mathbf H^{\theta_1}_n)}^{\frac{p_1}{\alpha+1}} \right]^{\frac {q\omega}{p_1}}.
        \end{aligned}
    \]
    The factors above can be controlled by Propositions \ref{prop: estimates 2} and \ref{prop: estimates 3}, so that we obtain
    \[
        \sup_{n\in\N} \E\left[ \left\|\mathbf u_n^{\alpha+1} \right\|^{\frac{q}{\alpha+1}}_{W^{s_\omega, p_\omega}([0,T], \mathbf H^0_n)}\right] < \infty.
    \]
    Moreover, by the elementary inequality $|a-b| \leq \left| a^\nu - b^\nu \right|^{1/\nu}$ for $a,b \geq 0$, $\nu \geq 1$, we have for $\mathbf v_n, \mathbf w_n \in \R^n_+$,
    \[
        \|\mathbf v_n - \mathbf w_n\|_{\mathbf H^0_n} \leq \left\| \mathbf v_n^\nu - \mathbf w_n^\nu \right\|_{\mathbf H^0_n}^{1/\nu}.
    \]
    Furthermore, note that for the choice $s\in\left(0, \gamma_{max} \right)$ we can always pick $s_1$ according to \eqref{eq: interpolation parameter ranges 1} such that $s = \frac{s_\omega}{\alpha+1}$. Moreover, since $p_0$ and $p_1$ are not constrained, for any $p \geq \alpha+1$ we can pick $p_0$ and $p_1$ large enough so that $p = p_\omega(\alpha+1)$. Thus, using
    \[
        \begin{aligned}
        &\E\left[\left\| \mathbf u_n \right\|^q_{W^{\frac{s_\omega}{\alpha+1},p_\omega(\alpha+1)}([0,T], \mathbf H^0_n)}\right] = \E\left[\left(\iint_{[0,T]^2} \frac{\|\mathbf u_n(\xi) - \mathbf u_n(\sigma)\|_{\mathbf H^0_n}^{p_\omega(\alpha+1)}}{|\xi-\sigma|^{1+s_\omega p_\omega}} \d \xi \d \sigma\right)^{\frac{q}{p_\omega(\alpha+1)}}\right] \\
        &\qquad\leq \E\left[\left(\iint_{[0,T]^2} \frac{\left\| \mathbf u_n^{\alpha+1}(\xi) - \mathbf u_n^{\alpha+1}(\sigma) \right\|_{\mathbf H^0_n}^{p_\omega}}{|\xi-\sigma|^{1+s_\omega p_\omega}} \d \xi \d \sigma\right)^{\frac{q}{p_\omega(\alpha+1)}} \right] = \E\left[\left\| \mathbf u_n^{\alpha+1} \right\|_{W^{s_\omega,p_\omega}([0,T], \mathbf H^0_n)}^{\frac{q}{\alpha+1}}\right]
        \end{aligned}
    \]
    we obtain
    \[
        \begin{aligned}
        &\E\left[\left\| \mathbf u_n \right\|^q_{W^{s,p}([0,T], \mathbf H^0_n)}\right] = \E\left[\left\| \mathbf u_n \right\|^q_{W^{\frac{s_\omega}{\alpha+1},p_\omega(\alpha+1)}([0,T], \mathbf H^0_n)}\right] \leq \E\left[\left\| \mathbf u_n^{\alpha+1} \right\|_{W^{s_\omega,p_\omega}([0,T], \mathbf H^0_n)}^{\frac{q}{\alpha+1}}\right].
        \end{aligned}
    \]
    Finally, taking into consideration the constraint for $s$ and noting that $p$ can be arbitrarily large, by Sobolev embedding into Hölder spaces \cite[Corollary VII.5.6.4]{amann2019linear}, we obtain the estimate of the Hölder norm for $\gamma < \gamma_{max}$.
\end{proof}

Ultimately, we derive a uniform estimate for the moments of the space-time Hölder norm.

\begin{proposition}
    \label{prop: estimates space-time holder}
    Recall the definitions of $\beta_1$ and $\beta_2$ in \eqref{eq: definition of beta1 and beta2}. Then, for $\gamma\in\left(0, 1 \right)$ we have
    \[
        \sup_{n\in\N} \E\left[\|\mathfrak u_n\|_{C^{\gamma\beta_1, \beta_2}_0(\mathcal Q_T)}^2 \right] < \infty.
    \]
\end{proposition}

\begin{proof}
    By virtue of Sobolev embedding into Hölder spaces (cf.~\cite[Corollary VII.5.6.4]{amann2019linear}), Lemma \ref{lemma: power sobolev estimate nu <= 1} and Lemma \ref{lemma: discretization estimates for function powers} we have for $\beta < \beta_2$,
    \[
        \left\| \mathfrak u_n \right\|_{C^\beta([0,1],\R)}^{2} \lesssim \left\| \mathfrak u_n \right\|_{W^{(\alpha-2)/\alpha^2, 2\alpha}([0,1],\R)}^{2} \lesssim \left\| \mathfrak u_n^\alpha \right\|_{H^{(\alpha-2)/\alpha}_0([0,1],\R)}^{2/\alpha} \lesssim \left\| \mathbf u_n^\alpha \right\|_{\mathbf H^{(\alpha-2)/\alpha}_n}^{2/\alpha}.
    \]
  With Proposition \ref{prop: estimates 2} we conclude that
    \begin{equation}
        \label{eq: holder norm in the last moment estimate}
        \sup_{n\in\N} \E\left[\sup_{r\in[0,T]}\|\mathfrak u_n(r)\|_{C^\beta}^2 \right] < \infty,
    \end{equation}
    which in particular implies the space regularity as claimed. For the time regularity, we can write
    \[
        |\mathfrak u_n(t_2, x) - \mathfrak u_n(t_1, x)| = |K_1 + K_2 + K_3|
    \]
    for
    \[
        \begin{aligned}
            K_1 &:= \frac1\delta \int_{x}^{x+\delta} (\mathfrak u_n(t_1, y) - \mathfrak u_n(t_1, x)) \d y ,\\
            K_2 &:= \frac1\delta \int_{x}^{x+\delta} (\mathfrak u_n(t_1, y) - \mathfrak u_n(t_2, y)) \d y ,\\
            K_3 &:= \frac1\delta \int_{x}^{x+\delta} (\mathfrak u_n(t_2, x) - \mathfrak u_n(t_2, y)) \d y .\\
        \end{aligned}
    \]
    For $K_1$ we get by \eqref{eq: holder norm in the last moment estimate} for $\beta < \beta_2$ and some positive random variable $C_1\in L^2(\Omega)$,
    \[
        \begin{aligned}
            |K_1| &\leq \frac1\delta \int_x^{x+\delta} \frac{|\mathfrak u_n(t_1, x) - \mathfrak u_n(t_1, y)| \cdot |x-y|^\beta }{|x-y|^\beta} \d y \\
            &\leq \|\mathfrak u_n(t_1)\|_{C^\beta} \frac1\delta \int_x^{x+\delta} |x-y|^\beta \d y \leq \frac{\delta^\beta \sup_{r\in[0,T]}\|\mathfrak u_n(r)\|_{C^\beta}}{\beta+1} \leq C_1 \delta^\beta.
        \end{aligned}
    \]
    The term $K_3$ is treated analogously. For $K_2$ we get by Lemma \ref{lemma: equivalence of discrete and continuous norms} and Proposition \ref{prop: estimates interpolation} for $\gamma' < \gamma_{max}$ and some positive random variable $C_2 \in L^2(\Omega)$ that
    \begin{align*}
        |K_2| &\leq \frac1\delta \int_x^{x+\delta} |\mathfrak u_n(t_1, y) - \mathfrak u_n(t_2, y)| \d y\\
        &\leq \delta^{-\frac12} \left(\int_x^{x+\delta} |\mathfrak u_n(t_1, y) - \mathfrak u_n(t_2, y)|^2 \d y\right)^{\frac12}\leq C_2 \delta^{-\frac12} |t_2 - t_1|^{\gamma'}.
    \end{align*}
     Combining the estimates for the $K_i$, for $\delta=|t_2 - t_1|^{\kappa}$ with some $\kappa>0$ we have for some $C \in L^2(\Omega)$ that
     \[
        |K_1 + K_2 + K_3| \leq C \left(|t_2 - t_1|^{\gamma' - \frac\kappa2} + |t_2 - t_1|^{\kappa\beta}\right).
     \]
     Note that since $\beta<\beta_2$ and $\gamma'<\gamma_{max}$, in order to maximize $\gamma' - \frac{\kappa}{2} \wedge \kappa\beta$ we pick $\kappa = \frac{\gamma_{max}}{\frac12 + \beta_2}$, which yields for $\gamma \in \left(0, 1\right)$,
     \[
        |K_1+K_2+K_3| \leq C|t_2 - t_1|^{\gamma \beta_1}.
     \]
\end{proof}

\section{Convergence of the discrete scheme}

\label{sec: convergence}

In this section, we prove tightness of the laws of the approximate system and will employ a modified version of the Skorokhod representation theorem for non-metrizable spaces to obtain a stochastic basis supporting the pre-limiting system and its almost-sure limit in $\Xi^\delta$ for any $\delta\in(0,1)$. We then show that this leads to a solution in the sense of Definition \ref{def: solution}.

Henceforth, for each $n\in\N$ we consider a martingale solution of \eqref{eq: discrete dynamics Rn} in the sense of Proposition \ref{prop: existence for discrete systems} denoted by
\[
    \left((\Omega_n, \mathcal F_n, \mathbb F_n, \P_n), \mathbf u_n, W_n \right)
\]
and we set
\[
    \mu_n := \P_n \circ (u_n, \mathfrak u_n, \eta_n, W_n)^{-1} = \P_n \circ \left( P_n^{-1} (\mathbf u_n), \mathfrak P_n^{-1} (\mathbf u_n), P_n^{-1} (\boldsymbol\eta_n) , W_n \right)^{-1}.
\]

\subsection{Compactness results}

Recall the definition of the space $\Xi^\delta = \Xi_u \times \Xi_{\mathfrak u}^\delta \times \Xi_\eta \times \Xi_W$ from Section \ref{sec: setup and main result}.

\begin{proposition}
    \label{prop: control of eta and tightness}
    There exists a subsequence of $\mu_n$ converging in law in $\Xi^\delta$ for any $\delta\in(0,1)$ to a limiting law $\mu$.
\end{proposition}

\begin{proof}
   From Proposition \ref{prop: estimates space-time holder} we know that the sequence $u_n,\ n\in\N,$ is bounded in $L^2(\Omega, C^{\gamma \beta_1, \beta_2}_0(\mathcal Q_T))$ for $\gamma\in\left(0, 1\right)$. As for any $\nu\in(0,\gamma)$ the space $C^{\gamma\beta_1, \beta_2}_0(\mathcal Q_T)$ is compactly embedded into $C^{\nu \beta_1, \nu \beta_2}_0(\mathcal Q_T)$, the ball $\bar B_R$ in $C^{\gamma\beta_1, \beta_2}_0(\mathcal Q_T)$ is a compact subset of $C^{\nu \beta_1, \nu \beta_2}_0(\mathcal Q_T)$. Moreover, for $R>0$ we have
    \[
        \P_n \left(u_n \in C^{\nu\beta_1, \nu\beta_2}_0(\mathcal Q_T) \setminus \bar B_R\right) \leq \P_n \left(\|u_n\|_{C^{\gamma\beta_1,\beta_2}_0(\mathcal Q_T)} > R \right) \leq R^{-2} \E_n \left[\|u_n\|^2_{C^{\gamma\beta_1, \beta_2}_0(\mathcal Q_T)}\right],
    \]
    so that for any $\varepsilon\in(0,1)$ there is a large enough ball $\bar B_R$, which is a compact set with
    \[
        \P_n \left(u_n \in \bar B_R\right) \geq 1 - \varepsilon.
    \]
    For $\eta_n$, we have that \eqref{eq: estimates for discretized coefficients r} from Lemma \ref{lemma: estimates for discretized coefficients} and Proposition \ref{prop: estimates 1} imply that
    \begin{equation}
        \label{eq: control of moments of eta}
        \begin{aligned}
        &\E_n \left[\|\eta_n\|_{L^2\left([0,T], L^2\right)}^2 \right] = \E_n \left[\int_0^T \left\|\eta_n(r) \right\|_{L^2}^2 \d r \right] = \E_n \left[\int_0^T \left\|P_n^{-1} \circ \1_{\mathbf u_n(r)=0} \odot \mathbf r_n(\mathbf u_n(r)) \right\|_{L^2}^2 \d r \right] \\
        &\qquad\lesssim \E_n \left[ \int_0^T \left\|\mathbf r_n (\mathbf u_n(r)) \right\|^2_{\mathbf H^0_n} \d r \right] \lesssim \E_n \left[ \int_0^T \left( 1 + \left\|\mathbf u_n(r) \right\|^2_{\mathbf H^0_n} \right) \d r \right] < \infty.
        \end{aligned}
    \end{equation}
    Moreover, closed balls in $L^2\left([0,T], L^2\right)$ are compact in the weak topology. Therefore, noting that for $R>0$ we have
    \[
        \P_n \left(\eta_n \in \Xi_\eta \setminus \bar B_R \right) = \P_n \left(\|\eta_n\|_{L^2\left([0,T], L^2\right)} > R\right) \leq  R^{-2} \E_n \left[\|\eta_n\|_{L^2\left([0,T], L^2\right)}^2 \right],
    \]
    it follows that for any $\varepsilon\in(0,1)$ there is a large enough ball $\bar B_R$, which is a compact set in the weak topology with
    \[
        \P_n \left(\eta_n \in \bar B_R\right) \geq 1 - \varepsilon.
    \]
    Since the space $\Xi_W$ is Polish, the law $\P \circ W^{-1}$ is tight as well. As tightness of marginals implies joint tightness of the sequence of the laws of $(\mathfrak u_n, \eta_n, W)$ in the product topology $\Xi_{\mathfrak u}^\delta \times \Xi_{\eta} \times \Xi_W$, it admits a subsequence converging in law $(\mathfrak u, \eta, W) \in \Xi_{\mathfrak u}^\delta \times \Xi_{\eta} \times \Xi_W$. Moreover, noting that $\Xi_u = C(\mathcal Q_T) = C\left([0,T], C^0\right)$, for any $\delta > 0$ we have
    \[
        \begin{aligned}
        &\left\| \mathfrak u_n - u_n \right\|_{C(\mathcal Q_T)} = \sup_{t\in[0,T]} \|\mathfrak u_n(t) - u_n(t)\|_{C^0} \leq \sup_{t\in[0,T]} \max_{i\in\{1,\dots,n+1\}} |\mathbf u_n(t)(i) - \mathbf u_n(t)(i-1)| \\
        &\qquad\leq \sup_{t\in[0,T]} \sup_{0\leq \varepsilon \leq h_n} \sup_{|x-y| < \varepsilon} |\mathfrak u_n(t,x) - \mathfrak u_n(t,y)| \leq \sup_{t\in[0,T]} \left\| \mathfrak u_n(t)\right\|_{C^{\delta \beta_2}} h_n^{\delta\beta_2} \leq h_n^{\delta\beta_2} \|\mathfrak u_n\|_{\Xi_{\mathfrak u}^\delta}.
        \end{aligned}
    \]
    Therefore, for any $\varepsilon > 0$ we have
    \[
        \P_n \left( \left\| \mathfrak u_n - u_n \right\|_{C(\mathcal Q_T)} > \varepsilon \right) \leq \varepsilon^{-2} h_n^{2\delta\beta_2} \sup_{n\in\N} \E_n \left[ \|\mathfrak u_n\|_{\Xi_{\mathfrak u}^\delta}^2 \right] \rightarrow 0.
    \]
    Since convergence of $\mathfrak u_n$ in law in $\Xi_{\mathfrak u}^\delta$ implies convergence in law in $\Xi_u$, it follows by \cite[Corollary 3.3.3]{ethier2009markov} that $u_n \rightarrow \mathfrak u$ in law in $\Xi_u$.

\end{proof}

We now invoke the modified Skorokhod representation theorem by \cite{jakubowski1998almost}.

\begin{proposition}
    \label{prop: invocation of the modified skorohod theorem}
    There exists a probability space $\left(\widetilde \Omega, \widetilde{\mathcal F}, \widetilde\P\right)$ supporting sequences of random variables
    \begin{gather*}
        \widetilde u_n: \widetilde \Omega \rightarrow C\left([0,T], L^2\right), \quad
        \widetilde{\mathfrak u}_n: \widetilde \Omega \rightarrow \bigcap_{\nu\in(0,1)} C^{\nu \beta_1, \nu \beta_2}_0(\mathcal Q_T), \quad
        \widetilde{\eta}_n: \widetilde \Omega \rightarrow L^2\left([0,T], L^2\right),\quad n\in\N,
    \end{gather*}
    as well as random variables
    \begin{gather*}
        \widetilde{u} \in L^2 \left(\widetilde \Omega, \bigcap_{\nu\in(0,1)} C^{\nu \beta_1, \nu \beta_2}_0(\mathcal Q_T) \right), \qquad
        \widetilde \eta \in L^2 \left(\widetilde \Omega, L^2 \left([0,T], L^2\right) \right),
    \end{gather*}
    and $U$--valued processes 
    $\widetilde W, \widetilde W_n: n\in\N$ such that:
    \begin{enumerate}[label=(\roman*)]
        \item the laws $\widetilde\P \circ \left(\widetilde u_n, \widetilde{\mathfrak u}_n, \widetilde \eta_n, \widetilde W_n\right)^{-1}$ and $\widetilde\P \circ \left(\widetilde u, \widetilde{u}, \widetilde \eta, \widetilde W\right)^{-1}$ coincide with $\mu_n$ and $\mu$, respectively.
        \item the sequence $\left(\widetilde u_n, \widetilde{\mathfrak u}_n, \widetilde \eta_n, \widetilde W_n\right)$ converges $\widetilde\P$-a.s. to $\left(\widetilde u, \widetilde u, \widetilde \eta, \widetilde W\right)$ in $\Xi^\delta$ for any $\delta\in(0,1)$.
    \end{enumerate}
\end{proposition}

\begin{proof}
    With the tightness result in Proposition \ref{prop: control of eta and tightness}, the statement follows by invoking the modified Skorokhod representation theorem by \cite{jakubowski1998almost} in order to account for non-metrizability of $\Xi_\eta$.
\end{proof}

\subsection{Convergence and identification of the limit}

Consider an arbitrary function $\varphi \in C^3_c([0,1],\R)$. We define $L_n: V_n \rightarrow V_n$ and $\Sigma_n: V_n \rightarrow V_n$ as
\[
    L_n := P_n^{-1} \circ \mathbf L_n \circ P_n, \qquad \Sigma_n (u_n) := P_n^{-1} \circ \boldsymbol\Sigma_n(\mathbf u_n) = \sum_{k=1}^n \mu_k^2 [B(u_n)I_n g_k]^2.
\]
Note that for $\mathbf v_n, \mathbf w_n \in \R^n$,
\[
    P_n^{-1} \circ (\mathbf v_n \odot \mathbf w_n) = v_n \cdot P_n^{-1} (\mathbf w_n) = v_n \cdot w_n.
\]
Therefore, the discrete system \eqref{eq: discrete dynamics Rn} can be reformulated as a family of processes $\{(u_n, \eta_n): n\in\N\}$ taking values in $V_n\times V_n$ satisfying
\begin{equation}
    \begin{aligned}
    \label{eq: discrete dynamics Vn}
    &\langle I_n \varphi, u_n(t) \rangle = \langle I_n\varphi, u_n(0) \rangle + \int_0^t \left\langle I_n\varphi, -L_n (u_n^\alpha(r))
    \right\rangle \d r \\
    &\qquad+ \sum_{k=1}^n \mu_k \int_0^t \left\langle I_n\varphi, \1_{u_n(r) > 0} \cdot I_n B(u_n(r))(g_k) \right\rangle \d \xi_k(r) + \int_0^t \langle I_n\varphi, \eta_n(r) \rangle \d r.
    \end{aligned}
\end{equation}

We start with the following lemma indicating that the relationship between $u_n$ and $\eta_n$ is preserved for $\widetilde u_n$ and $\widetilde \eta_n$.

\begin{lemma}
    $\widetilde\P$-a.s. we have that
    \[
        \widetilde{\eta}_n = \1_{\widetilde{u}_n = 0} R(u_n).
    \]
\end{lemma}

\begin{proof}
    Let $\psi\in C^\infty(\mathcal Q_T)$. By equality of laws, measurability of $R$ and $u\mapsto \1_{u=0}$, and control of moments of $\eta_n$ in \eqref{eq: control of moments of eta},
    \[
        \begin{aligned}
        &\widetilde \E\left[ \left| \int_0^T \int_0^1 \left(\1_{\widetilde{u}_n(r,x) = 0} R(\widetilde{u}_n(r))(x) - \widetilde\eta_n(r,x) \right) \psi(r,x) \d x \d r \right|\right] \\
        &\qquad=\E_n \left[ \left| \int_0^T \int_0^1 \left(\1_{u_n(r,x) = 0} R(u_n(r))(x) - \eta_n(r,x) \right) \psi(r,x) \d x \d r \right|\right] = 0,
        \end{aligned}
    \]
    which concludes the proof, since $\psi$ is arbitrary.
\end{proof}

Let $\widetilde{\mathbb F} := \left\{\widetilde{\mathcal F}_t\right\}_{t\in[0,T]}$ and $\widetilde{\mathbb F}_n := \left\{\widetilde{\mathcal F}^n_t\right\}_{t\in[0,T]}$ be the $\widetilde\P$-augmented canonical filtrations generated by $(\widetilde u, \widetilde W)$ and $(\widetilde u_n, \widetilde W_n)$, respectively, i.e.
\[
    \begin{aligned}
        \widetilde{\mathcal F}_t &:= \sigma\left(\sigma(\widetilde u|_{[0,t]}, \widetilde W|_{[0,t]} ) \cup \{N \in \widetilde{\mathcal F} : \widetilde\P(N) = 0 \} \right),\\
        \widetilde{\mathcal F}^n_t &:= \sigma\left(\sigma(\widetilde u_n|_{[0,t]}, \widetilde W_n|_{[0,t]} ) \cup \{N \in \widetilde{\mathcal F} : \widetilde\P(N) = 0 \} \right). 
    \end{aligned}
\]

We now want to show that the processes $\widetilde W$ and $\widetilde W_n$ are $Q$--Wiener processes which can be rewritten in terms of independent Brownian motions.

\begin{lemma}
    $\widetilde W$ and $\widetilde W_n$ are $Q$--Wiener processes adapted to the filtrations $\widetilde{\mathbb F}$ and $\widetilde{\mathbb F}_n$, respectively, which can be rewritten as
    \begin{equation}
        \label{eq: decomposition of W}
       \widetilde W(t) = \sum_{k=1}^\infty \mu_k \widetilde\xi_k(t) g_k,  \qquad \widetilde W_n(t) = \sum_{k=1}^\infty \mu_k \widetilde\xi_k^n(t) g_k,
    \end{equation}
    where $\widetilde \xi_k$ and $\widetilde \xi_k^n,\ k\in\N$, are identically distributed independent one-dimensional standard Brownian motions with respect to $\widetilde{\mathbb F}$ and $\widetilde{\mathbb F}_n$, respectively.
\end{lemma}

\begin{proof}

    The $\widetilde W_n$ are $Q$--Wiener processes with the decomposition \eqref{eq: decomposition of W} since
    \[
        \widetilde\P \circ \widetilde W_n^{-1} = \P_n \circ W_n^{-1}.
    \]
    In order to prove that $\widetilde W$ is an $\widetilde{\mathbb F}$-martingale, note that by equality of laws
    \[
        \sup_{n\in\N} \widetilde\E\left[\left\| \widetilde W_n(t) \right\|_{U}^2\right] = \sup_{n\in\N} \E_n \left[\|W_n(t)\|_{U}^2\right] < \infty.
    \]
    Therefore, for $0 \leq s < t \leq T$ and a continuous function $\Psi: \left.\Xi_u\right|_{[0,s]} \times \left.\Xi_W\right|_{[0,s]} \rightarrow [0,1]$, by equality of laws
    \[
        \begin{aligned}
        &\widetilde \E\left[\Psi\left( \left.\widetilde u\right|_{[0,s]}, \left.\widetilde W\right|_{[0,s]} \right) \left( \widetilde W(t) - \widetilde W(s) \right) \right] = \lim_{n\rightarrow\infty} \widetilde \E\left[\Psi\left( \left.\widetilde u_n\right|_{[0,s]}, \left.\widetilde W_n\right|_{[0,s]} \right) \left( \widetilde W_n(t) - \widetilde W_n(s) \right) \right] \\
        &\qquad = \lim_{n\rightarrow\infty} \E_n \left[\Psi\left( \left.u_n\right|_{[0,s]}, \left.W_n\right|_{[0,s]} \right) \left( W_n(t) - W_n(s) \right) \right] = 0.
        \end{aligned}
    \]
    Furthermore, note that for $k,\ell\in\N$ similarly as above it follows that
    \[
        \mu_k^{-1} \mu_\ell^{-1} \left\langle \widetilde W_n(t), g_k \right\rangle_{U} \left\langle \widetilde W_n(t), g_\ell \right\rangle_{U} - \delta_{k\ell} t
    \]
    is an $\widetilde{\mathbb F}^n$-martingale, so that the decomposition \eqref{eq: decomposition of W} follows by convergence in $\Xi_W$ and Levy's characterization of Brownian motion.
\end{proof}

Next, we aim to prove the convergence of the stochastic integral. For $k, n\in\N$ we define the processes
\begin{equation}
    \label{eq: definition discrete martingales}
    \begin{aligned}
    M_{1,n}(t) :=& \left\langle I_n\varphi, u_n(t) - u_n(0) \right\rangle - \int_0^t \left\langle I_n\varphi, -L_n (u_n^\alpha(r))  \right\rangle \d r - \int_0^t \langle I_n\varphi, \eta_n(r) \rangle \d r \\
    =& \sum_{k=1}^n \mu_k \int_0^t \left\langle I_n\varphi, \1_{u_n(r)>0} \cdot I_n B(u_n(r))(g_k) \right\rangle \d \xi_k(r),\\
    M_{2,n}(t) :=& M_{1,n}^2(t) - \int_0^t \left\langle I_n\varphi, \1_{u_n(r) > 0} \cdot \Sigma_n(u_n(r)) \right\rangle \d r, \\
    M_{3,n}(t) :=& \begin{cases}
        M_{1,n}(t) \xi_k(t) - \mu_k \int_0^t \left\langle I_n\varphi, \1_{u_n(r)>0} \cdot I_n B(u_n(r))(g_k) \right\rangle \d r , & \text{for } k \leq n,\\
        M_{1,n}(t) \xi_k(t), & \text{for } k > n,
    \end{cases}
    \end{aligned}
\end{equation}
where we omit the dependence of $M_{3,n}$ on $k\in\N$ for brevity. For each of the above processes $M_{i,n}(t)$ supported on $(\Omega_n, \mathcal F_n, \mathbb F_n, \P_n)$ we also denote their respective counterparts supported on $\left(\widetilde \Omega, \widetilde{\mathcal F}, \widetilde\P\right)$ by $\widetilde M_{i,n}(t)$, with $u_n$, $\eta_n$, $\xi_k$ replaced by $\widetilde u_n$, $\widetilde \eta_n$, $\widetilde \xi_k^n$, respectively. Similarly, we define their continuous counterparts:
\begin{equation}
    \label{eq: definition continuous martingales}
    \begin{aligned}
        \widetilde M_{1}(t) &:= \langle \varphi, \widetilde u(t) \rangle - \langle \varphi, \widetilde u(0) \rangle - \int_0^t \left\langle \partial_x^2 \varphi, \widetilde u^\alpha(r) \right\rangle \d r- \int_0^t \langle \varphi, \widetilde \eta(r) \rangle \d r, \\
        \widetilde M_{2}(t) &:= \widetilde M_{1}^2(t) - \int_0^t \left\langle \varphi, \1_{\widetilde u(r)>0} \Sigma(\widetilde u(r)) \right\rangle \d r, \\
        \widetilde M_{3}(t) &:= \widetilde M_{1}(t) \widetilde \xi_k(t) - \mu_k \int_0^t \langle \varphi, \1_{\widetilde u(r)>0} B(\widetilde u(r))(g_k) \rangle \d r, \\
    \end{aligned}
\end{equation}
where we again omit the dependence of $\widetilde M_{3}$ on $k\in\N$ for brevity.

Let us fix a continuous function $\Psi: \left.\Xi_u\right|_{[0,s]} \times \left.\Xi_W\right|_{[0,s]} \rightarrow [0,1]$. We note that the processes defined in \eqref{eq: definition discrete martingales} are $\mathbb F_n$-martingales, so that by equality of laws the processes $\widetilde M_{i,n}(t)$ are $\widetilde{\mathbb F}_n$-martingales and for $t>s$,
\begin{equation}
    \label{eq: martingality of the discrete processes}
    \widetilde\E\left[\Psi \left(\widetilde u_n|_{[0,s]}, \widetilde W_n|_{[0,s]} \right) \left(\widetilde M_{i,n}(t) - \widetilde M_{i,n}(s) \right)\right] = 0.
\end{equation}
Now in a series of lemmata we intend to pass to the limit in \eqref{eq: martingality of the discrete processes} for $i\in\{1,2,3\}$.

\begin{lemma}
    \label{lemma: M1}
    For $0 \leq s < t \leq T$,
    \[
        \widetilde\E\left[\Psi\left(\widetilde u|_{[0,s]}, \widetilde W|_{[0,s]}\right) \left(\widetilde M_{1}(t) - \widetilde M_{1}(s) \right) \right] = 0.
    \]
\end{lemma}

\begin{proof}
    By discrete partial integration we have
    \[
        \left\langle I_n\varphi, L_n (\widetilde u_n^\alpha) \right\rangle = \left\langle L_n (I_n\varphi), \widetilde u_n^\alpha \right\rangle.
    \]
    Moreover, since $\varphi \in C^3_c([0,1],\R)$, we have
    \[
        -L_n (I_n\varphi) \rightarrow \partial_x^2 \varphi \quad \text{ in } L^2.
    \]
    Furthermore, convergence of $\widetilde u_n \rightarrow \widetilde u$ in $\Xi_u$ implies convergence of $\widetilde u_n^\alpha \rightarrow \widetilde u^\alpha$ in $\Xi_u$. Now combining this with the fact that $\widetilde \eta_n$ converges $\widetilde\P$-a.s. in $\Xi_\eta$, we have $\widetilde\P$-a.s.
    \begin{equation}
        \label{eq: Pas convergence of deterministic terms}
        \begin{aligned}
        \langle I_n\varphi, \widetilde u_n(t) - \widetilde u_n(s) \rangle &\rightarrow \langle \varphi, \widetilde u(t) - \widetilde u(s) \rangle,\\
        \int_s^t \langle I_n\varphi, \widetilde \eta_n(r) \rangle \d r &\rightarrow \int_s^t \langle \varphi, \widetilde \eta(r) \rangle \d r , \\
        \int_s^t \left\langle I_n\varphi, -L_n (\widetilde u_n^\alpha(r)) \right\rangle \d r &\rightarrow \int_s^t \left\langle \partial_x^2 \varphi, \widetilde u^\alpha(r) \right\rangle \d r.
        \end{aligned}
    \end{equation}
    Moreover, by continuity of $\Psi$, convergence of $\widetilde u_n$ in $\Xi_u$, and convergence of $\widetilde W_n$ in $\Xi_W$ we have $\widetilde\P$-a.s.
    \begin{equation}
        \label{eq: convergence of psi}
        \Psi \left(\widetilde u_n|_{[0,s]}, \widetilde W_n|_{[0,s]}\right) \rightarrow \Psi \left(\widetilde u|_{[0,s]}, \widetilde W|_{[0,s]}\right).
    \end{equation}
    Recalling \eqref{eq: martingality of the discrete processes} with $i=1$ we have that
    \begin{equation}
        \label{eq: Mn - Mn expectation zero}
        \begin{aligned}
        &\widetilde\E\left[\Psi \left( \widetilde u_n|_{[0,s]}, \widetilde W_n|_{[0,s]} \right) \left(\widetilde M_{1,n}(t) - \widetilde M_{1,n}(s) \right)\right] = 0.
        \end{aligned}
    \end{equation}
    Note that for each term on the right-hand side of \eqref{eq: Pas convergence of deterministic terms} we have uniform bounds of some positive moments due to Proposition \ref{prop: estimates 1}, Proposition \ref{prop: estimates 2}, and the growth condition \eqref{eq: sojourn coefficient growth condition}. This enables us to pass to the limit in \eqref{eq: Mn - Mn expectation zero} by dominated convergence, which yields the statement.
\end{proof}

In order to address $\widetilde M_{2}$, we first prove the following lemma concerning the difference between the discrete variance coefficients and the limiting variance coefficient.

\begin{lemma}
    \label{eq: convergence difference sigma and sigma n}
    If $u_n \rightarrow u$ in $C([0,1], \R)$, then
    \[
        \left\| \Sigma(u) - \Sigma_n(u_n) \right\|_{L^1} \rightarrow 0.
    \]
\end{lemma}

\begin{proof}
    We can write
    \[
        \begin{aligned}
            &\left\| \Sigma(u) - \Sigma_n(u_n) \right\|_{L^1} \leq \sum_{i=1}^3 R_{i,n},
        \end{aligned}
    \]
    where
    \[
        \begin{aligned}
            R_{1,n} &:= \sum_{k=1}^n \mu_k^2 \left\| B(u)(g_k)^2 - I_n B(u)(g_k)^2 \right\|_{L^1},\\
            R_{2,n} &:= \sum_{k=1}^n \mu_k^2 \left\| I_n B(u_n)(g_k)^2 - I_n B(u)(g_k)^2 \right\|_{L^1}, \\
            R_{3,n} &:= \sum_{k=n+1}^\infty \mu_k^2 \left\| B(u)(g_k) \right\|_{L^2}^2.
        \end{aligned}
    \]
    
    \emph{Estimating $R_{1,n}$.} Note that for every $k\in\N$, by Lebesgue's differentiation theorem, that
    \[
        B(u)(g_k) - I_n B(u)(g_k) \rightarrow 0, \qquad \mbox{$\d x$-a.e.}
    \]
    Moreover, note that
    \[
        |B(u)(g_k) - I_n B(u)(g_k)|^2 \leq B(u)(g_k)^2 + I_n B(u)(g_k)^2
    \]
    and due to \eqref{eq: In(u) leq u} we have
    \[
        \sum_{k=1}^\infty \mu_k^2 \left\| B(u)(g_k)^2 + I_n B(u)(g_k)^2 \right\|_{L^1} \lesssim \sum_{k=1}^\infty \mu_k^2 \| B(u)(g_k)\|_{L^2}^2 = \|B(u)\|_{L_2(U_0, L^2)}^2,
    \]
    which implies by dominated convergence that
    \[
        \sum_{k=1}^\infty \mu_k^2 \left\| B(u)(g_k) - I_n B(u)(g_k) \right\|_{L^2}^2 \rightarrow 0.
    \]
    Together, these results imply that
    \[
        \begin{aligned}
            &\sum_{k=1}^\infty \mu_k^2 \left\| B(u)(g_k)^2 - I_n B(u)(g_k)^2 \right\|_{L^1} \\
            &\qquad \lesssim \sum_{k=1}^\infty \mu_k^2 \left\| B(u)(g_k) - I_n B(u)(g_k) \right\|_{L^2} \left\|B(u)(g_k) + I_n B(u)(g_k) \right\|_{L^2} \\
            &\qquad \lesssim \|B(u)\|_{L_2(U_0, L^2)} \left(\sum_{k=1}^\infty \mu_k^2 \| B(u)(g_k) - I_n B(u)(g_k) \|_{L^2}^2 \right) \rightarrow 0.\\
        \end{aligned}
    \]
    
    \emph{Estimating $R_{2,n}$.} Since $b\in C^1_b$, we have for $u_n \rightarrow u$ in $C([0,1], \R)$ due to \eqref{eq: noise coloring} that
    \[
        \begin{aligned}
            \|B(u_n) - B(u)\|_{L_2(U_0, L^2)} &= \sum_{k=1}^\infty \mu_k^2 \left\|B(u_n)(g_k) - B(u)(g_k)\right\|_{L^2} \lesssim \| u_n - u \|_{L^\infty} \sum_{k=1}^\infty \mu_k^2 \|g_k\|_{L^2} \rightarrow 0,\\
            \|B(u)\|_{L_2(U_0, L^2)} &= \sum_{k=1}^\infty \mu_k^2 \left\|B(u)(g_k)\right\|_{L^2} \lesssim \|u\|_{L^\infty} \sum_{k=1}^\infty \mu_k^2 \|g_k\|_{L^2} \lesssim \|u\|_{L^\infty}.
        \end{aligned}
    \]
    Therefore, we obtain with \eqref{eq: In(u) leq u} that
    \[
        \begin{aligned}
            R_{2,n} &\leq \left( \sum_{k=1}^\infty \mu_k^2 \left\| I_n \left[B(u_n)(g_k) - B(u)(g_k) \right] \right\|_{L^2}^2 \right)^{1/2} \left(\sum_{k=1}^\infty \mu_k^2 \left( \|B(u_n)(g_k)\|_{L^2}^2 + \|B(u)(g_k)\|_{L^2}^2 \right) \right)^{1/2} \\
            &\leq \|B(u_n) - B(u)\|_{L_2(U_0, L^2)} \left( \|B(u_n)\|_{L_2(U_0, L^2)} + \|B(u)\|_{L_2(U_0, L^2)} \right) \rightarrow 0.
        \end{aligned}
    \]
    
    \emph{Estimating $R_{3,n}$.} Since $B(u) \in L_2(U_0, L^2)$, it follows directly that $R_{3,n} \rightarrow 0$.

\end{proof}

We now prove the following lemma addressing $\widetilde\P$-a.s.~convergence of the diffusion coefficients containing the indicator function. As noted before, this lemma relies on the occupation time formula for semimartingales applied to the pre-limiting system.

\begin{lemma}
    \label{lemma: convergence of indicator times sigma}
    For $\psi \in L^2(\mathcal Q_T)$, we have that $\widetilde\P$-a.s.
    \[
        \begin{aligned}
            &\lim_{n\rightarrow\infty} \int_0^T \left\langle I_n \psi(r), \1_{\widetilde u_n(r) > 0} \Sigma_n (\widetilde u_n(r)) \right\rangle \d r = \int_0^T \left\langle \psi(r), \1_{\widetilde u(r) > 0} \Sigma(\widetilde u(r)) \right\rangle\d r.
        \end{aligned}
    \]
\end{lemma}

\begin{proof}

    We first prove the statement for $\psi(t,x) = \phi(x) \1_{[a, b]}(t)$ for some $\phi \in L^\infty([0,1])$ and $0 \leq a < b \leq T$. Note that for any $\varepsilon > 0$, $u_n \in V_n$, and $u\in L^2$ we can write
    \begin{equation}
        \label{eq: equation 1 for stochastic proof}
        \begin{aligned}
        &\left\langle I_n \phi, \1_{u_n > 0} \Sigma_n(u_n) \right\rangle =
        \left\langle \1_{u<\varepsilon} \cdot I_n \phi, \1_{u_n > 0} \Sigma_n (u_n) \right\rangle - \left\langle \1_{u\geq\varepsilon} \cdot I_n \phi, \1_{u_n = 0} \Sigma_n (u_n) \right\rangle + \left\langle \1_{u\geq\varepsilon} \cdot I_n \phi, \Sigma_n(u_n) \right\rangle .
        \end{aligned}
    \end{equation}
    Therefore, we can rewrite
    \[
        \begin{aligned}
            \left| \int_a^b \left\langle I_n\phi, \1_{\widetilde u_n(r) > 0} \Sigma_n(\widetilde u_n(r)) \right\rangle \d r - \int_a^b \left\langle \phi, \1_{\widetilde u(r) > 0} \Sigma(\widetilde u(r)) \right\rangle\d r \right| \leq \sum_{i=1}^3 |R_{i,\varepsilon,n}|
        \end{aligned}
    \]
    with
    \[
        \begin{aligned}
            R_{1,\varepsilon,n} &:= \int_a^b \left\langle \1_{\widetilde u(r)<\varepsilon} \cdot I_n \phi, \1_{\widetilde u_n(r) > 0} \Sigma_n (\widetilde u_n(r)) \right\rangle\d r, \\
            R_{2,\varepsilon,n} &:= - \int_a^b \left\langle \1_{\widetilde u(r)\geq\varepsilon} \cdot I_n \phi, \1_{\widetilde u_n(r) = 0} \Sigma_n (\widetilde u_n(r)) \right\rangle \d r, \\
            R_{3,\varepsilon,n} &:= \int_a^b \left\langle \1_{\widetilde u(r) \geq\varepsilon} \cdot I_n \phi, \Sigma_n (\widetilde u_n(r)) \right\rangle \d r - \int_a^b \left\langle \phi, \1_{\widetilde u(r) > 0} \Sigma(\widetilde u(r)) \right\rangle\d r.
        \end{aligned}
    \]
    It thus suffices to prove that every residual term is vanishing $\widetilde\P$-a.s. 
    
    \emph{Estimating $R_{1,\varepsilon,n}$}. By the occupation time formula for semimartingales \cite[Corollary VI.1.6]{RevuzContinuousMartingalesBrownian1999} it follows for $i\in\{1, \dots, n\}$ that $\widetilde\P$-a.s.
    \[
        \lim_{\varepsilon\downarrow 0} \int_a^b \1_{\widetilde{\mathbf u}_n(r)(i) < \varepsilon} \d \llangle \widetilde{\mathbf u}_n(i) \rrangle_r = 0.
    \]
    Therefore, using that $x \mapsto \1_{x < \varepsilon}$ is a lower semicontinuous function and that $\widetilde u_n$ converges $\widetilde\P$-a.s.~in $\Xi_u$, we obtain $\widetilde\P$-a.s.
    \begin{equation}
        \label{eq: equation 3 for stochastic proof}
        \begin{aligned}
            \lim_{n\rightarrow 0} \lim_{\varepsilon\downarrow 0} |R_{1,\varepsilon,n}| &\leq \lim_{n\rightarrow 0} \lim_{\varepsilon\downarrow 0} \left| \int_a^b \left\langle \1_{\widetilde u(r)<\varepsilon} \cdot I_n |\phi|, P_n^{-1} \circ \1_{\widetilde{\mathbf u}_n(r) > 0} \odot \boldsymbol\Sigma_n (\widetilde{\mathbf u}_n(r)) \right\rangle\d r \right| \\
            &\leq \lim_{n\rightarrow 0} \lim_{\varepsilon\downarrow 0} \int_a^b \left\langle \1_{\widetilde u_n(r)<\varepsilon} \cdot I_n |\phi|, P_n^{-1} \circ \1_{\widetilde{\mathbf u}_n(r) > 0} \odot \boldsymbol\Sigma_n (\widetilde{\mathbf u}_n(r)) \right\rangle \d r \\
            &= \lim_{n\rightarrow 0} \lim_{\varepsilon\downarrow 0} \sum_{i=1}^{n} h_n^{-1/2} \langle |\phi|, e_{i,n} \rangle \cdot h_n \int_a^b \1_{\widetilde{\mathbf u}_n(r)(i)<\varepsilon} \1_{\widetilde{\mathbf u}_n(r)(i)>0 } \boldsymbol\Sigma_{n}(\widetilde{\mathbf u}_n(r))(i) \d r \\
            &\leq \left\|\phi\right\|_{L^\infty} \lim_{n\rightarrow 0} \sum_{i=1}^n h_n \lim_{\varepsilon\downarrow 0} \int_a^b \1_{\widetilde{\mathbf u}_n(r)(i) < \varepsilon} \d \llangle \widetilde{\mathbf u}_n(i) \rrangle_r = 0.
        \end{aligned}
    \end{equation}

    \emph{Estimating $R_{2,\varepsilon,n}$}. Note that
    \[
        \begin{aligned}
            \lim_{\varepsilon\downarrow 0} |R_{2,\varepsilon,n}| &= \lim_{\varepsilon\downarrow 0} \left| \int_a^b \left\langle \1_{\widetilde u(r)\geq\varepsilon} \cdot I_n \phi, \1_{\widetilde u_n(r) = 0} \odot \Sigma_n (\widetilde u_n(r)) \right\rangle \d r \right| \\
            &= \lim_{\varepsilon\downarrow 0} \left| \int_a^b \left\langle \1_{\widetilde u(r)\geq\varepsilon} \1_{u_n(r) = 0} \cdot I_n \phi, \Sigma_n (\widetilde u_n(r)) \right\rangle \d r \right| \\
            &\lesssim \left\|\phi\right\|_{L^\infty}  \lim_{\varepsilon\downarrow 0} \int_a^b \left\|\1_{\widetilde u_n(r) = 0} \1_{\widetilde u(r)\geq\varepsilon} \right\|_{L^2}  \left\| \Sigma_n (\widetilde u_n(r)) \right\|_{L^2} \d r  \\
            &\lesssim \left\|\phi\right\|_{L^\infty} \sup_{r\in[0,T]} \left\| \Sigma_n( \widetilde u_n(r) ) \right\|_{L^2}  \lim_{\varepsilon\downarrow 0} \int_a^b \left\| \1_{\widetilde u_n(r) = 0} \1_{\widetilde u(r)\geq\varepsilon} \right\|_{L^2} \d r.
        \end{aligned}
    \]
    Note that $\1_{x \geq \varepsilon} \rightarrow \1_{x > 0}$ as $\varepsilon \downarrow 0$, so that for all $(t,x)\in \mathcal Q_T$,
    \[
        \1_{\widetilde u_n(t) = 0} (x)\1_{\widetilde u(t) \geq \varepsilon}(x) \stackrel{\varepsilon\rightarrow 0}{\longrightarrow} \1_{\widetilde u_n(t) = 0}(x) \1_{\widetilde u(t) > 0}(x).
    \]
    Moreover, since $\widetilde u_n \rightarrow \widetilde u$ in $\Xi_u$ $\widetilde\P$-a.s., we have $\widetilde\P$-a.s.~for all $(t,x)\in \mathcal Q_T$,
    \[
        \1_{\widetilde u_n(t) = 0}(x) \1_{\widetilde u(t) > 0}(x) \stackrel{n\rightarrow\infty}{\longrightarrow} 0.
    \]
    Combining these results gives
    \[
        \lim_{n\rightarrow\infty} \lim_{\varepsilon\downarrow 0} \int_a^b \left\| \1_{\widetilde u_n(r) = 0} \1_{\widetilde u(r)\geq\varepsilon} \right\|_{L^2} \d r = \lim_{n\rightarrow\infty} \int_a^b \left\| \1_{\widetilde u_n(r) = 0} \1_{\widetilde u(r) > 0} \right\|_{L^2} \d r = 0.
    \]
    Moreover, note that due to \eqref{eq: In(u) leq u},
    \begin{equation}
        \label{eq: uniform estimate for Pn Sigma(u)}
        \left\|\Sigma_n(\widetilde u_n) \right\|_{L^2} = \left\|I_n \sum_{k=1}^n [B(\widetilde u_n)(\mu_k g_k)]^2 \right\|_{L^2} \leq \left\|\sum_{k=1}^\infty [B(\widetilde u_n)(\mu_k g_k)]^2 \right\|_{L^2} = \left\| \Sigma(\widetilde u_n) \right\|_{L^2}.
    \end{equation}
    Thus by \eqref{eq: estimates for discretized coefficients Sigma} from Lemma \ref{lemma: estimates for discretized coefficients} and $\widetilde\P$-a.s.~convergence of $\widetilde u_n \rightarrow \widetilde u$ in $\Xi_u$ we have
    \begin{equation}
        \label{eq: control of pn-1 Sigma}
        \limsup_{n\rightarrow\infty} \sup_{r\in[0,T]} \left\|\Sigma_n (\widetilde u_n (r)) \right\|_{L^2} \lesssim \limsup_{n\rightarrow\infty} \left(1 + \sup_{r\in[0,T]} \left\| \widetilde u_n^2(r) \right\|_{L^2} \right) \lesssim 1 + \sup_{n\in\N} \left\| \widetilde u_n \right\|_{\Xi_u}^2 < \infty.
    \end{equation}
    Combining the above estimates we obtain $\widetilde\P$-a.s.
    \[
        \lim_{n\rightarrow 0} \lim_{\varepsilon\downarrow 0} |R_{2,\varepsilon,n}| = 0.
    \]
    
    \emph{Estimating $R_{3,\varepsilon,n}$.} Again noting that $\1_{x \geq \varepsilon} \rightarrow \1_{x > 0}$ as $\varepsilon \downarrow 0$, we obtain that $\widetilde\P$-a.s.
    \begin{equation}
        \label{eq: equation 2 for stochastic proof}
        \begin{aligned}
            & \lim_{\varepsilon\downarrow 0} \int_a^b \left\langle \1_{\widetilde u(r) \geq\varepsilon} \cdot I_n \phi, \Sigma_n (\widetilde u_n(r)) \right\rangle \d r = \int_a^b \left\langle \1_{\widetilde u(r) > 0} \cdot I_n \phi, \Sigma_n (\widetilde u_n(r)) \right\rangle \d r.
        \end{aligned}
    \end{equation}
    Moreover, due to \eqref{eq: In(u) - u -> 0} and \eqref{eq: control of pn-1 Sigma} we have $\widetilde\P$-a.s.
    \[
        \left| \int_a^b \left\langle \1_{\widetilde u(r) > 0} \cdot (I_n \phi - \phi), \Sigma_n (\widetilde u_n(r)) \right\rangle \d r \right| \lesssim \sup_{r\in[0,T]} \left\|\Sigma_n (\widetilde u_n (r)) \right\|_{L^2} \|I_n \phi - \phi\|_{L^2} \rightarrow 0.
    \]
    Furthermore, due to the regularity of $\Sigma$ it holds $\widetilde\P$-a.s.
    \[
        \int_a^b \left\langle \1_{\widetilde u(r) > 0} \phi, \Sigma_n (\widetilde u_n(r)) - \Sigma(\widetilde{u}(r)) \right\rangle \d r \leq \left\| \phi \right\|_{L^\infty} \int_a^b \left\| \Sigma_n(\widetilde{u}_n(r)) - \Sigma(\widetilde{u}(r)) \right\|_{L^1} \d r\rightarrow 0.
    \]
    Altogether these results imply that $\widetilde\P$-a.s.
    \[
        \lim_{n\rightarrow\infty} |R_{3,\varepsilon,n}| = 0,
    \]
    so that the statement follows for $\psi(t,x) = \phi(x)\1_{[a,b]}(t)$. To deduce that the statement holds for all $\psi \in L^2(\mathcal Q_T)$, we consider an approximating sequence
    \[
        \psi_k(t,x) := \sum_{i=1}^k \varphi_i(x) \1_{[a_i, b_i]}(t)\qquad \text{such that} \qquad \|\psi_k - \psi\|_{L^2(\mathcal Q_T)} \rightarrow 0.
    \]
    By dominated convergence $\widetilde\P$-a.s. we have that
    \[
        \begin{aligned}
            &\lim_{n\rightarrow\infty} \int_0^T \left\langle \psi(r), \1_{\widetilde u_n(r) > 0} \Sigma_n(\widetilde u_n(r)) \right\rangle \d r \\
            &\qquad= \lim_{n\rightarrow\infty} \int_0^T \left\langle \sum_{i=1}^\infty I_n \1_{[a_i, b_i]}(r) \varphi_i, \1_{\widetilde u_n(r) > 0} \Sigma_n(\widetilde u_n(r)) \right\rangle \d r \\
            &\qquad= \sum_{i=1}^\infty \lim_{n\rightarrow\infty} \int_0^T \1_{[a_i, b_i]}(r) \left\langle I_n\varphi_i, \1_{\widetilde u_n(r) > 0} \Sigma_n(\widetilde u_n(r)) \right\rangle \d r \\
            &\qquad= \sum_{i=1}^\infty \int_0^T \1_{[a_i, b_i]}(r) \left\langle \varphi_i, \1_{\widetilde u(r) > 0} \Sigma(\widetilde u(r)) \right\rangle \d r  = \int_0^T \left\langle \psi(r),  \1_{\widetilde u(r) > 0} \Sigma(\widetilde u(r)) \right\rangle.
        \end{aligned}
    \]
\end{proof}

\begin{lemma}
    \label{lemma: M2}
    For $0 \leq s < t \leq T$ it holds
    \[
        \widetilde\E\left[\Psi\left(\widetilde u|_{[0,s]}, \widetilde W|_{[0,s]}\right) \left(\widetilde M_{2}(t) - \widetilde M_{2}(s) \right)\right] = 0.
    \]
\end{lemma}

\begin{proof}
    
    Using \eqref{eq: estimates for discretized coefficients Sigma} from Lemma \ref{lemma: estimates for discretized coefficients}, Proposition \ref{prop: estimates 1}, and Lemma \ref{lemma: convergence of indicator times sigma}, we obtain by dominated convergence
    \begin{equation}
        \label{eq: convergence with O epsilon 2}
        \begin{aligned}
            &\lim_{n\rightarrow\infty} \widetilde\E\left[ \int_0^T \left\langle I_n\varphi, \1_{\widetilde u_n(r) > 0} \Sigma_n(u_n(r)) \right\rangle \d r \right] = \widetilde\E\left[\int_0^T \left\langle \varphi, \1_{\widetilde u(r) > 0} \Sigma(\widetilde u(r)) \right\rangle\d r \right].
        \end{aligned}
    \end{equation}
    We obtain, from \eqref{eq: convergence of psi} and \eqref{eq: convergence with O epsilon 2}, that
    \begin{equation}
        \label{eq: ultimate convergence of sigma}
        \begin{aligned}
            &\lim_{n\rightarrow\infty} \widetilde\E\left[\Psi\left(\widetilde u_n|_{[0,s]}, \widetilde W_n|_{[0,s]} \right) \left( \int_0^T \left\langle I_n\varphi, \1_{\widetilde u_n(r) > 0}\Sigma_n(\widetilde u_n(r)) \right\rangle \d r \right) \right] \\
            &\qquad= \widetilde\E\left[\Psi\left(\widetilde u|_{[0,s]}, \widetilde W|_{[0,s]} \right) \left( \int_0^T \left\langle \varphi, \1_{\widetilde u(r) > 0} \Sigma(\widetilde u(r)) \right\rangle\d r \right) \right].
        \end{aligned}
    \end{equation}
    We also have
    \[
        \begin{aligned}
        \lim_{n\rightarrow\infty} \widetilde\E\left[\Psi(\widetilde u_n|_{[0,s]}, \widetilde W_n|_{[0,s]}) \widetilde M_{1,n}^2(t)\right] &= \widetilde\E\left[\Psi(\widetilde u|_{[0,s]}, \widetilde W|_{[0,s]}) \widetilde M_{1}^2(t)\right], \\
        \lim_{n\rightarrow\infty} \widetilde\E\left[\Psi(\widetilde u_n|_{[0,s]}, \widetilde W_n|_{[0,s]}) \widetilde M_{1,n}^2(s)\right] &= \widetilde\E\left[\Psi(\widetilde u|_{[0,s]}, \widetilde W|_{[0,s]}) \widetilde M_{1}^2(s)\right]. \\
        \end{aligned}
    \]
    By definition for any $n\in\N$ the processes $M_{2,n}$ are continuous $\mathbb F$-martingales. Therefore, by equality of the laws we have
    \[
        \widetilde\E\left[\Psi(\widetilde u_n|_{[0,s]}, \widetilde W_n|_{[0,s]}) \left(\widetilde M_{2,n}(t) - \widetilde M_{2,n}(s) \right)\right] = \E_n \left[\Psi(u_n|_{[0,s]}, W|_{[0,s]})(M_{2,n}(t) - M_{2,n}(s))\right] = 0.
    \]
    Thus it follows that
    \[
        \widetilde\E\left[\Psi(\widetilde u|_{[0,s]}, \widetilde W|_{[0,s]}) \left( \widetilde M_{2}(t) - \widetilde M_{2}(s) \right) \right] = 0.
    \]

\end{proof}

\begin{lemma}
    \label{lemma: M3}
    For $0 \leq s < t \leq T$,
    \[
        \widetilde\E\left[\Psi\left(\widetilde u|_{[0,s]}, \widetilde W|_{[0,s]}\right) \left(\widetilde M_{3}(t) - \widetilde M_{3}(s) \right)\right] = 0.
    \]
\end{lemma}

\begin{proof}
    The proof follows in the same way as the previous lemma by considering $\llangle \widetilde {\mathbf u}_n, \xi_k \rrangle_t$ instead of $\llangle \widetilde {\mathbf u}_n \rrangle_t$.
\end{proof}

We can now turn to the proof of the main theorem.

\begin{proof}[Proof of Theorem \ref{thm: main result}]

    \emph{Convergence of stopping times.} Consider
    \[
        \widetilde \tau_n := T \wedge \inf\{t\in[0,T]: \mathscr E_n(\widetilde{\mathbf u}_n(t)) \geq h_n^{-\kappa} \}.
    \]
    From the equality of laws, Markov's inequality, and Proposition \ref{prop: estimates 1} we deduce for every $t\in[0,T]$ that
    \[
        \widetilde\P \left(\widetilde \tau_n < t\right) = \P_n  \left( \tau_n < t\right) = \P_n  \left(\sup_{r\in[0,t]} \mathscr E_n(\mathbf u_n(r)) > h_n^{-\kappa} \right) \leq h_n^\kappa \E_n \left[\sup_{r\in[0,t]} \left\|\mathbf u_n^{(\alpha+1)/2}(r)\right\|_{\mathbf H^0_n}^2 \right] \lesssim h_n^\kappa.
    \]
    It follows that $\widetilde \tau_n \rightarrow T$ in probability and hence $\widetilde\P$-a.s.~along a subsequence.

    \emph{Non-negativity.}  Since the approximate system $\widetilde u_n$ is non-negative $\widetilde\P$-a.s.~and convergent to $\widetilde{u}$ in $\Xi_u$, it follows directly that $\widetilde u\geq 0$ $\widetilde\P \otimes \d t \otimes \d x$-a.e. 

    \emph{Identification of $\eta$.} $\widetilde\P$-a.s. we have that
    \[
        \begin{aligned}
        &\int_0^T \left\langle \varphi, \eta(r) \right\rangle \d r = \lim_{n\rightarrow\infty} \int_0^T \left\langle I_n\varphi, \1_{u_n(r) = 0} R(u_n(r)) \right\rangle \d r \\
        &\qquad= \lim_{n\rightarrow\infty} \int_0^T \left\langle I_n\varphi, R(u_n(r)) \right\rangle \d r - \lim_{n\rightarrow\infty} \int_0^T \left\langle I_n\varphi, \1_{u_n(r) > 0} R(u_n(r)) \right\rangle \d r \\
        &\qquad=: \lim_{n\rightarrow\infty} K_{1,n} + \lim_{n\rightarrow\infty} K_{2,n}.
        \end{aligned}
    \]
    
    \emph{Estimating $K_{1,n}$}. Note that the growth condition \eqref{eq: sojourn coefficient growth condition} implies that
    \[
        \|R(u_n)\|_{L^2} \lesssim 1 + \|u_n\|_{L^2}.
    \]
    Therefore, from the continuity of $R$ combined with \eqref{eq: In(u) - u -> 0} we obtain
    \[
        \begin{aligned}
            &\left| \int_0^T \left\langle I_n\varphi, R(u_n(r)) \right\rangle \d r - \int_0^T \left\langle \varphi, R(u(r)) \right\rangle \d r \right|
            \\
            &\qquad\leq \|\varphi\|_{L^2} \int_0^T \left\| R(u_n(r)) - R(u(r)) \right\|_{L^2} \d r + \sup_{r\in[0,T]} \left\| R(u_n(r)) \right\|_{L^2} \left\| I_n\varphi - \varphi \right\|_{L^2} \rightarrow 0.
        \end{aligned}
    \]

    \emph{Estimating $K_{2,n}$.} We fix some $\varepsilon > 0$. On the one hand, because of the growth condition on $R$ in \eqref{eq: sojourn coefficient growth condition} there exists $\delta > 0$ such that
    \begin{equation}
        \label{eq: sojourn proof e/2 first}
        \begin{aligned}
            &\left| \int_0^T \left\langle \varphi, 
            \left( \1_{\widetilde u_n(r) > 0} - \1_{\widetilde u(r) > 0} \right) \1_{0 < \Sigma(\widetilde u(r)) < \delta} R(\widetilde u(r)) \right\rangle \d r \right|\\
            &\qquad \lesssim \|\varphi\|_{L^\infty} \left( 1 + \int_0^T \| \widetilde u(r) \|_{L^2}^2 \d r \right)^{1/2} \left(\iint_{\mathcal Q_T} \1_{(0, \delta)}(\Sigma(\widetilde u(r))) \d x \d r \right)^{1/2} < \frac{\varepsilon}{2}.
        \end{aligned}
    \end{equation}
    On the other hand, according to Lemma \ref{lemma: convergence of indicator times sigma}, we have for $\psi\in L^2(\mathcal Q_T)$ $\widetilde\P$-a.s.
    \[
        \begin{aligned}
            \int_0^T \left\langle \psi(r), \1_{\widetilde u_n(r) > 0}\Sigma_n (\widetilde u_n(r)) - \1_{\widetilde u(r) > 0} \Sigma(\widetilde u(r)) \right\rangle \d r \rightarrow 0. \\
        \end{aligned}
    \]
    Moreover, by Lemma \ref{eq: convergence difference sigma and sigma n}, $\widetilde\P$-a.s. we have that
    \[
        \begin{aligned}
            \int_0^T \left\langle \psi(r), \left( \Sigma(\widetilde u(r)) - \Sigma_n(\widetilde u_n(r)) \right) \1_{\widetilde u_n(r) > 0} \right\rangle \d r \rightarrow 0. \\
        \end{aligned}
    \]
    Thus we get
    \[
        \int_0^T \left\langle \psi(r), \left(\1_{\widetilde u_n(r) > 0} - \1_{\widetilde u(r) > 0} \right) \Sigma(\widetilde u(r)) \right\rangle \d r \rightarrow 0.
    \]
    Therefore, with the choice $\psi(r) = \varphi R(\widetilde u(r)) \Sigma(\widetilde u(r))^{-1} \1_{\Sigma(\widetilde u(r)) \geq \delta} \in L^\infty(\mathcal Q_T)$ there exists a random variable $N$ taking values in $\N$ such that for all $n\geq N$
    \begin{equation}
        \label{eq: sojourn proof e/2 second}
        \begin{aligned}
            \left|\int_0^T \left\langle \varphi, \left( \1_{\widetilde u_n(r) > 0} - \1_{\widetilde u(r) > 0} \right) \1_{\Sigma(\widetilde u(r)) \geq \delta} R(\widetilde u(r)) \right\rangle \d r\right| < \frac{\varepsilon}{2}.
        \end{aligned}
    \end{equation}
    Now, using that $R(\widetilde u) = \1_{\Sigma(\widetilde u) > 0} R(\widetilde u)$ due to the support condition in \eqref{eq: sojourn coefficient support condition}, the continuity of $R$, and by combining \eqref{eq: sojourn proof e/2 first} and \eqref{eq: sojourn proof e/2 second} we get $\widetilde\P$-a.s.
    \begin{equation}
        \label{eq: identification of eta second term}
        \begin{aligned}
            &\lim_{n\rightarrow\infty} \int_0^T \left\langle I_n\varphi, \1_{\widetilde u_n(r) > 0} R(\widetilde u_n(r)) \right\rangle \d r \\
            &\quad= \lim_{n\rightarrow\infty} \int_0^T \left\langle \varphi, \1_{\widetilde u_n(r) > 0} R(\widetilde u(r)) \right\rangle \d r
            = \lim_{n\rightarrow\infty} \int_0^T \left\langle \varphi , \1_{\widetilde u_n(r) > 0} \1_{\Sigma(\widetilde u(r)) > 0} R(\widetilde u(r)) \right\rangle \d r \\
            &\quad= \int_0^T \left\langle \varphi, \1_{\widetilde u(r) > 0} \1_{\Sigma(\widetilde u(r)) > 0} R(\widetilde u(r)) \right\rangle \d r 
            = \int_0^T \left\langle \varphi, \1_{\widetilde u(r) > 0} R(\widetilde u(r))\right\rangle \d r.
        \end{aligned}
    \end{equation}
    
    Therefore, we obtain $\widetilde\P$-a.s.
    \[
        \int_0^T \left\langle \varphi, \widetilde\eta(r) \right\rangle \d r = \int_0^T \left\langle \varphi, \1_{\widetilde u(r)>0} R(\widetilde u(r)) \right\rangle \d r.
    \]

    \emph{Identification of $u$.} From Lemma \ref{lemma: M1}, Lemma \ref{lemma: M2}, and Lemma \ref{lemma: M3} we know that the processes $\widetilde M_{1}$, $\widetilde M_{2}$, $\widetilde M_{3}$ defined in \eqref{eq: definition continuous martingales} are $\widetilde{\mathbb F}$-martingales. Therefore, we can invoke \cite[Proposition A.1]{hofmanova2013degenerate} to conclude that the relation \eqref{eq: continuous dynamics} is satisfied $\widetilde\P \otimes \d t$-a.s. for all $\varphi \in C^3_c([0,1],\R)$.
    
    \emph{Stickiness of the limiting system.} Note that \eqref{eq: identification of eta second term} holds for $R \equiv 1$ and $\varphi \equiv 1$, which entails
    \[
        \mathcal S(\widetilde u_n) \rightarrow \mathcal S(\widetilde u) \quad \widetilde\P\text{-a.s.}
    \]
    Thus, \eqref{eq: limiting stickiness condition} follows since
    \[
        \widetilde\P\left( \mathcal S(\widetilde u) > 0 \right) \geq \widetilde\P \left( \mathcal S(\widetilde u) \geq \varepsilon \right) \geq \limsup_{n\rightarrow\infty} \widetilde\P \left(\mathcal S(\widetilde u_n) \geq \varepsilon \right) > 0.
    \]

    \emph{Energy estimate.} The energy estimate \eqref{eq: energy estimate of solution} follows from Lemma \ref{lemma: discretization estimates for function powers}, Lemma \ref{lemma: equivalence of discrete and continuous norms}, Proposition \ref{prop: estimates 2} and Fatou's lemma.
    
\end{proof}

\appendix

\section{Estimates for powers}

\label{sec: appendix sobolev power estimates}

We reference the following statements, which allow us to control the Sobolev norms of fractional powers of non-negative functions.

\begin{lemma}
    \label{lemma: power sobolev estimate nu > 1}
    Let $\nu \geq 1$, $\theta \in \left(0,\nu\right)$, $p\in[1,\infty)$ and consider $u:[0,1]\rightarrow\R_+$. Then
    \[
        \|u^\nu\|_{W^{\theta, p}} \lesssim \|u\|_{W^{\theta,p}} \|u\|^{\nu-1}_{L^\infty}.
    \]
\end{lemma}

\begin{proof}
    Follows from \cite[Section 5.4.3, p.363]{runst2011sobolev}.
\end{proof}

\begin{lemma}
    \label{lemma: power sobolev estimate nu <= 1}
    Let $\nu \leq 1$, $\theta \in (0,1)$, $p\in[1,\infty)$ and consider $u:[0,1]\rightarrow\R_+$. Then
    \[
        \|u^\nu\|_{W^{\nu\theta,\frac{p}{\nu}}} \lesssim \|u\|^{\nu}_{W^{\theta,p}}.
    \]
\end{lemma}

\begin{proof}
    Follows from \cite[Section 5.4.4, p.365]{runst2011sobolev}.
\end{proof}

\section{Proofs of auxiliary results}
\label{sec: appendix proof of auxiliary results}

\subsection{Proof of Lemma \ref{lemma: In continuity in Htheta}}

\label{subsec: proof of lemma In continuity in Htheta}

\begin{proof}[Proof of Lemma \ref{lemma: In continuity in Htheta}]
    We first prove that there exist $m_0, m_1 > 0$ such that for all $n\in\N$,
    \[
        \begin{aligned}
            \|\mathfrak I_n (u)\|_{L^2} &\leq m_0 \|u\|_{L^2} \quad \forall u \in L^2, \\
            \|\mathfrak I_n (u)\|_{H^1_0} &\leq m_1 \|u\|_{H^1_0} \quad \forall u \in H^1_0 .
        \end{aligned}
    \]

    For the first estimate, noting that since $\sin$ is a smooth function, we have
    \[
        \langle g_k, \mathfrak e_{i,n} \rangle = \frac{2\sqrt{2} \left(-\sin(\pi h_n (i-1) k) + 2\sin(\pi h_n i k) - \sin(\pi h_n (i+1) k) \right)}{\pi^2 h_n^{3/2} k^2} \lesssim \frac{h_n^{1/2}}{k^2}.
    \]
    Therefore
    \[
        \begin{aligned}
            \langle u, \mathfrak e_{i,n} \rangle &= \sum_{k=1}^\infty \langle u, g_k \rangle \langle g_k, \mathfrak e_{i,n} \rangle \leq h_n^{1/2} \left(\sum_{k=1}^\infty \langle u, g_k \rangle^2 \right)^{1/2} \left( \sum_{k=1}^\infty \frac{1}{k^4} \right)^{1/2} \lesssim h_n^{1/2} \|u\|_{L^2}.
        \end{aligned}
    \]
    Thus we can write
    \[
        \begin{aligned}
            \|\mathfrak I_n (u)\|_{L^2}^2 = \int_0^1 \left( \sum_{i=1}^n \langle u, \mathfrak e_{i,n} \rangle \mathfrak e_{i,n}(x) \right)^2 \d x \lesssim \|u\|_{L^2}^2 \cdot h_n \int_0^1 h_n^{-1} \d x = \|u\|_{L^2}^2.
        \end{aligned}
    \]
    
    For the second estimate note that
    \[
        \begin{aligned}
            &\langle u, \mathfrak e_{j,n} - \mathfrak e_{j-1,n} \rangle = - h_n^{-3/2} \int_{(j-1)h_n}^{jh_n} u(x-h_n)(x - (j-1)h_n) \d x \\
            &\qquad\qquad + 2 h_n^{-3/2} \int_{(j-1)h_n}^{jh_n} u(x)(x - (j-1/2)h_n) \d x - h_n^{-3/2} \int_{(j-1)h_n}^{jh_n} u(x+h_n)(x - jh_n) \d x \\
            &\qquad= h_n^{-1/2} \int_{(j-1)h_n}^{jh_n} \left\{ \left( \frac{u(x) - u(x-h_n)}{h_n} \right) (x - (j-1) h_n) - \left( \frac{u(x+h_n) - u(x)}{h_n} \right) (x - j h_n) \right\} \d x \\
            &\qquad\lesssim h_n \left(\int_{(j-1)h_n}^{jh_n} \left( \frac{u(x) - u(x-h_n)}{h_n} \right)^{2} \d x \right)^{1/2} + h_n \left( \int_{(j-1)h_n}^{jh_n} \left( \frac{u(x+h_n) - u(x)}{h_n} \right)^2 \d x\right)^{1/2}, \\
        \end{aligned}
    \]
    which implies that
    \[
        \langle u, \mathfrak e_{j,n} - \mathfrak e_{j-1,n} \rangle^2 \lesssim h_n^2 \int_{(j-2)h_n}^{jh_n} \left( \frac{u(x) - u(x-h_n)}{h_n} \right)^{2} \d x.
    \]
    Therefore, we have
    \[
        \begin{aligned}
            \left\| \mathfrak I_n(u) \right\|_{H^1_0}^2 &= \int_0^1 \left( \sum_{i=1}^{n} \langle u, \mathfrak e_{i,n} \rangle \partial_x \mathfrak e_{i,n}(x) \right)^2 \d x = \sum_{j=1}^{n+1} \int_{(j-1)h_n}^{jh_n} \left( h_n^{-3/2} \langle u, \mathfrak e_{j,n} - \mathfrak e_{j-1,n} \rangle \right)^2 \d x \\
            &= h_n^{-2} \sum_{j=1}^{n+1} \langle u, \mathfrak e_{j,n} - \mathfrak e_{j-1,n} \rangle^2 \lesssim \int_0^1 \left( \frac{u(x) - u(x-h_n)}{h_n} \right)^{2} \d x \lesssim \|u\|_{H^1_0}^2.
        \end{aligned}
    \]

    Finally, due to the interpolation property \cite[Section 2.5.1]{runst2011sobolev} there exists $C>0$ such that for any $\theta\in(0,1)$ we have $\mathfrak I_n \in \mathcal L(H^\theta_0, H^\theta_0)$ and
    \[
        \left\| \mathfrak I_n \right\|_{\mathcal L\left(H^\theta_0, H^\theta_0\right)} \leq C \left\| \mathfrak I_n \right\|_{\mathcal L\left(L^2, L^2\right)}^{1-\theta} \left\| \mathfrak I_n \right\|_{\mathcal L\left(H^1_0, H^1_0\right)}^{\theta}.
    \]
\end{proof}

\subsection{Proof of Lemma \ref{lemma: discretization estimates for function powers}}

\label{subsec: proof of discretization estimates for function powers}

\begin{proof}[Proof of Lemma \ref{lemma: discretization estimates for function powers}]
    For brevity, in this proof we denote $u_i := \mathbf u_n(i)$ for $i\in\{1,\dots,n\}$. We also denote $\Delta_{i} := [ih_n, (i+1)h_n]$ and $\Delta_{i,j} := [ih_n, (i+1)h_n] \times [jh_n, (j+1)h_n]$.
    
    Note that for the Sobolev-Slobodeckii norm, due to \cite[Section 2.3.1]{runst2011sobolev}, one has for any $\delta>0$ the equivalence
    \[
        \|u\|_{H^\theta}^2 = \iint_{[0,1]^2} \frac{|u(x) - u(y)|^2}{|x-y|^{1+2\theta}} \d x \d y \sim \iint_{[0,1]^2} \1_{|x-y|<\delta} \frac{|u(x) - u(y)|^2}{|x-y|^{1+2\theta}} \d x \d y.
    \]
    Therefore, without loss of generality we show the statement for
    \begin{equation}
        \label{eq: characterization of slobodeckii norm with delta}
        \left\| u \right\|_{H^\theta}^2 := \sum_{\substack{i,j=0\\|i-j| \leq 1}}^n \iint_{\Delta_{i,j}} \frac{|u(x) - u(y)|^2}{|x-y|^{1+2\theta}} \d x \d y.
    \end{equation}
    Also note that
    \begin{equation}
        \label{eq: definition sobolev slobodeckii Delta i with delta}
        \left\| u \right\|_{H^\theta(\Delta_i)}^2 = 2 \int_0^{h_n} \int_{ih_n}^{(i+1)h_n - \delta} \frac{|u(x+\delta) - u(x)|^2}{\delta^{1+2\theta}} \d x \d \delta
    \end{equation}
    and
    \begin{equation}
        \label{eq: sobolev slobodeckii Delta i for piecewise linear}
        \left\| \mathfrak P_n^{-1}(\mathbf u_n) \right\|_{H^\theta(\Delta_i)}^2 = \iint_{\Delta_{i,i}} \frac{\left| \left(\frac{x-y}{h_n}\right) (u_{i+1} - u_i) \right|^2}{|x-y|^{1+2\theta}}dxdy = h_n^{1-2\theta} |u_{i+1} - u_i|^2.
    \end{equation}

    We first show that
    \begin{equation}
        \label{eq: discretization power estimate second inequality first claim}
        \left\| \mathfrak P_n^{-1}(\mathbf u_n)^\mu \right\|_{H^\theta(\Delta_{i})} \lesssim \left\| \mathfrak P_n^{-1}(\mathbf u_n^\mu) \right\|_{H^\theta(\Delta_{i})}.
    \end{equation}
    Without loss of generality we assume that $u_i \neq u_{i+1}$, since otherwise 
    \[
        \left\| \mathfrak P_n^{-1}(\mathbf u_n)^\mu \right\|_{H^\theta(\Delta_{i})} = \left\| \mathfrak P_n^{-1}(\mathbf u_n^\mu) \right\|_{H^\theta(\Delta_{i})} = 0.
    \]
    Note that
    \[
        \left| \left(u_i + \left(\frac{y-ih_n}{h_n}\right)(u_{i+1} - u_i) \right)^\mu - \left(u_i + \left(\frac{x-ih_n}{h_n}\right)(u_{i+1} - u_i) \right)^\mu \right| \leq \left|u_{i+1}^\mu - u_i^\mu\right|,
    \]
    so that \eqref{eq: discretization power estimate second inequality first claim} follows since
    \[
        \left\| \mathfrak P_n^{-1}(\mathbf u_n)^\mu \right\|_{H^\theta(\Delta_{i})}^2 \lesssim h_n^{1-2\theta} \left|u_{i+1}^\mu - u_i^\mu\right|^2 = \left\| \mathfrak P_n^{-1} (\mathbf u_n^\mu) \right\|_{H^\theta(\Delta_i)}^2.
    \]

    Next, we show that for any $\mathbf u_n\in \R^n$ and $\nu\geq 1$ we have
    \begin{equation}
        \label{eq: discretization power estimate second inequality second claim}
        \iint_{\substack{\Delta_{i,i-1} \\\cup \Delta_{i,i+1}}} \frac{\left| \mathfrak P_n^{-1}(\mathbf u_n)^\nu(x) - \mathfrak P_n^{-1}(\mathbf u_n)^\nu(y) \right|^2}{|x-y|^{1+2\theta}} \d x \d y \lesssim \sum_{j=i-1}^{i+1} \left\| \mathfrak P_n^{-1}(\mathbf u_n)^\nu \right\|_{H^\theta(\Delta_{j})}^2.
    \end{equation}
    Indeed, if for $(x,y)\in\Delta_{i,i+1}$ we denote $\delta_1 = (i+1)h_n - x$ and $\delta_2 = y - (i+1)h_n$, so that in particular $0 \leq \delta_1, \delta_2 \leq h_n$, we can write
    \[
        \begin{aligned}
            &\frac{\left| \mathfrak P_n^{-1}(\mathbf u_n)^\nu(x) - \mathfrak P_n^{-1}(\mathbf u_n)^\nu(y) \right|^2}{|x-y|^{1+2\theta}} = \frac{\left| (u_{i+1} + \frac{\delta_1}{h_n} (u_i - u_{i+1}))^\nu - (u_{i+1} - \frac{\delta_2}{h_n} (u_{i+2} - u_{i+1}))^\nu \right|^2}{|\delta_1 + \delta_2|^{1+2\theta}} \\
            &\qquad\lesssim \frac{\left| u_{i+1}^\nu  - (u_{i+1} + \frac{\delta_1}{h_n} (u_i - u_{i+1}))^\nu \right|^2}{|\delta_1|^{1+2\theta}} + \frac{\left| (u_{i+1} - \frac{\delta_2}{h_n} (u_{i+2} - u_{i+1}))^\nu - u_{i+1}^\nu\right|^2}{|\delta_2|^{1+2\theta}} \\
            &\qquad= \frac{\left| \mathfrak P_n^{-1}(\mathbf u_n)^\nu(x+\delta_1) - \mathfrak P_n^{-1}(\mathbf u_n)^\nu(x) \right|^2}{|\delta_1|^{1+2\theta}} + \frac{\left| \mathfrak P_n^{-1}(\mathbf u_n)^\nu(y) - \mathfrak P_n^{-1}(\mathbf u_n)^\nu(y-\delta_2) \right|^2}{|\delta_2|^{1+2\theta}}.
        \end{aligned}
    \]
    We have an analogous claim for $(x,y) \in \Delta_{i,i-1}$, so that \eqref{eq: discretization power estimate second inequality second claim} follows.

    Finally, we combine \eqref{eq: discretization power estimate second inequality first claim} and \eqref{eq: discretization power estimate second inequality second claim} to obtain
    \begin{equation}
        \label{eq: discretization power estimate first inequality final computation}
        \begin{aligned}
            \left\| \mathfrak P_n^{-1} (\mathbf u_n)^\mu \right\|_{H^\theta}^2 &= \sum_{\substack{i,j=0\\|i-j| \leq 1}}^n \iint_{\Delta_{i,j}} \frac{\left| \mathfrak P_n^{-1}(\mathbf u_n)^\mu(x) - \mathfrak P_n^{-1}(\mathbf u_n)^\mu(y) \right|^2}{|x-y|^{1+2\theta}} \d x \d y\\
            &\lesssim \sum_{i=1}^n \left\| \mathfrak P_n^{-1}(\mathbf u_n)^\mu \right\|_{H^\theta(\Delta_{i})}^2  \lesssim \sum_{i=1}^n \left\| \mathfrak P_n^{-1}\left(\mathbf u_n^\mu\right) \right\|_{H^\theta(\Delta_{i})}^2\\
            &\leq \sum_{\substack{i,j=0\\|i-j| \leq 1}}^n \iint_{\Delta_{i,j}} \frac{\left| \mathfrak P_n^{-1} \left(\mathbf u_n^\mu\right)(x) - \mathfrak P_n^{-1}\left(\mathbf u_n^\mu\right)(y) \right|^2}{|x-y|^{1+2\theta}} \d x \d y = \left\| \mathfrak P_n^{-1} \left(\mathbf u_n^\mu\right) \right\|_{H^\theta}^2.
        \end{aligned}
    \end{equation}
    
\end{proof}

\subsection{Proof of Lemma \ref{lemma: discrete sobolev power estimates}}

\label{subsec: proof of discrete sobolev power estimates}

\begin{proof}[Proof of Lemma \ref{lemma: discrete sobolev power estimates}]

    For brevity, in this proof we write $u_i := \mathbf u_n(i)$ for $i\in\{1,\dots,n\}$. We also denote by $\Delta_{i} := [ih_n, (i+1)h_n]$ and $\Delta_{i,j} := [ih_n, (i+1)h_n] \times [jh_n, (j+1)h_n]$. 
    
    In the case $\mu \leq 1$, using Lemma \ref{lemma: equivalence of discrete and continuous norms}, the elementary inequality $|a-b| \leq \left| a^\nu - b^\nu \right|^{1/\nu}$ for $a,b \geq 0$ and $\nu \geq 1$, equations \eqref{eq: characterization of slobodeckii norm with delta} and \eqref{eq: definition sobolev slobodeckii Delta i with delta}, inequality \eqref{eq: discretization power estimate second inequality second claim} with $\nu=1$, we obtain
    \[
        \begin{aligned}
            &\left\| \mathbf u_n^\mu \right\|_{\mathbf H^\theta_n}^2 \lesssim \left\| \mathfrak P_n^{-1} \left( \mathbf u_n^\mu \right) \right\|_{H^\theta_0}^2 = \sum_{\substack{i,j=0\\|i-j| \leq 1}}^n \iint_{\Delta_{i,j}} \frac{\left| \mathfrak P_n^{-1}(\mathbf u_n^\mu)(x) - \mathfrak P_n^{-1}(\mathbf u_n^\mu)(y) \right|^2}{|x-y|^{1+2\theta}} \d x \d y \lesssim \sum_{i=1}^n \left\| \mathfrak P_n^{-1}(\mathbf u_n^\mu) \right\|_{H^\theta(\Delta_{i})}^2 \\
            &\qquad\lesssim \sum_{i=1}^n \int_{0}^{h_n} \int_{ih_n}^{(i+1)h_n-\delta} \frac{\left| \frac{\delta}{h_n} \left(u_{i+1}^\mu - u_i^\mu\right) \right|^2}{\delta^{1+2\theta}} \d x \d \delta \leq h_n^{-2(1-\mu)} \sum_{i=1}^n \int_{0}^{h_n} \int_{ih_n}^{(i+1)h_n-\delta} \frac{\left| \frac{\delta}{h_n} \left(u_{i+1} - u_i\right) \right|^{2\mu}}{\delta^{1+2\theta-2(1-\mu)}} \d x \d \delta \\
            &\qquad\leq h_n^{-2(1-\mu)} \left( \sum_{i=1}^n \int_{ih_n}^{(i+1)h_n} \int_{0}^{h_n} \delta \d \delta \d x \right)^{1-\mu} \left( \sum_{i=1}^n \int_{ih_n}^{(i+1)h_n} \int_{0}^{h_n} \frac{\left| \frac{\delta}{h_n} \left(u_{i+1} - u_i\right) \right|^{2}}{\delta^{1+2\theta/\mu}} \d \delta \d x \right)^\mu \\
            &\qquad\lesssim \left(\sum_{i=1}^n \left\| \mathfrak P_n^{-1} (\mathbf u_n) \right\|_{H^{\theta/\mu}_0(\Delta_i)}^{2} \right)^\mu \leq \left( \sum_{\substack{i,j=0\\|i-j| \leq 1}}^n \iint_{\Delta_{i,j}} \frac{\left| \mathfrak P_n^{-1}(\mathbf u_n)(x) - \mathfrak P_n^{-1}(\mathbf u_n)(y) \right|^2}{|x-y|^{1+2\theta/\mu}} \d x \d y \right)^\mu \\
            &\qquad= \left\| \mathfrak P_n^{-1} \left(\mathbf u_n\right) \right\|_{H^{\theta/\mu}_0}^{2\mu} \lesssim \left\| \mathbf u_n \right\|_{\mathbf H^{\theta/\mu}_n}^{2\mu}.
        \end{aligned}
    \]
    The case $\mu \geq 1$ follows similarly, using the inequality
    \[
        \left| u_{i+1}^\mu - u_i^\mu \right| \leq \left| u_{i+1} - u_i \right| \left(u_{i}^{\mu-1} + u_{i+1}^{\mu-1} \right).
    \]
    
\end{proof}

\subsection{Proof of Lemma \ref{lemma: estimates for discretized coefficients}}

\label{subsec: proof of estimates for discretized coefficients}

\begin{proof}[Proof of Lemma \ref{lemma: estimates for discretized coefficients}]

    For brevity, in this proof we write $u_i := \mathbf u_n(i)$ for $i\in\{1,\dots,n\}$. We also denote by $\Delta_{i} := [ih_n, (i+1)h_n]$ and $\Delta_{i,j} := [ih_n, (i+1)h_n] \times [jh_n, (j+1)h_n]$.
    
    \emph{Estimate for $\mathbf B_n$.} We first want to prove that
    \begin{equation}
        \label{eq: coefficient discretization proof control of diagonal element}
        \begin{aligned}
            &\left\| \mathfrak P_n^{-1} \circ \left(\mathbf u_n^\mu \odot \mathbf B_n(\mathbf u_n)(g_k) \right) \right\|_{H^\theta(\Delta_i)}^2 \lesssim \left\| \mathfrak P_n^{-1} \left( \mathbf u_n^\mu \right) \right\|_{H^\theta(\Delta_i)}^2 + \left\| \mathfrak P_n^{-1} \left( \mathbf u_n^{\mu+1} \right) \right\|_{H^\theta(\Delta_i)}^2 \\
            &\qquad + k^2 \left\| \mathfrak P_n^{-1} \left( \mathbf u_n^{\mu} \right) \right\|_{H^0(\Delta_{i-1} \cup \Delta_i)}^2 + k^2 \left\| \mathfrak P_n^{-1} \left( \mathbf u_n^{\mu+1} \right) \right\|_{H^0(\Delta_{i-1} \cup \Delta_i)}^2.
        \end{aligned}
    \end{equation}
    Note that since $b\in C^1_b$, it holds for $(x,x+\delta)\in\Delta_{i,i}$,
    \[
        \begin{aligned}
            &\mathfrak P_n^{-1} \circ \left(\mathbf u_n^\mu \odot \mathbf B_n(\mathbf u_n)(g_k) \right) (x+\delta) - \mathfrak P_n^{-1} \circ \left(\mathbf u_n^\mu \odot \mathbf B_n(\mathbf u_n)(g_k) \right)(x) \\
            &\qquad= \frac{\delta}{h_n} \left\{ h_{n}^{-1/2} u_{i+1}^\mu \left\langle B\left( P_n^{-1}(\mathbf u_n) \right) \cdot g_k, e_{i+1,n} \right\rangle - h_{n}^{-1/2} u_{i}^\mu \left\langle B\left( P_n^{-1}(\mathbf u_n) \right) \cdot g_k, e_{i,n} \right\rangle \right\} \\
            &\qquad= \frac{\delta}{h_n} \left\{ h_n^{-1} u_{i+1}^\mu \int_{ih_n}^{(i+1)h_n} b\left( y, u_{i+1} \right) g_k(y) \d y - h_n^{-1} u_{i}^\mu \int_{(i-1)h_n}^{ih_n} b \left( y, u_{i} \right) g_k(y) \d y \right\} \\
            &\qquad= \frac{\delta}{h_n} \left( K_1 + K_2 + K_3 \right)
        \end{aligned}
    \]
    with
    \[
        \begin{aligned}
            K_1 &:= h_n^{-1} \left( u_{i+1}^\mu - u_{i}^\mu \right) \int_{(i-1)h_n}^{ih_n} b(y, u_{i}) g_k(y) \d y, \\
            K_2 &:= h_n^{-1} u_{i+1}^\mu \int_{(i-1)h_n}^{ih_n} \left( b(y+h_n, u_{i}) g_k(y+h_n) \d y - b(y, u_{i}) g_k(y) \right) \d y, \\
            K_3 &:= h_n^{-1} u_{i+1}^\mu \int_{ih_n}^{(i+1)h_n} g_k(y) \int_{u_{i}}^{u_{i+1}} \partial_u b(y, v) \d v \d y.
        \end{aligned}
    \]
    Therefore, it follows from \eqref{eq: definition sobolev slobodeckii Delta i with delta} that
    \begin{equation}
        \label{eq: coefficient discretization proof equation for H theta Delta i}
        \begin{aligned}
            &\left\| \mathfrak P_n^{-1} \circ \left(\mathbf u_n^\mu \odot \mathbf B_n(\mathbf u_n)(g_k) \right) \right\|_{H^\theta(\Delta_i)}^2 = 2\int_{0}^{h_n} \int_{ih_n}^{(i+1)h_n-\delta}  \frac{\left| \left(\frac{\delta}{h_n} \right) (K_1 + K_2 + K_3)\right|^2}{\delta^{1+2\theta}} \d x \d \delta.
        \end{aligned}
    \end{equation}
    Due to the growth condition on $b$ we have
    \[
        K_1 \leq h_n^{-1} (1 + u_i) \left| u_{i+1}^\mu - u_{i}^\mu \right| \left|\int_{(i-1)h_n}^{ih_n} g_k(y) \d y \right| \lesssim (1 + u_i) \left| u_{i+1}^\mu - u_{i}^\mu \right|.
    \]
    Since moreover $b\in C^1_b$, we deduce from the mean value theorem that
    \[
        \begin{aligned}
            K_2 &\leq u_{i+1}^\mu \int_{(i-1)h_n}^{ih_n} \frac{\left| b(y+h_n, u_{i}) g_k(y+h_n) \d y - b(y, u_{i}) g_k(y) \right|}{h_n} \d y \\
            &\lesssim u_{i+1}^\mu h_n \left(\sup_{x\in[0,1]} \left| \partial_x b(x, u_i) \right| + k \sup_{x\in[0,1]} |b(x, u_i)| \right) \lesssim h_n k  \left(u_{i}^\mu + u_{i}^{\mu+1} + u_{i+1}^\mu + u_{i+1}^{\mu+1} \right). 
        \end{aligned}
    \]
    Furthermore, again since $b\in C^1_b$ we have
    \[
        \begin{aligned}
            \left| \int_{ih_n}^{(i+1)h_n} g_k(y) \int_{u_{i}}^{u_{i+1}} \partial_u b(ih_n, v) \d v \d y \right| &\lesssim |u_{i+1} - u_{i}| \left| \int_{ih_n}^{(i+1)h_n} g_k(y) \d y \right| \lesssim h_n |u_{i+1} - u_{i}|,\\
        \end{aligned}
    \]
    which implies that
    \[
        K_3 \lesssim u_{i+1}^\mu |u_{i+1} - u_i|.
    \]
    Combining these estimates on the $K_i$, we obtain
    \[
        K_1 + K_2 + K_3 \lesssim h_n k \left(u_{i}^\mu + u_{i}^{\mu+1} + u_{i+1}^\mu + u_{i+1}^{\mu+1} \right) + (1 + u_i) |u_{i+1}^\mu - u_i^\mu| + u_{i+1}^\mu |u_{i+1} - u_i|,
    \]
    so that in particular
    \[
        |K_1 + K_2 + K_3|^2 \lesssim h_n^2 k^2 \left( \left| u_{i}^\mu \right|^2 + \left| u_{i}^{\mu+1} \right|^2 + \left| u_{i+1}^\mu \right|^2 + \left| u_{i+1}^{\mu+1} \right|^2 \right) + \left| u_{i+1}^{\mu} - u_{i}^{\mu} \right|^2 + \left| u_{i+1}^{\mu+1} - u_{i}^{\mu+1} \right|^2.
    \]
    Plugging this into \eqref{eq: coefficient discretization proof equation for H theta Delta i} entails \eqref{eq: coefficient discretization proof control of diagonal element}. Thus we obtain
    \[
        \begin{aligned}
            & \left\|\mathbf u_n^{\mu} \odot \mathbf B_n(\mathbf u_n) \Pi_n \right\|_{L_2(U_0, \mathbf H^0_n)}^2 = \sum_{k=1}^n \mu_k^2 \left\| \mathbf u_n^\mu \odot \mathbf B_n(\mathbf u_n)(g_k) \right\|_{\mathbf H^\theta_n}^2 \\
            &\qquad\stackrel{\text{Lemma \ref{lemma: equivalence of discrete and continuous norms}}}{\lesssim} \sum_{k=1}^n \mu_k^2 \left\| \mathfrak P_n^{-1} \circ \left(\mathbf u_n^\mu \odot \mathbf B_n(\mathbf u_n)(g_k) \right) \right\|_{H^\theta}^2 \stackrel{\eqref{eq: discretization power estimate second inequality second claim}}{\lesssim} \sum_{k=1}^n \sum_{i=1}^n \mu_k^2 \left\| \mathfrak P_n^{-1} \circ \left(\mathbf u_n^\mu \odot \mathbf B_n(\mathbf u_n)(g_k) \right) \right\|_{H^\theta(\Delta_i)}^2 \\
            &\qquad\stackrel{\eqref{eq: coefficient discretization proof control of diagonal element}}{\lesssim} \sum_{k=1}^n \sum_{i=1}^n \mu_k^2 \left\{ k^2 \left\| \mathfrak P_n^{-1} \left( \mathbf u_n^{\mu} \right) \right\|_{H^0(\Delta_i)}^2 + k^2 \left\| \mathfrak P_n^{-1} \left( \mathbf u_n^{\mu+1} \right) \right\|_{H^0(\Delta_i)}^2 + \left\| \mathfrak P_n^{-1} \left( \mathbf u_n^\mu \right) \right\|_{H^\theta(\Delta_i)}^2 + \left\| \mathfrak P_n^{-1} \left( \mathbf u_n^{\mu+1} \right) \right\|_{H^\theta(\Delta_i)}^2 \right\} \\
            &\qquad\leq \left\| \mathfrak P_n^{-1} \left( \mathbf u_n^{\mu} \right) \right\|_{H^0}^2 \sum_{k=1}^n k^2 \mu_k^2 + \left\| \mathfrak P_n^{-1} \left( \mathbf u_n^{\mu+1} \right) \right\|_{H^0}^2 \sum_{k=1}^n k^2 \mu_k^2 + \left\| \mathfrak P_n^{-1} \left( \mathbf u_n^\mu \right) \right\|_{H^\theta}^2 \sum_{k=1}^n \mu_k^2 + \left\| \mathfrak P_n^{-1} \left( \mathbf u_n^{\mu+1} \right) \right\|_{H^\theta}^2 \sum_{k=1}^n \mu_k^2 \\
            &\qquad \stackrel{\text{Lemma \ref{lemma: poincare type result} and \eqref{eq: noise coloring}}}{\lesssim} \left\| \mathfrak P_n^{-1} \left( \mathbf u_n^\mu \right) \right\|_{H^\theta}^2 + \left\| \mathfrak P_n^{-1} \left( \mathbf u_n^{\mu+1} \right) \right\|_{H^\theta}^2 \stackrel{\text{Lemma \ref{lemma: equivalence of discrete and continuous norms}}}{\lesssim} \left\|\mathbf u_n^\mu\right\|_{\mathbf H^{\theta}_n}^2 + \left\|\mathbf u_n^{\mu+1}\right\|_{\mathbf H^{\theta}_n}^2.
        \end{aligned}
    \]

    \emph{Estimate for $\boldsymbol\Sigma_n$.} It suffices to prove that
    \begin{equation}
        \label{eq: coefficient discretization proof sigma claim}
        \begin{aligned}
            &\left\| \mathfrak P_n^{-1} \circ \left(\mathbf u_n^\mu \odot \boldsymbol 
            \Sigma_n(\mathbf u_n) \right) \right\|_{H^\theta(\Delta_i)}^2 \lesssim \left\| \mathfrak P_n^{-1} \left( \mathbf u_n^\mu \right) \right\|_{H^\theta(\Delta_i)}^2 + \left\| \mathfrak P_n^{-1} \left( \mathbf u_n^{\mu+2} \right) \right\|_{H^\theta(\Delta_i)}^2 \\
            &\qquad + k^2 \left\| \mathfrak P_n^{-1} \left( \mathbf u_n^{\mu} \right) \right\|_{H^0(\Delta_{i-1} \cup \Delta_i)}^2 + k^2 \left\| \mathfrak P_n^{-1} \left( \mathbf u_n^{\mu+2} \right) \right\|_{H^0(\Delta_{i-1} \cup \Delta_i)}^2,
        \end{aligned}
    \end{equation}
    so that the rest follows as above for $\mathbf B_n$. For $(x,x+\delta)\in\Delta_{i,i}$ we have that
    \[
        \begin{aligned}
            &\mathfrak P_n^{-1} \circ \left(\mathbf u_n^\mu \odot \boldsymbol\Sigma_n(\mathbf u_n) \right) (x+\delta) - \mathfrak P_n^{-1} \circ \left(\mathbf u_n^\mu \odot \boldsymbol\Sigma_n(\mathbf u_n) \right)(x) \\
            &\qquad= \frac{\delta}{h_n} \left\{ u_{i+1}^\mu \sum_{k=1}^\infty \mu_k^2 \left( h_n^{-1} \int_{ih_n}^{(i+1)h_n} b(y, u_{i+1}) g_k(y) \d y \right)^2 - u_{i}^\mu \sum_{k=1}^\infty \mu_k^2 \left( h_n^{-1} \int_{(i-1)h_n}^{ih_n} b(y, u_i) g_k(y) \d y \right)^2 \right\}. \\
        \end{aligned}
    \]
    We therefore have by \eqref{eq: definition sobolev slobodeckii Delta i with delta}
    \begin{equation}
        \label{eq: coefficient discretization proof sigma Delta i equation}
        \left\| \mathfrak P_n^{-1} \circ \left(\mathbf u_n^\mu \odot \boldsymbol\Sigma_n(\mathbf u_n) \right) \right\|_{H^\theta(\Delta_i)}^2 \leq 2\int_{0}^{h_n} \int_{ih_n}^{(i+1)h_n-\delta}  \frac{\left| \left(\frac{\delta}{h_n} \right) (K_1 + K_2)\right|^2}{\delta^{1+2\theta}} \d x \d \delta,
    \end{equation}
    where
    \[
        \begin{aligned}
            K_1 &:= \left( u_{i+1}^\mu - u_{i}^\mu \right) \sum_{k=1}^\infty \mu_k^2 \left( h_n^{-1} \int_{(i-1)h_n}^{ih_n} b(y, u_{i+1}) g_k(y) \d y \right)^2, \\
            K_2 &:= u_{i}^\mu \sum_{k=1}^\infty \mu_k^2 \left[ \left( h_n^{-1} \int_{ih_n}^{(i+1)h_n} b(y, u_{i+1}) g_k(y) \d y \right)^2 - \left( h_n^{-1} \int_{(i-1)h_n}^{ih_n} b(y, u_i) g_k(y) \d y \right)^2 \right]. \\
        \end{aligned}
    \]
    Note that due to the growth condition on $b$ and \eqref{eq: noise coloring} we have
    \[
        K_1 \leq (1 + u_{i+1})^2 \left| u_{i+1}^\mu - u_i^\mu \right| \leq \left| u_{i+1}^\mu - u_i^\mu \right| + \left| u_{i+1}^{\mu+2} - u_i^{\mu+2} \right| .
    \]
    Moreover, since $b\in C^1_b$ we have
    \begin{equation}
        \label{eq: coefficient discretization proof sigma I2 first}
        \begin{aligned}
            & \left| \int_{ih_n}^{(i+1)h_n} b(y, u_{i+1}) g_k(y) \d y - \int_{(i-1)h_n}^{ih_n} b(y, u_i) g_k(y) \d y \right| \\
            &\qquad= \left| \int_{(i-1)h_n}^{ih_n} \left( b(y + h_n, u_i) g_k(y + h_n) - b(y, u_i) g_k(y) \right) \d y \right| + \left| \int_{ih_n}^{(i+1)h_n} g_k(y) \int_{u_i}^{u_{i+1}} \partial_u b(y, z) \d z \d y \right| \\
            &\qquad\lesssim \left| h_n \int_{(i-1)h_n}^{ih_n} \left( \frac{b(y + h_n, u_i) g_k(y + h_n) - b(y, u_i) g_k(y)}{h_n} \right) \d y \right| + h_n |u_{i+1} - u_i| \\
            &\qquad \leq h_n^2 \left( \sup_{x\in[0,1]} |\partial_x b(x, u_i)| + k \sup_{x\in[0,1]} b(x, u_i) \right) + h_n |u_{i+1} - u_i| \\
            &\qquad\lesssim h_n^2 k \left(1 + u_i\right) + h_n |u_{i+1} - u_i|.
        \end{aligned}
    \end{equation}
    Furthermore, we have due to the growth condition on $b$ that
    \begin{equation}
        \label{eq: coefficient discretization proof sigma I2 second}
        \left| \int_{(i-1)h_n}^{(i+1)h_n} b(y, u_{i+1}) g_k(y) \d y + \int_{(i-1)h_n}^{ih_n} b(y, u_i) g_k(y) \d y \right| \lesssim h_n (1 + u_{i} + u_{i+1}).
    \end{equation}
    Combining \eqref{eq: coefficient discretization proof sigma I2 first} and \eqref{eq: coefficient discretization proof sigma I2 second}, we obtain from \eqref{eq: noise coloring} that
    \[
        \begin{aligned}
            K_2 &\leq \sum_{k=1}^\infty \mu_k^2 u_i^\mu h_n^{-2} \left| \left( \int_{ih_n}^{(i+1)h_n} b(y, u_{i+1}) g_k(y) \d y \right)^2 - \left( \int_{(i-1)h_n}^{ih_n} b(y, u_i) g_k(y) \d y \right)^2 \right| \\
            &\leq \sum_{k=1}^\infty \mu_k^2 u_i^\mu \left(1 + u_{i} + u_{i+1} \right) \left( h_n k \left(1 + u_i\right) + |u_{i+1} - u_i| \right) \\
            &\lesssim \sum_{k=1}^\infty \mu_k^2 \left\{ h_n k \left(u_i^\mu + u_{i+1}^{\mu} + u_{i}^{\mu+2} + u_{i+1}^{\mu+2} \right) + \left|u_{i+1}^\mu - u_i^\mu\right| + \left|u_{i+1}^{\mu+2} - u_i^{\mu+2}\right| \right\} \\
            &\lesssim h_n \left(u_i^\mu + u_{i+1}^{\mu} + u_{i}^{\mu+2} + u_{i+1}^{\mu+2} \right) + \left|u_{i+1}^\mu - u_i^\mu\right| + \left|u_{i+1}^{\mu+2} - u_i^{\mu+2}\right|.
        \end{aligned}
    \]
    Now the estimates on the $K_i$ give
    \[
        |K_1 + K_2|^2 \leq h_n^2 \left(\left|u_i^\mu\right|^2 + \left|u_{i+1}^{\mu}\right|^2 + \left|u_{i}^{\mu+2}\right|^2 + \left|u_{i+1}^{\mu+2}\right|^2 \right) + \left|u_{i+1}^\mu - u_i^\mu\right|^2 + \left|u_{i+1}^{\mu+2} - u_i^{\mu+2}\right|^2.
    \]
    Plugging this into \eqref{eq: coefficient discretization proof sigma Delta i equation} yields \eqref{eq: coefficient discretization proof sigma claim}.
    
    \emph{Estimate for $\mathbf r_n$.} The estimate for $\mathbf r_n$ follows analogously noting that $r \in C^1_b$.
\end{proof}

\subsection{Proof of Lemma \ref{lemma: u Lu H-1}}

\label{subsec: proof of u Lu H-1}

\begin{proof}[Proof of Lemma \ref{lemma: u Lu H-1}]
    We write $\mathbf L_n = \mathbf M_n \boldsymbol \Lambda_n \mathbf M_n^T$, where $\mathbf M_n = (\mathbf m_{1,n}, \dots, \mathbf m_{n,n})$ is the orthogonal matrix consisting of eigenvectors of $\mathbf L_n$ and $\boldsymbol\Lambda_n = \operatorname{diag} ( \lambda_{1,n}, \dots, \lambda_{n,n} )$ is the diagonal matrix of the corresponding eigenvalues given by
    \[
        \lambda_{k,n} = \frac{4}{h_n^2} \sin^2 \left( \frac{k\pi h_n}{2} \right), \qquad \mathbf m_{k,n}(i) = \sqrt{2 h_n} \sin\left( k\pi i h_n \right).
    \]
    We first aim to show that
    \begin{equation}
        \label{eq: u Lu H-1 first claim}
        \left\| \mathbf m_{k,n} \odot \mathbf m_{\ell,n} \right\|_{\mathbf H^{-1}_n}^2 \lesssim \frac{1}{n} \lambda_{k,n}^{-1} \lambda_{\ell,n} a_{k,\ell}^2\quad\text{with}\quad a_{k,\ell}^2 := \frac{k^2 \left(k^2 + \ell^2 \right)}{\ell^2 \left(1 + (k+l)^2 + (k^2 - \ell^2)^2 \right)}.
    \end{equation}
    Without loss of generality we consider the case $k+\ell\in 2\Z$, the other case follows analogously. Note first that we have
    \begin{equation}
        \label{eq: u Lu H-1 expression for mk ml H-1 2}
        \begin{aligned}
            &\left\| \mathbf m_{k,n} \odot \mathbf m_{\ell,n} \right\|_{\mathbf H^{-1}_n}^2 = \sum_{m=1}^n \lambda_{m,n}^{-1} \left\langle \mathbf m_{k,n} \odot \mathbf m_{\ell,n}, \mathbf m_{m,n} \right\rangle^2 \\
            &\qquad= \frac{h_n^5}{8} \sum_{m=1}^n \left( 1 + \cot^2\left(\frac{m\pi}{2(n+1)}\right) \right) \left( 4 \sum_{i=1}^n \sin(k \pi i h_n) \sin(\ell \pi i h_n) \sin(m \pi i h_n) \right)^2.
        \end{aligned}
    \end{equation}
    By virtue of trigonometric identities we can write
    \[
        \begin{aligned}
            &4 \sin(k \pi i h_n) \sin(\ell \pi i h_n) \sin(m \pi i h_n) \\
            &\qquad= \sin((k+\ell-m) \pi i h_n) + \sin((k-\ell+m) \pi i h_n) - \sin((k+\ell+m) \pi i h_n) - \sin((k-\ell-m) \pi i h_n).
        \end{aligned}
    \]
    Moreover, using $\sum_{k=1}^n \sin (a k) = \frac{\sin \left(\frac{a(n+1)}{2}\right) \sin \left(\frac{a n}{2}\right)}{\sin \left(\frac{a}{2}\right)}$ and trigonometric identities, we have for $k \pm \ell \pm m \notin \{0,2(n+1)\}$
    \[
        \sum_{i=1}^n \sin((k \pm \ell \pm m) \pi i h_n) = \frac{ \sin\left(\frac{(k \pm \ell \pm m)\pi}{2} \right) \sin\left(\frac{(k \pm \ell \pm m)\pi}{2} \left(1 - \frac{1}{n+1}\right) \right)}{\sin\left(\frac{(k \pm \ell \pm m)\pi}{2} \left(\frac{1}{n+1} \right) \right)} = \1_{m \in 2\Z+1} \cot\left(\frac{(k \pm \ell \pm m)\pi}{2(n+1)} \right).
    \]
    Therefore, using that $k+\ell\in 2\Z$ and Laurent series expansion of $\cot(x)$ at $x=0$ and $x=\pi$ we obtain
    \[
        \begin{aligned}
            &\left( 4\sum_{i=1}^n \sin(k \pi i h_n) \sin(\ell \pi i h_n) \sin(m \pi i h_n) \right)^2 \\
            &\qquad= \left( \sum_{i=1}^n \left\{ \sin((k+\ell-m) \pi i h_n) + \sin((k-\ell+m) \pi i h_n) - \sin((k+\ell+m) \pi i h_n) - \sin((k-\ell-m) \pi i h_n) \right\} \right)^2 \\
            &\qquad= \1_{m \in 2\Z + 1} \left\{ \cot \left(\frac{(k+\ell-m)\pi}{2(n+1)} \right) + \cot \left(\frac{(k-\ell+m)\pi}{2(n+1)} \right) - \cot \left(\frac{(k+\ell+m)\pi}{2(n+1)} \right) - \cot \left(\frac{(k-\ell-m)\pi}{2(n+1)} \right) \right\}^2 \\
            &\qquad\lesssim \1_{m \in 2\Z + 1} (n+1)^2 \left\{ \frac{1}{((k+\ell)-m)^2} + \frac{1}{((k-\ell)+m)^2} + \frac{1}{((k-\ell)-m)^2} + \frac{1}{((k+\ell)+m - 2(n+1))^2} \right\}.
        \end{aligned}
    \]
    Plugging this into \eqref{eq: u Lu H-1 expression for mk ml H-1 2}, we obtain
    \begin{equation}
        \label{eq: u Lu H-1 sums}
        \begin{aligned}
            &\left\| \mathbf m_{k,n} \odot \mathbf m_{\ell,n} \right\|_{\mathbf H^{-1}_n}^2 \lesssim h_n^5 \sum_{m=1}^n \left( 1 + \cot^2\left(\frac{m\pi}{2(n+1)}\right) \right) \left( 4 \sum_{i=1}^n \sin(k \pi i h_n) \sin(\ell \pi i h_n) \sin(m \pi i h_n) \right)^2 \\
            &\qquad\lesssim h_n \sum_{\substack{m=1\\m\in 2\Z+1}}^n \frac{1}{m^2} \left\{ \frac{1}{((k+\ell)-m)^2} + \frac{1}{((k-\ell)+m)^2} + \frac{1}{((k-\ell)-m)^2} + \frac{1}{((k+\ell)+m - 2(n+1))^2} \right\}. \\
        \end{aligned}
    \end{equation}
    In order to estimate each of the sums above, we first note that for $p\in\{k+\ell, k-\ell\}$ such that $p\neq 0$, we have
    \[
        \int \frac{\d x}{x^2(p - x)^2} = \frac{1}{p^3} \left( \frac{p}{p-x} - \frac{p}{x} - 2\log\left| p - x \right| + 2\log|x| \right),
    \]
    so that we can bound
    \[
        \sum_{\substack{m=1\\m\in 2\Z+1}}^n \frac{1}{m^2(p-m)^2} \lesssim \int_1^{p-1} \frac{\d x}{x^2(p - x)^2} + \int_{p+1}^{n+1} \frac{\d x}{x^2(p - x)^2} \lesssim \frac{1}{1+p^2}.
    \]
    Moreover, in the case $k = \ell$ it holds
    \[
        \sum_{\substack{m=1\\m\in 2\Z+1}}^n \frac{1}{m^2\left((k-\ell) \pm m\right)^2} = \sum_{m=1}^{\lceil n/2 \rceil} \frac{1}{(2m-1)^4} \lesssim \sum_{m=1}^{\infty} \frac{1}{m^4} < \infty.
    \]
    Altogether, the sums in \eqref{eq: u Lu H-1 sums} can be estimated, so that
    \[
        \begin{aligned}
            \sum_{\substack{m=1\\m\in 2\Z+1}}^n \frac{1}{m^2((k+\ell)+m-2(n+1))^2} + \sum_{\substack{m=1\\m\in 2\Z+1}}^n \frac{1}{m^2((k+\ell)-m)^2} &\lesssim \frac{1}{1 + (k+\ell)^2}, \\
            \sum_{\substack{m=1\\m\in 2\Z+1}}^n \frac{1}{m^2((k-\ell)+m)^2} + \sum_{\substack{m=1\\m\in 2\Z+1}}^n \frac{1}{m^2((k-\ell)-m)^2} &\lesssim \frac{1}{1 + (k-\ell)^2}.
        \end{aligned}
    \]
    Ultimately, \eqref{eq: u Lu H-1 first claim} follows since
    \[
        \left\| \mathbf m_{k,n} \odot \mathbf m_{\ell,n} \right\|_{\mathbf H^{-1}_n}^2 \lesssim \frac{1}{n} \left( \frac{1}{1+(k+\ell)^2} + \frac{1}{1+(k-\ell)^2} \right) \lesssim \frac{k^2 + \ell^2}{n (1 + (k+l)^2 + (k^2 - \ell^2)^2)} \lesssim \frac{1}{n} \lambda_{k,n}^{-1} \lambda_{\ell,n}^{1} a_{k,\ell}^2.
    \]
    
    Next, we show that
    \begin{equation}
        \label{eq: u Lu H-1 second claim}
        \sup_{n\in\N} \frac{1}{n} \left(\sum_{k=1}^n \sum_{\ell=1}^n a_{k,\ell}^2 \right) < \infty.
    \end{equation}
    Indeed, using that $a_{k,\ell} \leq a_{\ell,k}$ for $k \leq \ell$ we have
    \[
        \begin{aligned}
            &\sum_{k=1}^n \sum_{\ell=1}^n a_{k,\ell}^2 \lesssim \sum_{k=1}^n \sum_{\ell=1}^{k} a_{k,\ell}^2 = \sum_{k=1}^n \sum_{\ell=1}^{k} \frac{k^2 \left(k^2+\ell^2\right)}{\ell^2 \left(1 + (k+\ell)^2 + \left(k^2 - \ell^2\right)^2 \right)} \lesssim \sum_{k=1}^n k^4 \sum_{\ell=1}^k \frac{1}{\ell^2 \left(1 + \left(k^2 - \ell^2\right)^2\right)},
        \end{aligned}
    \]
    Moreover, noting that 
    \[
        \int \frac{\d x}{x^2(k^2 - x^2)^2} = \frac{1}{2k^5} \left(3 \operatorname{arctanh}\left(\frac{x}{k}\right) + \frac{kx}{k^2 - x^2} - \frac{2k}{x} \right)
    \]
    and
    \[
        \lim_{k\rightarrow\infty} \frac 1k \operatorname{arctanh}\left(\frac{k-1}{k}\right) = 0,
    \]
    we have a bound
    \[
        \sum_{\ell=1}^k \frac{1}{\ell^2 \left(1 + \left(k^2 - \ell^2\right)^2\right)} \leq \int_1^{k-1} \frac{\d x}{x^2(k^2 - x^2)^2} = O\left(\frac{1}{k^4}\right),
    \]
    so that the claim \eqref{eq: u Lu H-1 second claim} follows. Consequently, the statement \eqref{eq: discrete pointwise product h-1} follows by virtue of
    \[
        \begin{aligned}
            &\left\| \mathbf v_n \odot \mathbf w_n \right\|_{\mathbf H^{-1}_n} = \left\| \sum_{k,\ell=1}^n \langle \mathbf v_n, \mathbf m_{k,n} \rangle \langle \mathbf w_n, \mathbf m_{\ell,n} \rangle \mathbf m_{k,n} \odot \mathbf m_{\ell,n} \right\|_{\mathbf H^{-1}_n} \\
            &\qquad\leq \sum_{k,\ell=1}^n |\langle \mathbf v_n, \mathbf m_{k,n} \rangle| |\langle \mathbf w_n, \mathbf m_{\ell,n} \rangle | \left\| \mathbf m_{k,n} \odot \mathbf m_{\ell,n} \right\|_{\mathbf H^{-1}_n} 
            \lesssim \frac{1}{\sqrt n} \sum_{k,\ell=1}^n |\langle \mathbf v_n, \mathbf m_{k,n} \rangle | | \langle \mathbf w_n, \mathbf m_{\ell,n} \rangle | \lambda_{k,n}^{-1/2} \lambda_{\ell,n}^{1/2} a_{k,\ell} \\
            &\qquad= \frac{1}{\sqrt{n}} \sum_{\ell=1}^n |\langle \mathbf w_n, \mathbf m_{\ell,n} \rangle | \lambda_{\ell,n}^{1/2} \sum_{k=1}^n \lambda_{k,n}^{-1/2} |\langle \mathbf v_n, \mathbf m_{k,n} \rangle | a_{k,\ell}
            \leq \frac{1}{\sqrt n} \|\mathbf w_n\|_{\mathbf H^{1}_n} \left(\sum_{\ell=1}^n \left( \sum_{k=1}^n |\langle \mathbf v_n, \mathbf m_{k,n} \rangle | \lambda_{k,n}^{-1/2} a_{k,\ell} \right)^2 \right)^{1/2} \\
            &\qquad\leq \|\mathbf v_n\|_{\mathbf H^{-1}_n} \|\mathbf w_n\|_{\mathbf H^{1}_n} \frac{1}{\sqrt n} \left(\sum_{k=1}^n \sum_{\ell=1}^n a_{k,\ell}^2 \right)^{1/2} \lesssim \|\mathbf v_n\|_{\mathbf H^{-1}_n} \|\mathbf w_n\|_{\mathbf H^{1}_n}.
        \end{aligned}
    \]
\end{proof}

\newpage

\bibliographystyle{abbrv}

\bibliography{bib}

\end{document}